\documentclass[twoside,a4paper,10pt, final]{amsart}

\newcommand{\R}{\mathds R}

\newcommand{\N}{\mathds N}
\newcommand{\dist}{\mathrm{dist}}
\newcommand{\cat}{\mathrm{cat}}
\newcommand{\Dim}{\mathrm{dim}}
\newcommand{\cl}{\mathrm{cuplenght}}

\newcommand{\second}{\mathrm I\!\mathrm I}

\newcommand{\Ical}{\mathcal I}
\newcommand{\Fcal}{\mathcal F}
\newcommand{\Vcal}{\mathcal V}

\newcommand{\Rcal}{\mathcal R}
\newcommand{\Ccal}{\mathcal C}
\newcommand{\Scal}{\mathcal S}
\newcommand{\Dcal}{\mathcal D}
\newcommand{\Hcal}{\mathcal H}
\newcommand{\Lcal}{\mathcal L}

\newcommand{\Ucal}{\mathcal U}
\newcommand{\Mcal}{\mathfrak M}
\newcommand{\Ncal}{\mathfrak N}
\newcommand{\Dgot}{\mathfrak D}

\newcommand{\Ddt}{\tfrac{\mathrm D}{\mathrm dt}}
\newcommand{\Dds}{\tfrac{\mathrm D}{\mathrm ds}}

\usepackage{amsmath}
\usepackage{epsfig}
\usepackage{times}
\usepackage[hypertex]{hyperref}%
\usepackage{dsfont}

\numberwithin{equation}{section}

\title[Existence of orthogonal geodesic chords]%
{Existence of orthogonal geodesic chords on Riemannian manifolds with concave boundary\\
and homeomorphic to the $N$-dimensional disk}

\author[R.\ Giamb\`o ,\ F.\ Giannoni]{Roberto Giamb\`o, Fabio Giannoni}
\address{Dipartimento di Matematica e Informatica \hfill\break\indent
Universit\`a di Camerino\hfill\break\indent Italy}
\email{roberto.giambo@unicam.it, fabio.giannoni@unicam.it}
\author[P. Piccione]{Paolo Piccione}
\address{Departamento de Matem\'atica\hfill\break\indent Instituto de
Matem\'atica e Estat\'\i stica
\hfill\break\indent Universidade de S\~ao Paulo
\hfill\break\indent Brazil}
\email{piccione.p@gmail.com}


\date{March 3rd, 2010 (revised version)}

\begin{document}


\theoremstyle{plain}\newtheorem{teo}{Theorem}[section]
\theoremstyle{plain}\newtheorem{prop}[teo]{Proposition}
\theoremstyle{plain}\newtheorem{lem}[teo]{Lemma}
\theoremstyle{plain}\newtheorem{cor}[teo]{Corollary}
\theoremstyle{definition}\newtheorem{defin}[teo]{Definition}
\theoremstyle{remark}\newtheorem{rem}[teo]{Remark}
\theoremstyle{definition}\newtheorem{example}[teo]{Example}

\theoremstyle{plain}\newtheorem*{teon}{Theorem}


\begin{abstract}
In this paper we give a proof of the existence of an orthogonal geodesic chord
on a Riemannian manifold homeomorphic to a closed disk and with concave boundary.
This kind of study is motivated by the link (proved in \cite{GGP1}) of the multiplicity
problem with the famous Seifert conjecture (formulated in 1948 (ref.\ \cite{seifert})
about multiple brake orbits for a class of Hamiltonian systems at a fixed energy level.
\end{abstract}

\maketitle
\renewcommand{\contentsline}[4]{\csname nuova#1\endcsname{#2}{#3}{#4}}
\newcommand{\nuovasection}[3]{\medskip\hbox to \hsize{\vbox{\advance\hsize by -1cm\baselineskip=12pt\parfillskip=0pt\leftskip=3.5cm\noindent\hskip -2cm #1\leaders\hbox{.}\hfil\hfil\par}$\,$#2\hfil}}
\newcommand{\nuovasubsection}[3]{\medskip\hbox to \hsize{\vbox{\advance\hsize by -1cm\baselineskip=12pt\parfillskip=0pt\leftskip=4cm\noindent\hskip -2cm #1\leaders\hbox{.}\hfil\hfil\par}$\,$#2\hfil}}

\tableofcontents

\section{Introduction}\label{sec:intro}
We first recall the famous
conjecture formulated by H. Seifert in 1948 (ref.\ \cite{seifert})
concerning the number of brake orbits of a certain Hamiltonian
system in $\R^{2N}$.

\subsection{Brake orbits for a Hamiltonian system and statement of the Seifert conjecture}
Denote by $(q,p)=(q_1,\ldots,q_N,p_1,\ldots,p_N)$ the canonical coordinates
in $\R^{2N}$, and consider a $C^2$ Hamiltonian function $H:\R^{2N}\to \R$ of the form
\begin{equation}\label{eq:Ham}
H(q,p)=\frac12\sum_{i,j=1}^N a^{ij}(q)p_ip_j+V(q),
\end{equation}
where $q\mapsto(a^{ij}(q))_{i,j}$ is a map of class $C^2$ taking value in
the space of (symmetric) positive definite $N\times N$ matrices and $V:\R^N\to\R$
is the potential energy.
The corresponding Hamiltonian system is:
\begin{equation}\label{eq:HS}
\left\{
\begin{aligned}
&\dot p=-\frac{\partial H}{\partial q}\\
&\dot q=\frac{\partial H}{\partial p},
\end{aligned}
\right.
\end{equation}
where the dot denotes differentiation with respect to time. Since
the system \eqref{eq:HS} is  time independent, the total energy $E$
is constant along each solution, and there exists a large amount of
literature concerning the study of periodic solutions of autonomous
Hamiltonian systems with prescribed energy (see e.g.\ \cite{LZZ,Long}
and references therein).

For all $q\in\R^N$, denote by $\Lcal(q):\R^N\to \R^N$ the linear
isomorphism whose matrix with respect to the canonical basis is
$(a_{ij}(q))$, the inverse of $(a^{ij}(q))$. It is easily seen that,
if $(q,p)$ is a $C^1$ solution of \eqref{eq:HS}, then $q$ is
actually a map of class $C^2$ and
\begin{equation}\label{eq:p}
p=\Lcal(q)\dot q.
\end{equation}

A special class of periodic solutions of \eqref{eq:HS} consists of the so
called \emph{brake orbits}. A brake orbit for the system \eqref{eq:HS} is a
nonconstant solution $(q,p):\R\to\R^{2N}$ such that $p(0)=p(T)=0$
for some $T>0$; since $H$ is even in the momenta $p$, then a brake orbit is $2T$--periodic.
Moreover, if $E$ is the energy of a brake orbit $(q,p)$, then $V(q(0))=V(q(T))=E$.

Let $E>\inf V$ be fixed; we will make the following assumptions:
\begin{equation}\label{eq:regolare}
E \text{ is a regular value of }V, \text{ namely d}V(q)\not=0 \text{ for
all }q\in V^{-1}(E),
\end{equation}
\begin{equation}\label{eq:diffeo-disco}
\text{the closed sublevel }
V^{-1}(\left]-\infty,E\right]) \text{ is homeomorphic to the }N\text{--disk }
\mathds D^N,
\end{equation}
(here $\text d V(q)$ denotes the differential of $V$).
We refer the reader e.g. to
\cite{BK,GZ,LZZ,S,W} for multiplicity results of brake orbits.

In \cite{BK} and \cite{GZ}, under assumptions \eqref{eq:regolare} and \eqref{eq:diffeo-disco}, it is proved the existence of at least $N$ brake orbits, but using a very strong (ad hoc) assumption on the energy integral, to ensure that different minimax critical levels correspond to geometrically distinct brake orbits. (By geometrically distinct curves we mean curves whose images are distinct sets).

Papers \cite{LZZ,S,W} deal with more general Hamiltonian systems not necessarily of the form \eqref{eq:Ham}.
In \cite{S} it is proved that if $H$ is even in the variable $p$, and $H^{-1}(1)$ bounds
a star-shaped region and satisfies a suitable geometric condition (that allows once again to obtain geometrically distinct brake orbits), then there are at least $N$ brake orbits on $H^{-1}(1)$. In \cite{W} the author prove that for any $h$ sufficiently close to $H(z_0)$ with $z_0$ nondegenerate local minimum for $H$,
there are $N$ geometrically distinct brake orbits in $H^{-1}(h)$. This result generalize (in the case of a minimum point) a classical theorem of Lyapunov for nondegenerate critical points having the eigenvalues of the Hessian matrix with non integer ratio.

In \cite{LZZ} if $H$ is convex and even in both variables $p$ and $q$ the authors prove the existence of at least 2 geometrically distinct brake orbits on
$H^{-1}(1)$, provided that it is regular and compact.

In the paper \cite{seifert}, under assumptions \eqref{eq:regolare} and \eqref{eq:diffeo-disco}, Seifert proved the existence of a brake orbit, and he also conjectured the existence of at least
$N$ geometrically distinct brake orbits.
Up to the present day, the conjecture has been neither proved nor disproved.
\smallskip

It is well known that the lower estimate for the number of brake orbits
given in the Seifert conjecture cannot be improved in any dimension $N$.
Indeed, consider the Hamiltonian:
\[
H(q,p)=\frac12|p|^2+\sum_{i=1}^N\lambda_i^2 q_i^2,\qquad (q,p)\in\R^{2N},
\]
where $\lambda_i\not=0$ for all $i$. If $E>0$ and the squared ratios
$\left({\lambda_i}/{\lambda_j}\right)^2$ are irrational for all $i\ne j$,
then the only periodic solutions of \eqref{eq:HS} with energy
$E$ are the $N$ brake orbits moving along the axes of the ellipsoid with equation
\[
\sum_{i=1}^N\lambda_i^2q_i^2=E.
\]

In \cite{GGP1} it has been pointed out that the problem of finding brake orbits is equivalent to find
orthogonal geodesic chords on manifold with concave boundary.

In this paper we use a suitable functional approach
to prove the existence of one orthogonal geodesic chords for a manifold with concave boundary homeomorphic to the $N$--dimensional disk. This is the first step in view of obtaining multiplicity results for orthogonal geodesic chords in this kind of situation, and a new way to study Seifert conjecture.
\medskip

Let us now recall a few basic facts and notations from \cite{GGP1}.

\subsection{Geodesics in Riemannian Manifolds with Boundary}\label{sub:geodes}
Let $(M,g)$ be a smooth
Riemannian manifold with $\Dim(M)=N\ge2$. The symbol $\nabla$ will
denote the covariant derivative of the Levi-Civita connection of
$g$, as well as the gradient differential operator for smooth maps
on $M$. The Hessian $\mathrm H^f(q)$ of a smooth map $f:M\to\R$ at
a point $q\in M$ is the symmetric bilinear form $\mathrm
H^f(q)(v,w)=g\big((\nabla_v\nabla f)(q),w\big)$ for all $v,w\in
T_qM$; equivalently, $\mathrm H^f(q)(v,v)=\frac{\mathrm
d^2}{\mathrm ds^2}\big\vert_{s=0} f(\gamma(s))$, where
$\gamma:\left]-\varepsilon,\varepsilon\right[\to M$ is the unique
(affinely parameterized) geodesic in $M$ with $\gamma(0)=q$ and
$\dot\gamma(0)=v$. We will denote by $\Ddt$ the covariant
derivative along a curve, in such a way that $\Ddt\dot x=0$ is the
equation of the geodesics. A basic reference on the background material
for Riemannian geometry is \cite{docarmo}.

Let $\Omega\subset M$ be an open subset;
$\overline\Omega=\Omega\bigcup\partial \Omega$ will denote its
closure. There are several notion of convexity and concavity in
Riemannian geometry, extending the usual ones for subsets of the
Euclidean space $\R^N$. In this paper we will use a somewhat
strong concavity assumption for $\overline\Omega$, that we
will call "strong concavity" below, and which is stable by
$C^2$-small perturbations of the boundary. Let us first recall the
following definition (cf also \cite{BFG,GM}):

\begin{defin}\label{thm:defconvexity}
$\overline\Omega$ is said to be {\em convex\/} if every geodesic
 $\gamma:[a,b]\to \overline\Omega$ whose endpoints
$\gamma(a)$ and $\gamma(b)$ are in $\Omega$ has image entirely
contained in $\Omega$. Likewise, $\overline\Omega$ is said to be
{\em concave\/} if its complement $M\setminus\overline\Omega$ is
convex.
\end{defin}

If $\partial \Omega$ is a smooth embedded submanifold of $M$, let
$\second_{\mathfrak n}(x):T_x(\partial\Omega)\times
T_x(\partial\Omega)\to\R$ denote the {\em second fundamental form
of $\partial\Omega$ in the normal direction $\mathfrak n\in
T_x(\partial\Omega)^\perp$}. Recall that $\second_{\mathfrak
n}(x)$ is a symmetric bilinear form on $T_x(\partial\Omega)$
defined by:
\[\phantom{\qquad v,w\in T_x(\partial\Omega),}\second_{\mathfrak n}(x)(v,w)=g(\nabla_vW,
\mathfrak n),\qquad v,w\in T_x(\partial\Omega),\] where $W$ is any
local extension of $w$ to a smooth vector field along
$\partial\Omega$.

\begin{rem}\label{thm:remphisecond}
Assume that it is given a smooth function $\phi:M\to\R$ with the
property that $\Omega=\phi^{-1}\big(\left]-\infty,0\right[\big)$
and $\partial\Omega=\phi^{-1}(0)$, with $\nabla\phi\ne0$ on
$\partial\Omega$. \footnote{For example one can choose $\phi$ such that
$\vert\phi(q)\vert=\dist(q,\partial\Omega)$ for all $q$ in a
(closed) neighborhood of $\partial\Omega$, where
$\dist$ denotes the distance function on $M$ induced by $g$.}
The following equality
between the Hessian $\mathrm H^\phi$ and the second fundamental
form\footnote{%
Observe that, with our definition of $\phi$, $\nabla\phi$ is
a normal vector to $\partial\Omega$ pointing {\em outwards\/} from
$\Omega$.} of $\partial\Omega$ holds:
\begin{equation}\label{eq:seches}
\phantom{\quad x\in\partial\Omega,\ v\in
T_x(\partial\Omega);}\mathrm H^\phi(x)(v,v)=
-\second_{\nabla\phi(x)}(x)(v,v),\quad x\in\partial\Omega,\ v\in
T_x(\partial\Omega);\end{equation} Namely, if
$x\in\partial\Omega$, $v\in T_x(\partial\Omega)$ and $V$ is a
local extension around $x$ of $v$ to a vector field which is
tangent to $\partial\Omega$, then $v\big(g(\nabla\phi,V)\big)=0$
on $\partial\Omega$, and thus:
\[\mathrm H^\phi(x)(v,v)=v\big(g(\nabla\phi,V)\big)-g(\nabla\phi,\nabla_vV)=-\second_{\nabla\phi(x)}(x)(v,v).\]

For convenience, we will fix throughout the paper a function $\phi$
as above. We observe that, although the second fundamental form
is defined intrinsically, there is no a canonical choice for a
function $\phi$ describing the boundary of $\Omega$ as above.
\end{rem}

\begin{defin}\label{thm:defstrongconcavity}
We will say that that $\overline\Omega$
is {\em strongly concave\/} if $\second_{\mathfrak n}(x)$ is
negative definite for all $x\in\partial\Omega$ and all inward
pointing normal direction $\mathfrak n$.
\end{defin}

Note that if $\overline\Omega$ is  strongly concave, then geodesics
starting on $\partial\Omega$ tangentially to        $\partial\Omega$
locally move \emph{inside} $\Omega$.

\begin{rem}\label{rem:1.4}
Strong concavity is evidently a {\em $C^2$-open condition}.
Then, if $\overline\Omega$ is compact, there exists $\delta_0>0$
such that $H^\phi(q)[v,v]<0$ for any $q$ such that $\phi(q)\in[-\delta_0,\delta_0]$
and $g(\nabla\phi(q),v)=0$.
A simple contradiction argument based on Taylor expansion shows that
under the above conditions there are not critical points of $\phi$ for all
$q\in\phi^{-1}\big([-\delta_0,\delta_0]\big)$. In particular
the gradient flow of $\phi$ gives a strong deformation
retract of  $\phi^{-1}\big([-\delta_0,\delta_0]\big)$
onto $\phi^{-1}(0)=\partial\Omega$.
\end{rem}

\begin{rem}\label{rem:1.4bis}
The strong concavity condition gives us the following property, that will
be systematically used throughout the paper. Let $\delta_0$ be as in Remark~\ref{rem:1.4};
then:
\begin{equation}\label{eq:1.1bis}
\begin{matrix}
\text{for any non constant geodesic $\gamma:[a,b]\to\overline\Omega$ with $\phi(\gamma(a)=\phi(\gamma(b))=0$}\\
\text{and $\phi(\gamma(s))<0\,\forall s\in ]a,b[$, there exists $\overline s\in ]a,b[$ such that $\phi(\gamma(\overline s))<-\delta_0$.}
\end{matrix}
\end{equation}
Such property is proved easily by a contradiction argument obtained by looking
at the minimum point of the map $s\mapsto\phi(\gamma(s))$.
\end{rem}

The main objects of our study are geodesics in $M$ having image in
$\overline\Omega$ and with endpoints orthogonal to
$\partial\Omega$, that will be called {\em orthogonal geodesic chords}:

\begin{defin}\label{thm:defOGC}
A geodesic $\gamma:[a,b]\to M$ is called a {\em geodesic chord\/}
in $\overline\Omega$ if
$\gamma\big(\left]a,b\right[\big)\subset\Omega$ and
$\gamma(a),\gamma(b)\in\partial\Omega$; by a {\em weak geodesic
chord\/} we will mean a geodesic $\gamma:[a,b]\to M$ with image in
$\overline\Omega$, endpoints
$\gamma(a),\gamma(b)\in\partial\Omega$ and such that $\gamma(s_0) \in \partial \Omega$
for some $s_0 \in ]a,b[$. A (weak) geodesic chord is
called {\em orthogonal\/} if $\dot\gamma(a^+)\in
(T_{\gamma(a)}\partial\Omega)^\perp$ and $\dot\gamma(b^-)\in
(T_{\gamma(b)}\partial\Omega)^\perp$, where
$\dot\gamma(\,\cdot\,^\pm)$ denote the lateral derivatives (see
Figure~\ref{fig:WOGC}).
\end{defin}

\begin{figure}
\begin{center}
\psfull \epsfig{file=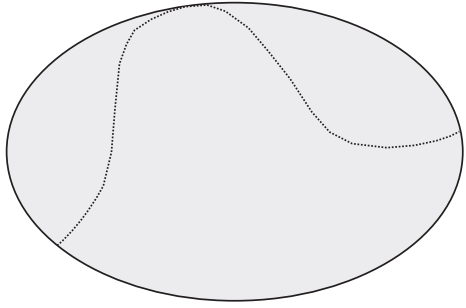, height=4cm} \caption{A weak
orthogonal geodesic chord (WOGC) $\gamma$
in~~$\overline\Omega$.}\label{fig:WOGC}
\end{center}
\end{figure}

For shortness, we will write \textbf{OGC} for ``orthogonal
geodesic chord'' and \textbf{WOGC} for ``weak orthogonal geodesic
chord''. Although the general class of weak orthogonal geodesic
chords are perfectly acceptable solutions of our initial geometrical
problem, our suggested construction of a variational setup works well
only in a situation where one can
exclude {\em a priori\/} the existence in $\overline\Omega$ of
orthogonal geodesic chords $\gamma:[a,b]\to\overline\Omega$ for
which there exists $s_0\in\left]a,b\right[$ such that
$\gamma(s_0)\in\partial\Omega$.

One does not lose generality in assuming
that there are no such WOGC's in $\overline\Omega$ by recalling the
following result from \cite{GGP1}:

\begin{prop}
\label{thm:noWOGC}
Let $\Omega\subset M$ be an open set
whose boundary $\partial\Omega$ is smooth and compact and with $\overline\Omega$
strongly concave.
Assume that there are only a finite number of orthogonal
geodesic chords in $\overline\Omega$. Then, there exists an
open subset $\Omega'\subset\Omega$ with the following properties:
\begin{enumerate}
\item\label{itm:nowogcs1} $\overline{\Omega'}$ is diffeomorphic to $\overline\Omega$
and it has smooth boundary;
\item\label{itm:nowogcs2} $\overline{\Omega'}$ is strongly concave;
\item\label{itm:nowogcs3} the number of   OGC's in $\overline{\Omega'}$  is less than
or equal to the number of   OGC's in $\overline\Omega$ ;
\item\label{itm:nowogcs4} there are not WOGC's in $\overline{\Omega'}$.
\end{enumerate}
\end{prop}
\begin{proof}
See \cite[Proposition~2.6]{GGP1}
\end{proof}

In the central result of this paper we shall give the proof of the existence of at least an OGC.
Note that, whenever $\partial\Omega$ is not connected the existence problem has a very simple solution
without assumptions on the geometry of $\overline\Omega$. Indeed it suffices to minimize the energy functional on the curves joining two different components of $\partial\Omega$. Whenever $\partial\Omega$ is connected the problem in general has no solution. A simple and famous counterexample can be found in
\cite{bos}.

The situation is different if we assume $\overline\Omega$ to be convex or concave.
In the convex case we wish first to recall the result proved by Bos in \cite{bos}: if $\partial\Omega$ is smooth,
$\overline\Omega$ convex and homeomorphic to the $N$-dimensional
disk, then there are at least $N$ geometrically distinct OGC's for $\overline\Omega$.
Such a result is a generalization
of the classical one of Ljusternik and Schnirelman (see \cite{LustSchn}),
where they treated convex subsets of $\R^N$
endowed with the Euclidean metric. Always in the convex case, in \cite{GiaMa}
it is studied the effect of the topology of $\overline\Omega$ on the number of OGC's.

\medskip

In this paper we prove the existence result in the concave case:
\begin{teon}
Let $\Omega$ be an open subset of $M$ with smooth boundary $\partial\Omega$,
such that $\overline\Omega$ is strongly concave and homeomorphic to
an $N$--dimensional disk. Suppose that there are not WOGC's.
Then there is at least an
orthogonal geodesic chord in $\overline\Omega$.
\end{teon}

\smallskip

Existence and multiplicity of OGC's in the case of compact manifolds having convex
boundary is typically proven by applying a curve-shortening argument.
From an abstract viewpoint, the curve-shortening process can be seen
as the construction of a flow in the space of paths, along whose trajectories
the length or energy functional is decreasing.

In this paper we will follow the same procedure, with the difference
that both the space of paths and the shortening flow have
to be defined appropriately.

Shortening a curve having image in a closed convex subset $\overline\Omega$
of a Riemannian manifold produces another curve in $\overline\Omega$; in
this sense, we think of the shortening flow as being ``inward pushing''
in the convex case.
As opposite to the convex
case, the shortening flow in the concave case will be ``outwards pushing'', and
this fact requires that one should consider only those portions of a curve
that remain inside $\overline\Omega$ when it is stretched outwards.

Unfortunately this kind of approach produces a lot of technical difficulties.
For this reason we shall try to describe the main ideas of the paper
before to face the main difficulties.

Obviously the concavity condition plays a central role in the variational setup of our
construction. ``Variational criticality'' relatively to the energy
functional will be defined in terms of ``outwards pushing'' infinitesimal
deformations of the path space (see Definition~\ref{thm:defvariatcrit}).
The class of variationally critical portions contains properly the set of
portions of paths consisting of  OGC's.
In order to construct the shortening flow, an accurate analysis of all
possible variationally critical portions is required (Section~\ref{sub:description}),
and the concavity condition will guarantee that such paths are
\emph{well behaved} (see Lemma~\ref{thm:leminsez4.4}, Proposition~\ref{thm:regcritpt}
and Proposition~\ref{thm:irregcritpt}). Thanks their good behavior, it will be possible
to move far away from critical portions which are not OGC's (choosing a suitable class of admissible homotopies).

\subsection{The functional framework}\label{sub:varframe}

Throughout the paper, $(M,g)$ will denote a Riemannian manifold
of class $C^2$;
all our constructions will be made in suitable (relatively) compact
subsets of $M$, and for this reason it will not be restrictive
to assume, as we will, that $(M,g)$ is complete.
Furthermore, we will work mainly in open subsets $\Omega$ of $M$ whose closure
is  homeomorphic to an $N$-dimensional disk, and
in order to simplify the exposition we will assume that, indeed,
$\overline\Omega$ is embedded topologically in $\R^N$, which will
allow to use an auxiliary linear structure on a  neighborhood
of $\overline\Omega$.
We will also assume that $\overline\Omega$ is strongly concave in $M$.

Recall that we have fixed a smooth function $\phi:M\to\R$ such that
$\Omega=\phi^{-1}\big(\left]-\infty,0\right[\big)$, $\nabla\phi\ne0$ on $\partial
\Omega$. The strong concavity of $\overline\Omega$ means
that the Hessian $\mathrm H^\phi$ is negative definite on $T(\partial\Omega)$.
As observed in Remark \ref{rem:1.4} there exists $\delta_0 > 0$ such that $\nabla \phi \neq 0$ and $H^{\phi}$ is negative definite in
$\phi^{-1}([-\delta_0,\delta_0])$.

\medskip

The symbol $H^1\big([a,b],\R^N\big)$ will denote the Sobolev
space of all absolutely continuous curves $x:[a,b]\to\R^N$
whose weak derivative is square integrable.
By $H^{1}_{0}\big([a,b],\R^N\big)$ we will denote the subspace of $H^1\big([a,b],\R^N\big)$
consisting of curves $x$ such that $x(a) = x(b) = 0$.
The Hilbert space norm of $H^1\big([a,b],\R^N\big)$ will be denoted
by $\Vert\cdot\Vert_{a,b}$:
\begin{equation}\label{eq:1.8fabio}
\Vert x\Vert_{a,b}=
\frac{1}{\sqrt 2}
\left(\max(\Vert x(a)\Vert^2_E,\Vert x(b)\Vert^2_E)+\int_a^b\Vert\dot x\Vert^2_E\right)^{1/2}
\end{equation}
where $\Vert\cdot\Vert_E$ denotes the Euclidean norm. Note that a simple computation shows that the above norm is equivalent to the usual one
in $H^1\big([a,b],\R^N\big)$. Note also that, setting \[\Vert x\Vert_{L^\infty([a,b],\mathbb{R}^N)}=\sup_{s\in[a,b]}\Vert x(s)\Vert_E,\]
the following inequality holds:
\begin{equation}\label{eq:1.10bis}
\Vert x\Vert_{L^\infty([a,b],\mathbb{R}^N)}\le \Vert x \Vert_{a,b}.
\end{equation}
 We shall use also the space $H^{2,\infty}$
which consists of curves having second derivative in $L^\infty$.

\begin{rem}\label{rem:rep}
In the development of our results, we will have to deal with
curves $x$ with variable domains $[a,b]\subset[0,1]$.
In this situation, by $H^1$-convergence (resp., weak $H^1$-convergence,
uniform convergence) of a sequence $x_n:[a_n,b_n]\to M$
to a curve $x:[a,b]\to M$ we will mean that $a_n$ tends to $a$, $b_n$ tends
to $b$ and  $\widehat x_n:[a,b]\to M$
is $H^1$-convergent (resp., weakly $H^1$-convergent,
uniformly convergent) to $x$ on $[a,b]$ as $n\to\infty$, where
$\widehat x_n$ is the unique affine reparameterization of $x$ on the interval
$[a,b]$.
\end{rem}

In the rest of the paper it will be useful to consider
the  flows $\eta^{+}(\tau,x)$ and $\eta^{-}(\tau,x)$ on the Riemannian manifold $M$ defined by
\begin{equation}\label{eq:flusso+}
\left\{\begin{array}{l}\dfrac{\mathrm d}{\mathrm d\tau}{\eta}^{+}(\tau)=
\dfrac{\nabla \phi(\eta^+)}{\Vert \nabla \phi(\eta^+) \Vert^{2}} \\[.5cm]
{\eta}^{+}(0)=x \in\big\{y \in M: -\delta_0 \leq \phi(y) \leq \delta_0\big\}, \end{array} \right.
\end{equation}

and

\begin{equation}\label{eq:flusso-}
\left\{\begin{array}{l}\dfrac{\mathrm d}{\mathrm d\tau}{\eta}^{-}(\tau)=
\dfrac{-\nabla \phi(\eta^-)}{\Vert \nabla \phi(\eta^-) \Vert^{2}} \\[.5cm]
{\eta}^{-}(0)=x \in\big\{y \in M: -\delta_0 \leq \phi(y) \leq \delta_0\big\}. \end{array} \right.
\end{equation}

Note that $\eta^{+}(\tau,x)$ and $\eta^{-}(\tau,x)$ are certainly well defined on $\phi^{-1}([-\delta_0,\delta_0])$, where $\nabla \phi \not=0$.

We conclude this section with the introduction of the following constant:
\begin{equation}\label{eq:1.9fabio}
K_0=\max_{x\in\phi^{-1}(]-\infty,\delta_0]}\Vert\nabla\phi(x)\Vert,
\end{equation}
where $\Vert\cdot\Vert$ is the norm induced by $g$, and $\delta_0$ is as in Remark \ref{rem:1.4}.

\section{Path space and maximal intervals}
\label{sec:pathspace}
In this section we shall describe the space of curves $\mathfrak M$ which will be the environment of our minimax approach
and the set $\mathfrak C \subset \mathfrak M$ homeomorphic
to  $\mathds{S}^{N-1} \times \mathds{S}^{N-1}$, carrying on all the topological information of $\mathfrak M$.

\bigskip

For $\mathcal C\subset\R^N$ and $a<b$, define:
\begin{equation}\label{eq:defH1abA}
H^1\big([a,b],\mathcal C\big)=\big\{x\in H^1\big([a,b],\R^N\big):x(s)\in \mathcal C\
\text{for all $s\in[a,b]$}\big\}.
\end{equation}

Let $\delta_0>0$ be as in Remark \ref{rem:1.4}.
Consider first the following set of paths
\begin{equation}\label{eq:defM}
\mathfrak M_0=\Big\{x\in
H^1\big([0,1],\phi^{-1}(\left]-\infty,\delta_0\right[)\big):
\phi(x(0))\ge 0,\,\phi(x(1))\ge 0\Big\},
\end{equation}
see Figure \ref{fig:3bis}.
\begin{figure}
\begin{center}
\psfull \epsfig{file=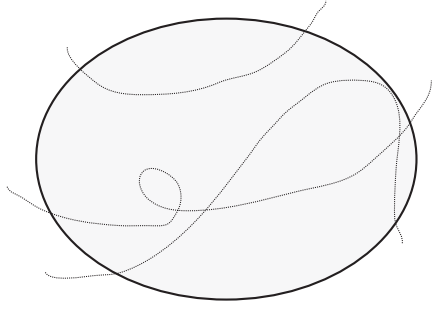, height=4cm} \caption{
Curves representing typical elements of the path space $\mathfrak
M_0$.}\label{fig:3bis}
\end{center}
\end{figure}
This is a subset of the Hilbert space $H^1\big([0,1],\R^N\big)$, and
it will be topologized with the induced metric.

The following result will be used systematically throughout the
paper:
\begin{lem}\label{thm:lem2.3}
If $x\in\mathfrak M_0$ and $[a,b]\subset[0,1]$ is such that
$x(a)\in\partial\Omega$ and there exists $\bar s\in[a,b]$ such that
$\phi(x(\bar s))\le-\delta<0$, then:
\begin{equation}\label{eq:2.9fabio}
b-a\ge\frac{\delta^2}{K_0^2}\left(\int_a^bg(\dot x,\dot x)\,\mathrm
ds\right)^{-1} ,
\end{equation}
where $K_0$ is defined in \eqref{eq:1.9fabio}.
\end{lem}
\begin{proof}
Since $\phi(x(a))=0$, if $\bar s$ is such that $\phi(x(\bar
s))=-\delta<0$, it is:
\[\begin{split}
\delta\le&\;\vert\phi(x(\bar s))-\phi(x(a))\vert\le\int_a^{\bar
s}\vert g\big(\nabla\phi(x(\sigma)),\dot x(\sigma)\big)\vert\,
\mathrm d\sigma \le\int_a^b            \vert
g\big(\nabla\phi(x(\sigma)),\dot x(\sigma)\big)\vert\, \mathrm
d\sigma\\& \le K_0\int_a^bg(\dot x,\dot x)^{\frac12}  \mathrm
d\sigma\le K_0\sqrt{b-a} \left(  \int_a^bg(\dot x,\dot x) \,\mathrm
d\sigma\right)^{\frac12},
\end{split}\]
from which \eqref{eq:2.9fabio} is easily deduced.
\end{proof}

For all $x\in\mathfrak M_0$, let $\mathcal I_{x}^{0}, \mathcal I_x$
denote the following collections of closed subintervals of $[0,1]$:
\[
\mathcal I_{x}^{0} = \big\{[a,b]\subset [0,1]: x([a,b]) \subset
\overline\Omega, x(a),x(b) \in \partial\Omega \big\},
\]
\[
\mathcal I_x=\big\{[a,b]\in \mathcal I_{x}^{0}
\text{ and $[a,b]$ is maximal with respect to this
property}\big\}.
\]
\smallskip

Although we will only prove an existence result, in this paper we will employ an
equivariant approach that aims also at multiplicity results, to be developed in the future.

To this aim we shall consider
the map $\Rcal:\mathfrak M_0 \to \mathfrak M_0 $:
\begin{equation}\label{eq:2.4}
\Rcal x(t)=x(1-t).
\end{equation}
We say that $\Ncal\subset \mathfrak M_0$ is
\emph{$\Rcal$--invariant} if $\Rcal(\Ncal)=\Ncal$; note that
$\mathfrak M_0$ is $\Rcal$-invariant. If $\Ncal$ is $\mathcal
R$-invariant, a homotopy $h:[0,1]\times\Ncal\to\mathfrak M_0$ is
called \emph{$\Rcal$--equivariant} if
\begin{equation}\label{eq:2.5}
h(\tau,\Rcal x)=\Rcal h(\tau,x),\quad\forall x\in\Ncal,\,\forall\tau\in[0,1].
\end{equation}

\medskip

The following Lemma allows to describe an $\Rcal$--invariant subset $\mathfrak C$ of ${\mathfrak M}_{0}$
which carries on  all the topological properties of ${\mathfrak M}_{0}$.

\begin{lem}\label{thm:corde} There exists there a continuous map $G:\partial\Omega\times\partial\Omega\to H^1([0,1],\overline\Omega)$ such that
\begin{enumerate}
\item\label{corde1} $G(A,B)(0)=A,\,\,G(A,B)(1)=B$;
\item\label{corde2} $A\not=B\,\Rightarrow\, G(A,B)(s)\in\Omega\,\forall s\in ]0,1[$;
\item\label{corde3} $G(A,A)(s)=A\,\forall s\in [0,1]$;
\item\label{corde4} $\Rcal G(A,B)=G(B,A)$, namely $G(A,B)(1-s)=G(B,A)(s)$ for all $s$, and for all $A,B$.
\end{enumerate}
\end{lem}

\begin{proof}
If $\overline\Omega$ is {\sl diffeomorphic} to the $N$--dimensional
disk $\mathds D^N=\{q\in\R^N\,:\,\Vert q \Vert_{E}\le 1\}$ and
$\psi:\overline\Omega \to \mathds D^N$ is a diffeomorphism, we set
\[
G(A,B)(s) = \psi^{-1}\big((1-s)\psi(A)+s\psi(B)\big),\,A,B\in \overline\Omega.
\]

In general, if $\overline\Omega$ is only homeomorphic to the
disk $\mathds D^N$, then the above definition produces curves
that in principle are only continuous. In order to produce curves
with an $H^1$-regularity, we use a broken geodesic approximation
argument. More precisely, choose a homeomorphism $\psi:\overline\Omega\to\mathds D^N$ and
consider for any $A,B\in \overline\Omega$ the curve
\[
c_{A,B}(s) = \psi^{-1}\big((1-s)\psi(A)+s\psi(B)\big)
\]
and the set
\[
\mathfrak C'= \{c_{A,B}: \,A,B \in \overline\Omega\}.
\]
Denote by $\varrho(\overline\Omega,g)$ the
infimum of the injectivity radii of all points of
 $\overline\Omega$ relatively
to the metric $g$ (cf. \cite{docarmo}). By the compactness of $\mathfrak C'$, there exists
$N_0\in\N$ with the property that
$\dist\big(c_{A,B}(a),c_{A,B}(b)\big)\le\varrho(\overline\Omega,g)$ whenever
 $\vert a-b\vert\le\frac1{N_0}$ for all
$c_{A,B}\in\mathfrak C'$. Finally, for all $c_{A,B}\in\mathfrak C'$, denote by
$\gamma_{A,B}$ the broken geodesic obtained as concatenation of the
curves $\gamma_k:[\frac{k-1}{N_0},\frac{k}{N_0}]\to M$ given by
the unique minimal geodesic in $(M,g)$ from $c_{A,B}(\frac{k-1}{N_0})$
to $c_{A,B}(\frac{k}{N_0})$, $k=1,\ldots,N_0$.

Note that, by compactness, $N_0$ can be chosen large enough to have
\[
\phi(\gamma_{A,B}(s)) < \delta_0 \text{ for any }s \in [0,1], \text{ for any }A,B \in \overline\Omega.
\]
Since the minimal geodesic in any convex normal neighborhood depend
continuously (with respect to the $C^2$--norm) on its endpoints, $\gamma_{A,B}$ depends continuously by $(A,B)$ in the $H^{1}$--norm.
Moreover $\gamma_{A,B}$ satisfies \eqref{corde1}, \eqref{corde3} and \eqref{corde4} of the statement.

Now, using the flow $\eta^{-}$ defined by \eqref{eq:flusso-}, we can set
\[
\hat\gamma(A,B)(s) = \eta^{-}(\max(0,\phi\big(\gamma_{A,B}(s)\big),\gamma_{A,B}(s)),
\]
so that
\[
\hat\gamma(A,B)(s) \in \overline\Omega \text{ for any }s \in [0,1], \text{ for any }A,B \in \overline\Omega.
\]
Since $\gamma(A,B)(0)=A$ and $\gamma(A,B)(1)=B$, it is $\phi(\gamma(A,B)(0))=\phi(\gamma(A,B)(1))=0$ and therefore also
$\hat\gamma_{A,B}$ satisfies \eqref{corde1}, \eqref{corde3} and \eqref{corde4} of the statement. Finally,
\[
G(A,B)(s) = \eta^{-}\Big(s(1-s)\max\big(0,\frac{\delta_0}2+\phi(\hat\gamma_{A,B}(s)\big),\hat\gamma_{A,B}(s)\Big),
\]
is the desired map.
\end{proof}

We set
\begin{equation}\label{eq:2.6bis}
\begin{split}
&\mathfrak C=\big\{G(A,B)\,:\,A,B\in\partial\Omega\big\},\\
&\mathfrak C_0=\{G(A,A)\,:\,A\in\partial\Omega\}.\end{split}\end{equation}

\begin{rem}\label{rem:conseguenze-homeo}
Note  that $\mathfrak C$ is homeomorphic to $\mathds{S}^{N-1}\times \mathds{S}^{N-1}$ by a homeomorphism
mapping $\mathfrak C_0$ onto $\{(A,A)\,:\,A\in \mathds{S}^{N-1}\}$, and that $\Rcal$
induces (by such homeomorphism) an action $\Scal$ on $\mathds{S}^{N-1}\times \mathds{S}^{N-1}$ given by
\begin{equation}\label{eq:2.6ter}
\Scal(A,B)=(B,A).
\end{equation}
\end{rem}

\medskip

Define now the following constant:
\begin{equation}\label{eq:defM0}
M_0=\sup_{x\in\mathfrak C}\int_0^1g(\dot x,\dot x)\,\mathrm dt.
\end{equation}
Since $\mathfrak C$ is compact and the integral in
\eqref{eq:defM0} is continuous in the $H^1$-topology, then
$M_0<+\infty$.
\bigskip
Finally we define the following subset of $\mathfrak M_0$:
\begin{equation}\label{eq:defMM}
\mathfrak M=\Big\{x\in\mathfrak M_0: \frac12\int_a^bg(\dot x,\dot
x)\,\mathrm dt< M_0 \quad \forall [a,b] \in \mathcal I_x \Big\}.
\end{equation}

We shall work in $\mathfrak M$ constructing flows in
$H^1\big([0,1],\R^N\big)$ for which $\mathfrak M$ is invariant.
We shall often use the notation
\begin{equation}\label{eq:fab}
f_{a,b}(x)=\frac12\int_a^b g(\dot x,\dot x)\,\mathrm dt.
\end{equation}

\section{Geometrically critical values and variationally
critical portions}\label{sec:functional}

In this section we will introduce two different notions
of {\em criticality\/} for curves in $\Mcal$.

\begin{defin}\label{thm:defgeomcrit}
A number $c\in\left]0,M_0\right[$ will be called a {\em geometrically critical
value} if there exists an OGC $\gamma$ parameterized in $[0,1]$ such that
$\frac12\int_0^1g(\dot\gamma,\dot\gamma)\,\text dt=c$.
A number which is not geometrically critical will be called  {\em geometrically regular value}.
\end{defin}

\bigskip

It is important to observe that, in view to obtain multiplicity results,
distinct geometrically critical values yield
geometrically distinct orthogonal geodesic chords:
\begin{prop}\label{thm:distinct}
Let $c_1\ne c_2$, $c_1,c_2>0$ be distinct geometrically critical
values  with corresponding OGC $x_1,x_2$. Then
$x_1\big([0,1]\big)\ne x_2\big([0,1]\big)$.
\end{prop}
\begin{proof}
The OGC's $x_1$ and $x_2$ are parameterized in the interval $[0,1]$. Assume by contradiction,
$x_1([0,1])=x_2([0,1])$. Since
\[
x_{i}(]0,1[) \subset \Omega \text{ for any }i=1,2,
\]
we have
\[\{x_1(0),x_1(1)\}=
\{x_2(0),x_2(1)\}.\]
Up to reversing the orientation of $x_2$,  we can assume $x_1(0)=x_2(0)$.
Since $x_1$ and $x_2$ are OGC's, $\dot x_1(0)$ and $\dot x_2(0)$
are parallel, but the condition $c_1\ne c_2$ says that $\dot x_1(0)\ne\dot x_2(0)$.
Then  there exists $\lambda > 0, \lambda\not=1$ such that $\dot x_2(0) =\lambda \dot x_1(0)$ and therefore,
by the uniqueness of the Cauchy problem for geodesics we have $x_2(s)=x_1(\lambda s)$. Up to exchanging $x_1$ with $x_2$,
we can assume $\lambda > 1$. Since $x_2(\frac{1}{\lambda}) = x_1(1) \in\partial\Omega$, the transversality of
$\dot x_2(0)$ to $\partial\Omega$ implies the existence of $\bar s \in ]\frac{1}{\lambda},1]$ such that
$x_2(\bar s) \not\in\overline\Omega$, getting a contradiction.
\end{proof}

A notion of criticality will now be given in terms of
variational vector fields.
For $x\in \Mcal$, let $\mathcal V^+(x)$ denote
the following closed convex cone of $T_x H^1\big([0,1],\R^N\big)$:
\begin{equation}\label{eq:defV+(x)}
\mathcal V^+(x)=\big\{V\in T_x H^1\big([0,1],\R^N\big):g\big(V(s),
\nabla\phi\big(x(s)\big)\big)\ge0\ \text{for $x(s)\in\partial\Omega$}\big\};
\end{equation}
vector fields in $\mathcal V^+(x)$ are interpreted as
infinitesimal variations of $x$ by curves stretching ``outwards''
from the set $\overline\Omega$.

\begin{defin}\label{thm:defvariatcrit}
Let $x\in\mathfrak M$ and $[a,b] \subset [0,1]$; we say that
$x_{\vert[a,b]}$ is a $\mathcal V^{+}$--\emph{\em variationally critical portion\/} of
$x$ if $x_{\vert[a,b]}$ is not constant and if
\begin{equation}
\label{eq:varcriticality}\int_a^bg\big(\dot x,\Ddt V\big)\,\mathrm
dt\ge0, \quad\forall\,V\in\mathcal V^+(x).
\end{equation}
\end{defin}

Similarly, for $x\in\Mcal$ we
define the cone:
\begin{equation}\label{eq:defV-(x)}
\mathcal V^-(x)=\big\{V\in T_x H^1\big([0,1],\R^N\big):g\big(V(s),
\nabla\phi\big(x(s)\big)\big)\le0\ \text{for $x(s)\in
\partial \Omega$}\big\},
\end{equation}
and we give the following

\begin{defin}\label{thm:defvariatcrit-}
Let $x\in\mathfrak M$ and $[a,b] \subset [0,1]$; we say that
$x_{\vert[a,b]}$ is a $\mathcal V^{-}$--\emph{\em variationally critical portion\/} of
$x$ if $x_{\vert[a,b]}$ is not constant and if
\begin{equation}
\label{eq:varcriticality-}\int_a^bg\big(\dot x,\Ddt V\big)\,\mathrm
dt\ge0, \quad\forall\,V\in\mathcal V^-(x).
\end{equation}
\end{defin}

The integral in \eqref{eq:varcriticality-} gives precisely the
first variation of the geodesic action functional in $(M,g)$ along
$x_{\vert[a,b]}$. Hence, variationally critical portions are
interpreted as those curves $x_{\vert[a,b]}$ whose geodesic energy
is {\em not decreased\/} after  infinitesimal variations by curves
stretching outwards from the set $\overline\Omega$. The motivation
for using outwards pushing infinitesimal variations is due to the
concavity of $\overline\Omega$. Indeed in the convex case it is customary
to use a curve shortening method in $\overline \Omega$,
that can be seen as the use of a flow constructed by
infinitesimal variations of $x$ in $\mathcal V^-(x)$,
keeping the endpoints of $x$ on $\partial\Omega$.

Flows obtained as integral flows of convex combinations of vector fields in $\mathcal V^+(x)$ are, in a certain sense,
the protagonists of our variational approach. However we shall use also integral flows
of convex combinations of vector fields in $\mathcal V^-(x)$ to avoid certain variationally critical portions that do not correspond
to OGC's.
\smallskip

Clearly, we are interested in determining existence of
geometrically critical values. In order to use a variational
approach we will first have to keep into consideration the more general
class of $\mathcal V^{+}$--variationally critical portions. A central issue in our
theory consists in studying the relations between $\mathcal V^{+}$--variational
critical portions $x_{\vert[a,b]}$ and OGC's. From now on
$V^{+}$--variationally critical portions, will be simple called
variationally critical portions.

\section{Classification of variationally critical portions}
\label{sub:description} Let us now take a look at how variationally
critical portions look like. We will interested to variationally
critical portion $x_{|[a,b]}$ such that $[a,b]\in \mathcal I_x^0$.

In first place, let us show
that smooth variationally critical points are OGC's; we need
a preparatory result:

\begin{lem}\label{thm:leminsez4.4}
Let $x\in\mathfrak M$ be fixed, and let $[a,b]\in \mathcal I_x^0$ be such
that $x_{\vert[a,b]}$ is a (non--constant) variationally critical
portion of $x$. Then:
\begin{enumerate}
\item\label{itm:freegeo} if $[\alpha,\beta] \subset [a,b]$ is such that $x(]\alpha,\beta[) \subset \Omega $,
then $x_{|[\alpha,\beta]}$ is a geodesic;
\item\label{itm:numfinint} $x^{-1}\big(\partial\Omega)\cap[a,b]$ consists of
a finite number of closed intervals and isolated points;
\item\label{itm:constconcomp} $x$ is constant on each connected
component of $x^{-1}\big(\partial\Omega)\cap[a,b]$;
\item\label{itm:piecC2} $x_{\vert[a,b]}$ is piecewise $C^2$, and the discontinuities of
$\dot x$ may occur only at points in $\partial\Omega$;
\item\label{itm:portgeo} each $C^2$ portion of $x_{\vert[a,b]}$ is a geodesic
in $\overline\Omega$.
\item\label{itm:addizionale} $\inf\{\phi(x(s))\,:\,s\in[a,b]\}<-\delta_0$.
\end{enumerate}
\end{lem}
\begin{proof}
Let $[\alpha,\beta]\subset[a,b]$ be such that $x([\alpha,\beta])\subset
\Omega$. In this case the set of the
restrictions to $[\alpha,\beta]$ of the vector fields in $\mathcal V^+(x)$ contains
the Hilbert space $H^1_0\big([\alpha,\beta],\R^N)$.
Variational criticality for $x$ implies then:
\[\int_\alpha^\beta g\big(\dot x,\Dds V)\,\mathrm ds\ge0,\quad\text{and}
\int_\alpha^\beta g\big(\dot x,-\Dds V)\,\mathrm ds\ge0,\] i.e.,
$\int_\alpha^\beta g\big(\dot x,\Dds V)\,\mathrm ds=0$ for all
$V\in H^1_0\big([\alpha,\beta],\R^N\big)$, which tells us that
$x_{\vert[\alpha,\beta]}$ is a geodesic in $\Omega$.
Similarly, if $[\alpha,\beta]\subset[a,b]$ is such that
$x\big(\left]\alpha,\beta\right[\big)\subset\Omega$ and
$x(\alpha),x(\beta)\in\partial\Omega$, by a limit argument we get
that $x_{\vert[\alpha,\beta]}$ is a geodesic proving \eqref{itm:freegeo}. Now observe that
the strong concavity assumption on $\overline\Omega$ implies (see
(\ref{eq:1.1bis})) that there exists $\bar
t\in\left]\alpha,\beta\right[$ such that $\phi(x(\bar
t))<-\delta_0$. By Lemma~\ref{thm:lem2.3},
$\beta-\alpha\ge\frac{\delta_0^2}{K_0^2}\left(\int_\alpha^\beta
g(\dot x,\dot x)\,\mathrm dt\right)^{-1}$ and hence the number of
such intervals $[\alpha,\beta]$ must be finite, which proves
part~\eqref{itm:numfinint} of the thesis.

Now, let us prove that if $[\alpha,\beta]\subset[a,b]$ is such
that $x([\alpha,\beta])\subset\partial\Omega$, then $x$ is
constant on $[\alpha,\beta]$.

The set of restrictions to
$[\alpha,\beta]$ of vector fields in $\mathcal V^+(x)$ clearly
contains (properly) the Hilbert subspace of
$H^1_{0}\big([\alpha,\beta],\R^N\big)$ consisting of those $V$
such that $V(s)\in T_{x(s)}\partial\Omega$ for any $s \in [\alpha,\beta]$. Variational
criticality in this case implies that
\[
\int_\alpha^\beta g(\dot x, \Dds V)\,\mathrm ds\ge 0
\quad\text{and}\quad \int_\alpha^\beta g(\dot x, -\Dds V)\,\mathrm
ds\ge 0
\]
for any $V$ such that $V(s)\in T_{x(s)}\big(\partial\Omega\big)$
for all $s\in[\alpha,\beta]$, and satisfying $V(\alpha)=V(\beta)=0$.
Then $x_{\vert[\alpha,\beta]}$ is a geodesic in $\partial\Omega$
with respect to $g$, so there exists $\lambda\in\mathcal
C^0([\alpha,\beta],\R)$ such that
\begin{equation}\label{eq:4.13bis}
\Dds\dot x(s)=\lambda(s)\nabla\phi(\gamma(s)),\qquad\forall
s\in[\alpha,\beta].
\end{equation}
Now consider the second derivative of $\phi(x(s))\equiv 0$ in $[\alpha,\beta]$. We obtain
\begin{multline*}
0=H^\phi(x(s))[\dot x(s),\dot x(s)]+g(\Dds \dot x(s),\nabla\phi(x(s)))=\\
=H^\phi(x(s))[\dot x(s),\dot
x(s)]+\lambda(s)\,g(\nabla\phi(x(s)),\nabla\phi(x(s))),
\end{multline*}
for all $s\in[\alpha,\beta]$. Suppose by contradiction that
$x_{\vert[\alpha,\beta]}$ is not constant. Then, by strong
concavity condition $H^\phi(x(s))[\dot x(s),\dot x(s)]<0$ so
\begin{equation}\label{eq:4.13ter}
\lambda(s)>0,\qquad\forall\, s\in[\alpha,\beta].
\end{equation}
But, choosing
\[
V(s)=
\begin{cases}
\sin\left(\frac{s-\alpha}{\beta-\alpha}\pi\right)\nabla\phi(x(s)) & \text{if\ } s\in[\alpha,\beta]\\
0 & \text{otherwise,}
\end{cases}
\]
then $V\in\mathcal V^+(x)$ and
\begin{multline*}
0\le \int_\alpha^\beta g(\dot x,\Dds V)\,\mathrm ds=-\int_\alpha^\beta g(\Dds x, V)\,\mathrm ds=\\
-\int_\alpha^\beta\lambda(s)\sin\left(\frac{s-\alpha}{\beta-\alpha}\pi\right)g(\nabla\phi(x(s))
,\nabla\phi(x(s)))\,\mathrm ds,
\end{multline*}
in contradiction with \eqref{eq:4.13ter}. Then
$x_{\vert[\alpha,\beta]}$ is constant and \eqref{itm:constconcomp}
is proven.

Since, as we have seen above, $x$ is a geodesic when it does not touch the boundary,
\eqref{itm:freegeo},
\eqref{itm:numfinint}
\eqref{itm:constconcomp}
imply
\eqref{itm:piecC2} and
\eqref{itm:portgeo}.

Finally, since $x_{|[a,b]}$ is not constant, it is not included in $\partial \Omega$, therefore by
\eqref{itm:freegeo} and \eqref{eq:1.1bis} we obtain \eqref{itm:addizionale}.
\end{proof}

We are now ready for the proof of the following:
\begin{prop}\label{thm:regcritpt}
Assume that there is no WOGC   in $\overline\Omega$. Let $x\in\Mcal$ and $[a,b] \in\mathcal I_x^0$
be such that:
\begin{itemize}
\item$x_{\vert[a,b]}$ is a variationally
critical portion of $x$,
\item
the restriction of $x$ to
$[a,b]$ is of class $C^1$.
\end{itemize}
Then, $x_{\vert[a,b]}$ is a orthogonal
geodesic chord in $\overline\Omega$ (with
$x\big(\left]a,b\right[\big)\subset \Omega$).
\end{prop}
\begin{proof}
$C^1$--regularity, \eqref{itm:numfinint} and
\eqref{itm:constconcomp} of Lemma \ref{thm:leminsez4.4} shows that
$x^{-1}(\partial\Omega)\cap[a,b]$ consists only of a finite number
of isolated points. Then, by the $C^1$--regularity on
$[a,b]$ and parts \eqref{itm:piecC2}--\eqref{itm:portgeo} of Lemma
\ref{thm:leminsez4.4}, $x$ is a geodesic on the whole interval
$[a,b]$. Moreover an integration by parts shows that $\dot x(a)$
and $\dot x(b)$ are orthogonal to $T_{x(a)}\partial\Omega$ and
$T_{x(b)}\partial\Omega$ respectively. Finally, since there are
not any  WOGC on $\overline\Omega$, $x_{\vert[a,b]}$ is an OGC.
\end{proof}

Variationally critical portions $x_{\vert[a,b]}$ of class
$C^1$ will be called {\em regular variationally critical
portions}; those critical portions that do not belong to this
class will be called {\em irregular}. Irregular variationally
critical portions of curves $x\in\Mcal$ are further divided into
two subclasses, described below.

Assume that there are not
WOGC's in $\overline\Omega$.
\begin{prop}\label{thm:irregcritpt}
Let $x\in\Mcal$ and let $[a,b]\in \mathcal I_x^0$ be such that $x_{\vert[a,b]}$ is an irregular variationally
critical portion of $x$.

Then, there exists a subinterval $[\alpha,\beta]\subset
[a,b]$ such that $x_{\vert[a,\alpha]}$ and $x_{\vert[\beta,b]}$
are constant (in $\partial\Omega$), $\dot x(\alpha^+)\in
T_{x(\alpha)}(\partial\Omega)^\perp$, $\dot x(\beta^-)\in
T_{x(\beta)}(\partial\Omega)^\perp$, and one of the two mutually
exclusive situations occur:
\begin{enumerate}
\item \label{itm:cusp} there exists a finite number of intervals
$[t_1,t_2]\subset\left]\alpha,\beta\right[$ such that
$x\big([t_1,t_2]\big)\subset\partial\Omega$ and that are maximal
with respect to this property; moreover, $x$ is constant on each
such interval $[t_1,t_2]$, and $\dot x(t_1^-)\ne\dot x(t_2^+)$;\footnote{%
When $t_1<t_2$ an easy partial integration argument shows that both $\dot x(t_1^-)$
and $\dot x(t_2^+)$ are orthogonal to $\partial\Omega$.}
\item \label{itm:stop} $x_{\vert[\alpha,\beta]}$ is an OGC in
$\overline\Omega$.
\end{enumerate}
\end{prop}
(Note that in the above statement it may happen $\alpha = a$ or
$\beta = b$.)
\begin{proof}
Totally analogous to the proof of Lemma~\ref{thm:leminsez4.4} and
Proposition~\ref{thm:regcritpt}.
\end{proof}

Irregular variationally critical portions in the class described
in part  \eqref{itm:cusp} will be called {\em of first type},
those described in part \eqref{itm:stop} will be called {\em of
second type}. An interval $[t_1,t_2]$ as in part \eqref{itm:cusp}
will be called a {\em cusp interval\/} of the irregular variational critical
portion $x_{|[a,b]}$ (see Figure~\ref{fig:irreg})
also in the degenerate case $t_1=t_2$.

\begin{figure}
\begin{center}
\psfull \epsfig{file=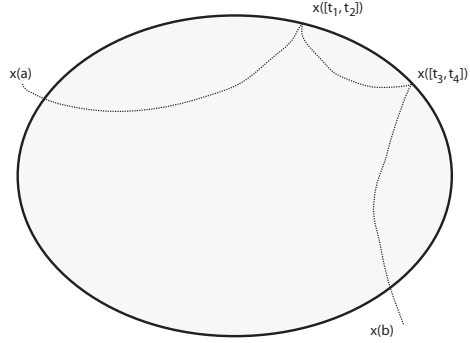, height=4.5cm}
\caption{Irregular
critical portions of curves in $\Mcal$, with cusp intervals
$[t_1,t_2]$ and $[t_3,t_4]$.
} \label{fig:irreg}
\end{center}
\end{figure}

\begin{rem}\label{rem:rem4.10}
We observe here that, due to the strict concavity assumption, if
$x_{|[a,b]}$ is an irregular variationally critical portion of
first  type and $[t_1,t_2],[s_1,s_2]$ are cusp intervals for $x$
contained in $[a,b]$ with $t_2<s_1$,    then
\[\text{there exists}\ s_0\in\left]t_2,
s_1\right[\text{\ with\ }\phi(x(s_0))< -\delta_0,\] (see
(\ref{eq:1.1bis})). This implies that the number of cusp intervals
of irregular variationally critical portions $x_{\vert[a,b]}$ is
uniformly bounded (cf. Lemma \ref{thm:lem2.3} and recall that $x \in
\mathfrak M$ and $[a,b] \in \mathcal I_x$ implies
$\frac12\int_a^bg(\dot x,\dot x)\,\mathrm dt< M_0$).

If $x_{\vert[a,b]}$ is a irregular critical portion of the first
type, and if $[t_1,t_2]$ is a cusp interval of $x_{\vert[a,b]}$, we will set
\begin{equation}\label{eq:4.15}
\Theta_x(t_1,t_2) = \text{the (non oriented) angle between the
vectors\ } \dot x(t_1^-) \text{\ and\ } \dot x(t_2^+).
\end{equation}
Observe that $\Theta_x(t_1,t_2)\in\left]0,\pi\right]$, (since
$x([a,b]) \subset \overline\Omega)$.
\end{rem}

\begin{rem}\label{thm:rem4.11bis}
We observe that if $[t_1,t_2]\subset[a,b]\in\mathcal I_x^0$ is a
cusp interval of $x_{\vert[a,b]}$, then the tangential components of $\dot
x(t_1^-)$ and of $\dot x(t_2^+)$ along $\partial\Omega$ are equal;
this is is easily obtained by an integration by parts. It follows
that if $\Theta_x(t_1,t_2)>0$, then $\dot x(t_1^-)$ and $\dot
x(t_2^+)$ cannot be both tangent to $\partial\Omega$.
\end{rem}

We will denote by $\mathcal Z$ the set of all curves having variationally critical portions:
\[\mathcal Z=\big\{x\in\Mcal:\exists\,[a,b]\subset[0,1]\ \text{such that
$x_{\vert[a,b]}$ is a variationally critical portion of $x$}
\big\}.\]
The following compactness property holds for $\mathcal
Z$:
\begin{prop}\label{thm:Zrcompatto}
Let $x_n$ be a sequence in $\mathcal Z$, $[a_n,b_n]\in \mathcal
I_{x_n}^0$ for any $n$ such that $x_{n \vert[a_n,b_n]}$ is a (non
constant) variationally critical portion of $x_n$. Then, up to
subsequences, $a_n$ converges to some $a$, $b_n$ converges to some
$b$, with $0\le a<b\le1$, and the sequence
$x_n:[a_n,b_n]\to\overline\Omega$ is $H^1$-convergent (in the sense
of Remark \ref{rem:rep}) to some curve $x:[a,b]\to\overline \Omega$
as $n\to\infty$, $[a,b] \in \mathcal I_x^0$, and $x_{\vert[a,b]}$ is
a variationally critical portion of $x$.
\end{prop}
\begin{proof}
By Lemma~\ref{thm:lem2.3} and \eqref{eq:1.1bis},  the sequence
$b_n-a_n$ is bounded away from $0$,
which implies the existence of subsequences converging in $[0,1]$
to $a$ and $b$ respectively, and with $a<b$.
If $x_n$ is a sequence of regular variationally critical portions,
then the conclusion follows easily observing that $x_n$,
and thus $\widehat x_n$ (its affine reparameterization in $[a,b]$) is a sequence of geodesics with image in a
compact set and having  bounded energy.

For the general case, one simply observes that the number of cusp
intervals of each $x_n$ is bounded uniformly in $n$, and the
argument above can be repeated by considering the restrictions of
$x_n$ to the complement of the union of all cusp intervals. Finally,
integrating by parts the term $\int_a^b g(\dot x,\Ddt V)\,\text dt$
one observes that it is non negative for all $ V\in\mathcal V^+(x)$,
hence $x_{\vert[a,b]}$ is a variationally critical portion of $x$
(note that the uniform convergency implies $[a,b] \in \mathcal I_x^0$).
\end{proof}

\begin{rem}\label{rem:rem4.12}
We point out  that the first part of the proof of Proposition
\ref{thm:Zrcompatto}  shows that if $x_n\in\mathcal Z$ ($\subset \mathfrak M$),
$[a_n,b_n]\in {\mathcal I_{x_n}^0}$ and $x_{n\vert[a_n,b_n]}$ is an
OGC, then, up to subsequences, there exists $[a,b]\subset[0,1]$ and
$x:[a,b]\to\overline\Omega$ such that $x_{n\vert[a_n,b_n]}\to
x_{\vert[a,b]}$ in $H^1$ and $x_{\vert[a,b]}$ is an OGC.
\end{rem}

Since we are assuming that there are not any WOGC in
$\overline\Omega$ by  Lemma \ref{thm:leminsez4.4}, Proposition
\ref{thm:regcritpt}, Proposition \ref{thm:irregcritpt} and
Proposition \ref{thm:Zrcompatto}, we immediately obtain the
following result.

\begin{cor}\label{thm:cor4.11bis}
There exists $d_0>0$ such that for any
$x_{|[a,b]}$ irregular variationally critical portion of first type with  $[a,b] \in \mathcal I_x^0$, there exists
a cusp interval $[t_1,t_2]\subset[a,b]$ for $x$ such that
\[
\Theta_x(t_1,t_2)\ge d_0.
\]
\end{cor}
The above corollary states that every irregular variational
critical portion of the first type has at least a cusp with
"amplitude" greater than or equal to a positive constant $d_0$.

\begin{rem}\label{rem:4.17}
For any $\delta\in\left]0,\delta_0\right]$ we have the following property: for any $x\in\mathfrak M$ and $[a,b] \in {\mathcal I}_x^0$
such that $x_{\vert[a,b]}$ is an irregular variationally critical
portion of first type, there exists an interval $[\alpha,\beta]\subset[a,b]$ and
a cusp interval  $[t_1,t_2]\subset[\alpha,\beta]$ such that:
\begin{equation}\label{eq:alphabeta}
\Theta_x(t_1,t_2)\ge d_0, \text{ and } \phi(x(\alpha))=\phi(x(\beta))=-\delta,
\end{equation}
where $d_0$ is given in Corollary \ref{thm:cor4.11bis}.

Note that $g\big(\nabla\phi(x(\alpha)),\dot x(\alpha)\big)>0$ and  $g\big(\nabla\phi(x(\beta)),\dot x(\beta)\big)<0$
by the strong concavity assumption.
\end{rem}

\section{Some results about critical portions with respect to ${\mathcal V}^{-}$}
\label{sub:V-criticalportions}

In this section we shall
discuss some results concerning the notion of
$\mathcal V^-$--variational critical portions. It is important to point out that,
to avoid curves having irregular variational critical portions of first type,
we will use integral flows for
fields in the cone $\mathcal V^-(x)$ defined in
\eqref{eq:defV-(x)}. This will be possible thanks to  the {\em
regularity\/} property described in  Lemma \ref{thm:lem5.3} below.

Moreover thanks to the assumption of strong concavity we shall obtain also Lemma \ref
{thm:4.10} that may have an important role to study the multiplicity problem.
\medskip

Let us denote
by $\mathcal A$ an arbitrary open subset of $M$ with regular boundary, and let
${\phi_{\mathcal A}:M\to\R}$ be a $C^2$-function such that
$\phi_{\mathcal A}(q)<0$ iff $q\in\mathcal A$, $\phi_{\mathcal
A \vert\partial\mathcal A}=0$ and $\mathrm \nabla\phi_{\mathcal
A}\ne0$ on $\partial\mathcal A$. (Note that in
Remark~\ref{thm:remphisecond}, $\phi=\phi_\Omega$). Set
\begin{equation}\label{eq:defVA-(x)}
\mathcal V^{-}_{\mathcal A}(x)=\big\{V\in T_x
H^1\big([0,1],\R^N\big):g\big(V(s), \nabla\phi_{\mathcal
A}\big(x(s)\big)\big)\le0\ \text{for $x(s)\in
\partial \mathcal A$}\big\}.
\end{equation}

\begin{lem}\label{thm:lem5.3}
Let $y\in H^1\big([a,b],\overline{\mathcal A}\big)$ be such that:
\begin{equation}\label{eq:eq5.6}
\int_a^bg\big(\dot y,\Ddt V)\,\mathrm
dt\ge0,\qquad\forall\,V\in\mathcal V^{-}_{\mathcal A}(y) \
\text{and}\ V(a)=V(b)=0.
\end{equation}
Then, $y\in H^{2,\infty}([a,b],\overline{\mathcal A})$ and in
particular it is of class $C^1$.
\end{lem}
\begin{proof}
We refer the reader e.g. to \cite[Lemma 3.2]{London} for the proof.
\end{proof}

\begin{defin}\label{def:4.9bis}
If $y_{|[a,b]}$ satisfies \eqref{eq:eq5.6}, it will be called
\emph{a ${\Vcal}_{\mathcal A}^-$--variationally critical portion.}
\end{defin}

\begin{rem}\label{thm:C1-reg}
It is important to observe that, while a variationally critical portion with respect to ${\mathcal V}^{-}$ is of class $C^1$, this is not the case for
variationally critical
portion with respect to ${\mathcal V}^{+}$. In particular, as pointed out in Proposition
\ref{thm:irregcritpt}, irregular variationally critical portions of first type are not of class $C^1$.
\end{rem}

Now, denote by $\nu(x)=\frac{\nabla{\phi_{\mathcal A}}(x)}{\Vert\nabla{\phi_{\mathcal A}}(x)\Vert}$ the outward pointing
unit normal vector field in $\partial{\mathcal A}$.
\begin{lem}\label{thm:lem4.16}
Let $y\in H^{1}\big([a,b],\overline{\mathcal A}\big)$ satisfying
\eqref{eq:eq5.6}. Then $g(\dot y,\dot y)$ is constant on $[a,b]$,
and
\begin{equation}\label{eq:8.1bis}
-\Dds\dot y(s)=\mu(s)\,\nu(y(s)) \text{ a.e.},\quad\text{ where\
}\mu(s)=
\begin{cases}
0,&\text{if\ }{\phi_{\mathcal A}}(y(s))<0,\\
\frac{H^{\phi_{\mathcal A}}(y)[\dot y,\dot
y]}{g(\nu(y),\nabla\phi_{\mathcal A}(y))},&\text{if\
}{\phi_{\mathcal A}}(y(s)=0.
\end{cases}
\end{equation}
Moreover
\begin{equation}\label{eq:segno-nu}
\mu(y(s)) \leq 0 \; a.e.
\end{equation}

\end{lem}
\begin{proof}
If $y_{\vert[a,b]}$ satisfies \eqref{eq:eq5.6}, by Lemma
\ref{thm:lem5.3}, $\Ddt\dot y\in L^\infty$, so we can integrate by
parts in \eqref{eq:eq5.6} obtaining
\[
\int_a^b-g\big(\Ddt\dot y,V\big)\,\text dt\ge 0,\qquad\text{for
all}\ V\in C^\infty_0([a,b],\R^N)
\]
satisfying \[g(V(s),\nu(y(s)))\le 0\ \text{for any $s$ such that
${\phi_{\mathcal A}}(y(s))=0$}.\]

Then $y$ is a geodesic where it does not intersect the boundary.
Moreover using vector fields $V$ such that $g(V(s),\nu(y(s)))=
0$ for all $s$ such that ${\phi_{\mathcal A}}(y(s))=0$ we obtain
\begin{equation}\label{eq:ast4}
-\Dds\dot y(s)=\mu(s)\nu(y(s)) \text{ a.e.},\qquad\text{for some\
} \mu:[a,b]\to\R.
\end{equation}
Since $\Ddt\dot y$ is in $L^\infty$, then $\mu\in
L^\infty\big([0,1],\R\big)$ (and $\mu(s)=0$ if ${\phi_{\mathcal
A}}(y(s))<0$).

Now $y$ is of class $C^1$ and ${\phi_{\mathcal A}}(y(s))\le 0\,\forall s\in[a,b]$. Therefore
\begin{equation}\label{eq:4.40bis}
g(\nabla{\phi_{\mathcal A}}(y(s)),\dot y(s))=0,\quad \text{for
every $s$ such that ${\phi_{\mathcal A}}(y(s))=0$},
\end{equation}
and, contracting both terms of \eqref{eq:ast4} with $\dot y$, we have
\[
g\big(\Ddt\dot y,\dot y\big)=0\qquad\text{a.e.\ in\ } [a,b].
\]
Then there exists $c_1>0$ such that
\[
g(\dot y,\dot y)=c_1\qquad\text{a.e.\ in\ } [a,b].
\]
But $y$ is of class $\mathcal C^1$ so
\[
g(\dot y,\dot y)=c_1,\qquad\forall\, s\in [a,b],
\]
Now contracting both members in \eqref{eq:ast4} with
$\nabla{\phi_{\mathcal A}}(y(s))$ we obtain
\begin{equation}\label{eq:astast}
-g\big(\Dds\dot y,\nabla{\phi_{\mathcal A}}(y)\big)=\mu(s)\,g\big(\nu(y),\nabla{\phi_{\mathcal A}}(y)),\quad\text{a.e.\ in $[a,b]$}.
\end{equation}
Set $C_y=\{s\,:\,{\phi_{\mathcal A}}(y(s))=0\}$ which is the set
where $\mu(s)$ can be nonzero. Using well known theorems in
Sobolev spaces (see \cite{GT}), by \eqref{eq:4.40bis} we have (differentiating in $C_y$)
\[
H^{\phi_{\mathcal A}}(y)[\dot y,\dot y]+g\big(\nabla{\phi_{\mathcal A}}(y),\Ddt\dot y)=0\quad\text{a.e.\ in\ }C_y.
\]
Then, by \eqref{eq:astast}
\[
\mu(s) g(\nu(y),\nabla{\phi_{\mathcal A}}(y))=H^{\phi_{\mathcal A}}(y)[\dot y,\dot y]\quad\text{a.e.\ in\ }C_y
\]
from which we deduce \eqref{eq:8.1bis}.

Finally to prove \eqref{eq:segno-nu} it is sufficient to apply formula \eqref{eq:eq5.6} with
\[
V(s) = -\sin\left(\frac{s-a}{b-a}\pi\right)\nabla{\phi_{\mathcal A}}(y(s)),
\]
since $y$ satisfies \eqref{eq:8.1bis}.
\end{proof}

\begin{rem}\label{rem:4.16bis}
Note that, under the assumption of strong concavity the set \[C_y=\big\{s\in[a,b]\,:\,\phi(y(s))=0\big\}\] consists
of a finite number of intervals. On each one of these, $y$ is of class $C^2$, and it satisfies the ``constrained geodesic'' differential equation
\begin{equation}\label{eq:4.30bis}
\Dds\dot y(s)=-\left[\frac1{g(\nu(y(s)),\nabla\phi(y(s)))}H^\phi(y(s))[\dot y(s),\dot y(s)\right]\nu(y(s)).
\end{equation}
\end{rem}

\begin{rem}\label{thm:remveripticritici} Now suppose that $y$ satisfies \eqref{eq:eq5.6} and $\overline{\mathcal A}$ is convex. In this case
\[
H^{\phi_{\mathcal A}}(y)[\dot y,\dot y]\geq 0 \text{ for any }s \text{ such that }{\phi_{\mathcal A}}(y(s))=0.
\]
Then from formulas~\eqref{eq:eq5.6} and \eqref{eq:8.1bis} we see that the function $\mu$
in \eqref{eq:8.1bis} is identically zero, and thus any
$\mathcal V^{-}_{\mathcal A}$--variationally critical portion  in
$\overline{\mathcal A}$ is a free geodesic in $\overline{\mathcal
A}$.
\end{rem}


Our main goal is to construct a functional for our minimax argument and a
family of flows in the path space on which the energy functional is non
increasing.
Moreover, to clarify another trait of the strong concavity condition, we wish to point out
the following property that may be useful also in view of obtaining multiplicity results:
for any curve $x\in \mathfrak M$ and for any interval $[a,b] \in \mathcal I_x$ where $x$ is uniformly
close to $\partial\Omega$, the curve $x_{|[a,b]}$ can be deformed in directions where
$\phi$ increases, without increasing the energy functional.
Such a property will be obtained thanks above all to Definition \ref{def:deltaintervallo} and Lemma \ref{thm:4.10} below.

\bigskip
\begin{defin}\label{def:deltaintervallo}
Fix $\delta \in ]0,\delta_0]$.  We say
that $[\alpha,\beta] \subset [0,1]$ \ is a $\delta$--interval
(for $x \in \mathfrak M$) if
\begin{equation}\label{eq:4.9bis}
\begin{split}
\phi(x(s) \geq -\delta \, \text{ for all }s \in [\alpha,\beta], \,\,
\phi(x(\alpha)) = \phi(x(\beta)) = 0,\\
\text{ and } \inf \{\phi(x(s)): s \in [\alpha,\beta]\} =
-\delta.
\end{split}
\end{equation}
We say that a $\delta$--interval $[\alpha,\beta]$ is minimal if it is minimal with respect
to the property above.
\end{defin}

\bigskip

\begin{lem}\label{thm:4.10}
For any $\delta \in ]0,\delta_0]$, for any $x\in \mathfrak M$, for any $[\alpha,\beta]$ $\delta$--interval for $x$, there exists
$V\in H_0^1([\alpha,\beta],\mathbb{R}^N)$ such that
\begin{equation}\label{itm:d-int2}
g(\nabla\phi(x(s)),V(s))\ge 0 \,\, \forall s\in[\alpha,\beta] \text{ satisfying }
\phi(x(s)) = -\delta \text{ and }\phi(x(s)) = 0;
\end{equation}
and
\begin{equation}\label{itm:d-int3}
\int_{\alpha}^{\beta} g(\dot x,\Dds V)\,\mathrm
ds < 0.
\end{equation}
\end{lem}

\begin{proof}
Clearly we can assume that the $\delta$--interval $[\alpha,\beta]$ is minimal with respect to \eqref{eq:4.9bis}. Suppose by contradiction that
for any
$V\in H_0^1([\alpha,\beta],\mathbb{R}^N)$, \eqref{itm:d-int2} and \eqref{itm:d-int3} are not satisfied at the same time. Then by the minimimality property of $[\alpha,\beta]$
we see that
\[
\int_{\alpha}^{\beta}g\big(\dot x,\Dds V\big),\mathrm ds\ge 0
\]
for any
$V\in H_0^1([\alpha,\beta],\mathbb{R}^N)$ satisfying
$g(\nabla\phi(x(s)),V(s))\ge 0$ for any $s\in[\alpha,\beta]$ such that
$\phi(x(s)) = -\delta$.

Then the curve $x_{\vert[\alpha,\beta]}$
is a $\mathcal V^{-}_{\mathcal A}$--variationally critical portion
in $\overline{\mathcal A}=\phi^{-1}\big(\left[-\delta,+\infty\right[\big)$.
Since $\delta \in ]0,\delta_0]$, by Remark \ref{rem:1.4}
$\overline{\mathcal A}$ is convex. Then from Remark~\ref{thm:remveripticritici}
  $x_{\vert[\alpha,\beta]}$ is a free geodesic.
Since it has endpoints on $\phi^{-1}(0)$, $0 < \delta \leq \delta_0$
and $\inf\limits_{[\alpha, \beta]}\phi(x(s))=-\delta$, we have a contradiction
with the strong concavity assumption on $\overline\Omega$ (cf. Remark \ref{rem:1.4}).
This concludes the proof.
\end{proof}

For the Palais-Smale condition described in Proposition \ref{thm:51new} below will be useful also the following notion.
\begin{defin}\label{def:delta-d-intervallo}
Fix $\delta \in ]0,\delta_0]$ and $\mathfrak d \in [0,\delta[$. We say
that $[\alpha,\beta] \subset [0,1]$ \ is a $(\delta,\mathfrak d)$--interval
(for $x \in \mathfrak M$) if
\begin{equation}\label{eq:4.9bis2}
\begin{split}
\phi(x(s) \geq -\delta \, \forall s \in [\alpha,\beta], \,\,
\phi(x(\alpha)) = \phi(x(\beta)) = -\mathfrak d,\\
\text{ and } \inf \{\phi(x(s)): s \in [\alpha,\beta]\} =
-\delta.
\end{split}
\end{equation}
\end{defin}

\section{The Palais--Smale condition for ${\mathcal V}^{+}$--variationally critical portion}
\label{sec:PS}
Let $\delta_0$ be as in Remark \ref{rem:1.4}.
We will now proceed with the statement and the proof of the analogue of the
classical Palais--Smale condition which is suited to our variational
framework; unfortunately, the technical nature of the following Propositions
cannot be avoided.
Their statements contains several different facts that should
be familiar to specialists in pseudo-gradient techniques.
Recall that when proving the Palais--Smale condition for $C^1$ functionals
on Banach manifolds, one first needs to show that at each non critical point
there must exist a direction along which the functional decreases.
Secondly, one needs to show that the  descent direction for the functional at
a given point can be chosen in such a way to be a decreasing direction
also in a neighborhood of this point; along such direction the
decreasing rate of the functional should be
bounded away from zero uniformly outside a neighborhood of the set of critical
points. In our context, a further complication is given by the fact that
our notion of criticality is defined for portions of the curve $x$, so that,
to single out
a decreasing direction  at each non critical path,
we need the possibility to patch together
vector fields defined on each non critical portion of $x$.
All these facts are condensed into Propositions \ref{thm:51new} and \ref{thm:PS} below;
for a better
understanding of the statement, it may be useful to keep in mind
the following description of the quantities involved in the hypotheses and the theses:
\begin{itemize}
\item the positive number $r$ is a measure of distance from variationally
critical portions;
\item $\theta(r)$ is the rate of increase of the function $\phi$ along the flow
at those points of the curve that are in $\phi^{-1}([-\kappa(r),\kappa(r)])$; and also at that points in
$\phi^{-1}(-\delta)$ and $\phi^{-1}(-\mathfrak d)$  corresponding to instants included in
$(\delta,\mathfrak d)$--intervals with  $\delta \in ]\kappa(r),\delta_0]$ and $\mathfrak d \in [0,\kappa(r)]$
 (cf.  \eqref{itm:51new-bis}
in the statement of Proposition \ref{thm:51new}).

\item $\mu(r)$ is the rate of decrease of the functional $\mathcal F$, that will be defined
formally in Section~\ref{sec:homotopies}.
\end{itemize}

\bigskip
For $[a,b]\subset[0,1]$, we denote by $\mathcal Z_{a,b}$ the set
of curves $w:[a,b]\to\overline\Omega$ that are  variationally critical portions:
\begin{multline}\label{eq:defZab}
\mathcal Z_{a,b}=\Big\{w\in H^{1}\big([a,b],\overline\Omega\big): w(a),w(b)\in\partial\Omega,\\
 \int_a^bg\big(\dot w,\Ddt V\big)\,\mathrm dt\ge0\
\text{for all $V\in\mathcal V^+(w)$}\Big\},
\end{multline}
where $\mathcal V^+$ is defined in \eqref{eq:defV+(x)}.

We shall need the following results:

\begin{prop}\label{thm:51new}
For all $r>0$, there exist positive numbers $\theta(r), \mu(r), \kappa(r)$ ($\kappa(r) < \delta_0$) satisfying the following
property: for all $x\in\mathfrak M_0$ and for all $[a,b] \subset [0,1]$  for
which
\begin{equation}\label{inOmega}
x(a),x(b) \in \partial \Omega \text{ and } x([a,b]) \subset
\overline \Omega
\end{equation}
(namely $[a,b] \in \mathcal I_x^0$),
\begin{equation}\label{eq:4.22f}
\frac14(\frac{3\delta_0}{4K_{0}})^{2} \leq \frac{b-a}{2}\int_a^bg(\dot x,\dot x)\,\mathrm ds, \quad
\frac12\int_a^bg(\dot x,\dot x)\,\mathrm ds\le M_0,
\end{equation}
and such that
\begin{equation}\label{eq:4.23f}
\big\Vert x_{|[a,b]}-y \big\Vert_{a,b}\ge r,\quad
\forall\,y\in\mathcal Z_{a,b},
\end{equation}
there exists a vector field  $V_x\in H^{1}\big([a,b],\R^N\big)$
such that the following conditions hold:
\begin{enumerate}
\item\label{itm:51new-bis} for any $\delta \in ]\kappa(r),\delta_0]$, for any $\mathfrak d \in [0,\kappa(r)]$, for any
$(\delta,\mathfrak d$)--interval $[c,d]\subset[a,b]$,
and for all $s\in[c,d]$ such that
$\phi(x(s) = -\delta$ or $\phi(x(s) = -\mathfrak d$,
it is $g(\nabla\phi(x(s)),V_x(s))\ge\theta(r)\Vert V_x\Vert_{a,b}$;
\smallskip
\item\label{itm:51new-a} $g\big(\nabla\phi(x(s)),V_x(s)\big)
\ge\theta(r)\Vert V_x\Vert_{a,b}$
for all $s\in [a,b]$ that satisfies $\phi(x(s)) \in [-\kappa(r),\kappa(r)]$;
\smallskip
\item\label{itm:51new-b}
$\int_{a}^{b}g\big(\dot x,\Dds V_x\big)\,\mathrm ds\le-\mu(r)\Vert V_x\Vert_{a,b}$.
\end{enumerate}
\end{prop}

\begin{rem} About assumption \eqref{eq:4.22f} we wish to point out that we shall look for OGC's with energy integral
$\geq \frac12(\frac{3\delta_0}{4K_{0}})^{2}$. On the other hand if $\frac{b-a}{2}\int_a^bg(\dot x,\dot x)\,\mathrm ds \leq
 \frac12(\frac{3\delta_0}{4K_{0}})^{2}$
and $x(a) \in \partial \Omega$, by Lemma \ref{thm:lem2.3} one deduce that $x([a,b]) \subset \{\phi\geq-\frac{3\delta_0}4$\} so, if also $x(b)\in
\partial \Omega$ $x_{|[a,b]}$ can not be a geodesic because strong concavity assumption.
\end{rem}

\begin{proof}[Proof of Proposition \ref{thm:51new}]
We argue by contradiction, assuming the existence of sequences
$\theta_n\to 0^+$, $\mu_n\to 0^+$, $\kappa_n\to 0^+$, $x_n\in \mathfrak M_0$,
$[a_n,b_n] \subset [0,1]$, satisfying

\begin{equation}\label{inOmega-n}
x_{n}(a_n),x_{n}(b_n) \in \partial \Omega \text{ and } x_{n}([a_n,b_n]) \subset \overline \Omega,
\quad\text{(i.e., $[a_n,b_n]\in\mathcal I^0_{x_n}$)}
\end{equation}

\begin{equation}\label{eq:5.5new}
\frac14(\frac{3\delta_0}{4K_{0}})^{2} \leq \frac{b_n-a_n}{2}\int_{a_n}^{b_n}g(\dot x_n,\dot x_n)\,\mathrm ds, \quad
\frac12\int_{a_n}^{b_n}g(\dot x_n,\dot x_n)\,\mathrm ds\le M_0
\end{equation}
and
\begin{equation}\label{eq:5.6new}
\big\Vert x_{n\vert[a_n,b_n]}-y \big\Vert_{a_n,b_n}\ge r,\quad
\forall\,y\in\mathcal Z_{a_n,b_n},
\end{equation}
and
such that, for any $V\in H^{1}([a_n,b_n],\R^N)$, the following properties
cannot hold at the same time:

\begin{enumerate}
\item[(C1)]\label{itm:51proof-bis}
for any $\delta \in ]\kappa_n,\delta_0]$, for any $\mathfrak d \in [0,\kappa_n]$,  for any  $(\delta,\mathfrak d)$--interval
$[c,d]\subset[a_n,b_n]$,
for all $s\in[c,d]$ such that
$\phi(x_n(s)) = -\delta$ or $\phi(x_n(s)) = -\mathfrak d$, it is
\[g(\nabla\phi(x_n(s)),V(s))\ge\theta_n\Vert V\Vert_{a_n,b_n};\]

\item[(C2)]\label{itm:51proof-a} $g\big(\nabla\phi(x_n(s)),V(s)\big)
\ge\theta_n\Vert V\Vert_{a_n,b_n}$ for all $s\in [a_n,b_n]$ that satisfies
$\phi(x_n(s)) \in [-\kappa_n,\kappa_n]$;

\item[(C3)]\label{itm:51proof-b}
$\int_{a_n}^{b_n}g\big(\dot x_{n},\Dds V\big)\,\mathrm ds\le-\mu_n\Vert V\Vert_{a_n,b_n}$.
\end{enumerate}

Note that, by \eqref{eq:5.5new},
$$ b_n-a_n\ge\frac{1}{4M_0}(\frac{3\delta_0}{4K_{0}})^{2}. $$
Then, up to subsequences, we can assume the existence of the limits
\[ 0\le\widehat a=\lim_{n\to\infty}a_n<\widehat b=\lim_{n\to\infty}b_n\le 1,\]
and up to affine reparameterizations we can assume $a_n=\widehat a,\,b_n=\widehat b$ for all $n$.

Moreover, up to subsequences again, we can assume that the
reparameterization $x_{n\vert[\widehat a,\widehat b]}$
is weakly $H^1$--convergent (and therefore  uniformly)
to some curve $x\in H^1([\widehat a,\widehat b],\R^N)$,
because $x_n$ satisfies \eqref{eq:5.5new} and $x_n(\widehat a)\in\partial\Omega$ is bounded.

To carry out the proof we shall show that
\begin{equation}\label{eq:5.7new}
x_{\vert[\widehat a,\widehat b]}\in\mathcal Z_{\widehat a,\widehat b},
\end{equation}
and
\begin{equation}\label{eq:5.8new}
x_{n\vert[\widehat a,\widehat b]}\to x_{\vert[\widehat a,\widehat b]}\quad\text{strongly in $H^1$,}
\end{equation}
obtaining a contradiction with \eqref{eq:5.6new}.

To prove \eqref{eq:5.7new}, suppose that there exists  $\delta \in ]0,\delta_0]$, and
a minimal $\delta$--interval for $x$,  $[c,d] \subset
[\widehat a,\widehat b]$.
Take $V\in\mathcal V^+(x)$ such that $V=0$ outside $[c,d]$ and such that
\begin{equation}\label{51proof-further-bis}
g\big(\nabla\phi(x(s)),V(s)\big) \geq 0
\text{ for all } s\in [c,d] : \phi(x(s)) = -\delta.
\end{equation}

Thanks to the minimality property of $[c,d]$, we can take
$V_n=V+\lambda_n\nabla\phi(x_n)$ with $\lambda_n\to 0^+$ and
such that $V_n$ satisfies (C1) and (C2)
above
with $V$ replaced by
$V_n$.
This can be proved since $x_{n\vert[\widehat a,\widehat b]}$
is included in $\overline \Omega$,
it is bounded in $H^1$, $x_n$ uniformly converges to $x$, $\theta_n$ and $\kappa_n$ are infinitesimals,
and $\nabla\phi$ does not vanish on $\phi^{-1}([-\delta_0,\delta_0])$.

Then property (C3)
cannot hold, so
\begin{equation}\label{eq:5.10new}
\int_{\widehat a}^{\widehat b}g\big(\dot x_n,\Dds V_n\big)\,\mathrm ds>-\mu_n\Vert V_n\Vert_{\widehat a,\widehat b}.
\end{equation}
Since $\Vert V_n\Vert_{\widehat a,\widehat b}$ is bounded, \eqref{eq:5.10new} gives
\[
\liminf_{n\to\infty}\int_{\widehat a}^{\widehat b} g\big(\dot x_n,\Dds V_n\big)\,\mathrm ds\ge 0.
\]
Since $\lambda_n\to 0$ and $x_n$ is bounded in $H^1$, $V_n$ is strongly $H^{1}$--convergent to $V$. Moreover
$\dot x_n$ is weakly $L^2$--convergent to $\dot x$ on $[\widehat a,\widehat b]$, so we have
\begin{equation}\label{eq:5.critport}
0\le\lim_{n\to\infty}\int_{\widehat a}^{\widehat b}g\big(\dot x_n,\Dds V_n\big)\,\mathrm ds=
\int_{\widehat a}^{\widehat b} g\big(\dot x,\Dds V\big)\,\mathrm ds=\int_{c}^{d} g\big(\dot x,\Dds V\big)\,\mathrm ds,
\end{equation}
because $V=0$ outside $[c,d]$.
Since $V \in \mathcal V^{+}(x)$ satisfying \eqref{51proof-further-bis} is arbitrary, by Lemma \ref{thm:4.10},
we obtain a contradiction. Then there are not $\delta$--intervals for $x$ in $[\widehat a, \widehat b]$ and we can therefore choose any $V \in \mathcal V^{+}(x)$ to construct $V_n$ as above proving that
$x_{\vert [\widehat a, \widehat b]}$
satisfies \eqref{eq:5.critport} for any $V \in \mathcal V^{+}(x)$. Then $x_{\vert [\widehat a, \widehat b]}$ satisfies \eqref{eq:5.7new}.

\bigskip
To prove \eqref{eq:5.8new}, let us choose a vector field $V_n$ along
$x_{n\vert[\widehat a,\widehat b]}$ of the form
\[ V_n=-\gamma_n+\lambda_n\nabla\phi(x_n),\]
where $\gamma_n=x_n-x$, and $\lambda_n$
can be chosen such that
$\lambda_n\to 0^+$ and $V_n$ satisfies properties (C1) and
(C2)
above with $V$ replaced by $V_n$. Then
property (C3)
cannot hold. Therefore it must be
\[
\liminf_{n\to\infty}\int_{\widehat a}^{\widehat b}g\big(\dot x_n,\Dds V_n\big)\,
\mathrm ds\ge 0.
\]
Since $\lambda_n\,\nabla\phi(x_n)$ is $H^1$--convergent to 0, and
$V_n = -\gamma_n + \lambda_n\nabla\phi(x_n)$,
\[
\liminf_{n\to\infty}\int_{\widehat a}^{\widehat b}
-g\big(\dot x_n,\dot \gamma_n\big)\,\mathrm ds\ge 0,
\]
that is
\[
\limsup_{n\to\infty}\int_{\widehat a}^{\widehat b}g\big(\dot x_n,\dot x_n-\dot x\big)\,\mathrm ds\le 0.
\]
Moreover, by $L^2$--weak convergence of $\dot x_n$ to $\dot x$, we have
\[
\lim_{n\to\infty}\int_{\widehat a}^{\widehat b}g\big(\dot x,\dot x_n-\dot x\big)\,\mathrm ds= 0,
\]
and from the last two relations above it follows that
\[
\limsup_{n\to\infty}\int_{\widehat a}^{\widehat b}g\big(\dot x_n-\dot x,\dot x_n-\dot x\big)\,\mathrm ds\le 0,
\]
which proves \eqref{eq:5.8new}, and the proof is complete.
\end{proof}

\bigskip

Now we need to extend the above
descent
direction, obtained on a portion of the curve $x_{|[a,b]}$,
to a neighborhood of it. Towards this goal some preliminary notations and results are needed.

First observe that,
due to the compactness of $\phi^{-1}(]-\infty,\delta_0])$ and the form of Christoffel symbols, there exists
two positive constants
$\ell_0$, $L_0$ such that, denoting by $\Vert\cdot\Vert_E$ the Euclidean norm and by $\Vert\cdot\Vert$ the $g$--norm,
\begin{equation}\label{eq:5.10bisnew}
\ell_0\Vert v\Vert_E^2\le \Vert v\Vert^2\le L_0\Vert v \Vert_E^2,
\end{equation}
for any $v \in \R^N$, for any $x \in \phi^{-1}(]-\infty,\delta_0])$.
Moreover
denoting by $g_x$ the metric tensor $g$ evaluated on $T_{x}M$,
there exists a constant
$G_0>0$  such that
\begin{equation}\label{eq:5.11new}
\vert g_x(v_1,v)-g_z(v_2,v)\vert\le G_0\left(\Vert v_1-v_2\Vert_E\,\Vert v\Vert_E+\Vert x-z\Vert_E\,\Vert v_1\Vert_E\,\Vert v\Vert_E\right),
\end{equation}
for any $x,z\in\phi^{-1}(]-\infty,\delta_0])$ and for any $v_1$, $v_2$, $v\in\R^N$. Finally we see also that there exists $L_1=L_1(M_0)$ such that
\begin{multline}\label{eq:5.12new}
\left(\int_a^b \Vert\Dds V\Vert^2_E,\mathrm ds\right)^{1/2}\le L_1\Vert V\Vert_{a,b},\\
\text{ for any $x\in\mathfrak M_0$ and $[a,b]\subset [0,1]$ such that
$\tfrac12\int_a^b g(\dot x,\dot x)\,\mathrm ds\le M_0$, }\\
\text { and for any
vector field $V\in H^{1}([a,b],\R^N)$ along $x$.}
\end{multline}

\medskip

Now, for any $a,b\in[0,1]$,  denote by $I_{a,b}$ the interval
$[a,b]$ (possibly reduced to a single point) if $b\ge a$
and the interval $[b,a]$ if $b<a$.
For any intervals $[a,b]$ and $[\alpha,\beta]$ we set:
\begin{equation}\label{eq:defDxalfabeta}
\mathcal
D(x,\alpha,\beta,a,b)=\frac{1}2\int_{I_{a,\alpha}\cup I_{b,\beta}}g(\dot
x,\dot x)\,\mathrm dt.
\end{equation}
The
following preparatory lemma holds.

\begin{lem}\label{thm:lem5.2new}
Fix $K>0$. For any $x,z\in\mathfrak M_0$, \, $\forall [a,b]$ and
$[a_z,b_z]$ in $[0,1]$, \, $\forall V\in H^1([0,1],\R^N)$, if
$\tfrac12\int_a^b g(\dot x,\dot x)\,\mathrm ds\le M_0$,
$\Dcal(x,a_z,b_z,a,b)\le K$ and $[a_z,b_z]\cap[a,b] \neq \emptyset$, it is
\begin{multline*}
\left\vert\int_a^b g_x\big(\dot x,\Dds V\big)\,\mathrm ds-\int_{a_z}^{b_z} g_z\big(\dot z,\Dds V)\,\mathrm ds\right\vert\le\\
{\sqrt2}\left(\sqrt{L_0 K}+G_0\Vert x-z\Vert_{a_z,b_z}\left(1 +\sqrt{\frac{M_0+K}{\ell_0}}\right)\right)L_1\Vert V\Vert_{0,1},
\end{multline*}
\end{lem}

\begin{proof}
We have:
\begin{multline*}
\Big\vert\int_a^b g_x\big(\dot x,\Dds V\big)\,\mathrm ds -\int_{a_z}^{b_z} g_z\big(\dot z,\Dds V\big)\,\mathrm ds\Big\vert\le\\
\Big\vert\int_a^b g_x\big(\dot x,\Dds V\big)\,\mathrm ds -\int_{a_z}^{b_z} g_x\big(\dot x,\Dds V\big)\,\mathrm ds\Big\vert+
\Big\vert\int_{a_z}^{b_z} g_x\big(\dot x,\Dds V\big)\,\mathrm ds -\int_{a_z}^{b_z} g_z\big(\dot z,\Dds V\big)\,\mathrm ds\Big\vert\\ \le
\int_{I_{a,a_z}\cup I_{b,b_z}}\Vert\dot x\Vert\,\Vert\Dds V\Vert\,\mathrm ds+\\
G_0\left[\int_{a_z}^{b_z}\Vert\dot x-\dot z\Vert_E\,\Vert\Dds V\Vert_E\,\mathrm ds+\int_{a_z}^{b_z}\Vert x-z\Vert_E\,\Vert\dot x\Vert_E\,\Vert\Dds V\Vert_E\,\mathrm ds\right]\le\\
\left(\int_{I_{a,a_z}\cup I_{b,b_z}}g(\dot x,\dot x)\,\mathrm ds\right)^{1/2}\left(\int_0^1\Vert\Dds V\Vert_E^2\,\mathrm ds\right)^{1/2}\sqrt{L_0}+\\
G_0\Vert\dot x-\dot z\Vert_{L^2([a_z,b_z],\R^N)}\left(\int_0^1\Vert\Dds V\Vert_E^2\right)^{1/2}+\\
G_0\Vert x-z\Vert_{L^\infty([a_z,b_z],\R^N)}\left(\int_{a_z}^{b_z}\Vert\dot x\Vert^2_E\,\mathrm ds\right)^{1/2}
\left(\int_0^1\Vert\Dds V\Vert_E^2\right)^{1/2}.
\end{multline*}
Second inequality above is due to Schwarz inequality,
the definition of $I_{a,\alpha}$, and \eqref{eq:5.11new}  ,
while the third is given by H\"{o}lder inequality and \eqref{eq:5.10bisnew}.
Now by \eqref{eq:1.8fabio} we have
\[
\Vert\dot x-\dot z\Vert_{L^2([a_z,b_z],\R^N)}\leq{\sqrt 2}\Vert x-z\Vert_{a_z,b_z},\quad
\Vert x-z\Vert_{L^\infty([a_z,b_z],\R^N)}\le\Vert x-z\Vert_{a_z,b_z}.
\]
Therefore by \eqref{eq:defDxalfabeta}, assumption $\Dcal(x,a_z,b_z,a,b) \leq K$ and
\eqref{eq:5.12new} it is :
\begin{multline*}
\left\vert\int_a^b g_x\big(\dot x,\Dds V\big)\,\mathrm ds-\int_{a_z}^{b_z} g_z\big(\dot z,\Dds V)\,\mathrm
ds\right\vert\le\\
\left(\sqrt{2L_0 K}+G_0\Vert x-z\Vert_{a_z,b_z}\left(\sqrt 2+\left(\int_{a_z}^{b_z}\Vert\dot x\Vert_E\,\mathrm ds\right)^{1/2}\right)\right)L_1\Vert V\Vert_{0,1},
\end{multline*}
and finally, since
\begin{align*}
&\left(\int_{a_z}^{b_z}\Vert\dot x\Vert^2_E\,\mathrm ds\right)^{1/2}\le\tfrac{1}{\sqrt{\ell_0}}\left(\int_{a_z}^{b_z}g(\dot x,\dot x)\,\mathrm ds\right)^{1/2},\\
&\left\vert\tfrac12\int_{a_z}^{b_z}g(\dot x,\dot x)\,\mathrm ds-\tfrac12\int_a^b g(\dot x,\dot x)
\,\mathrm ds\right\vert\le\tfrac12\int_{I_{a,a_z}\cup I_{b,b_z}}g(\dot x,\dot x)\,\mathrm ds\le K,\\
\intertext{and}
&\tfrac12\int_a^b g(\dot x,\dot x)\,\mathrm ds\le M_0,
\end{align*}
we obtain the proof.
\end{proof}

Let now $\mu(r)$ be as in Proposition \ref{thm:51new}, and $L_0$ and $L_1$ be the constants given by \eqref{eq:5.10bisnew} and
\eqref{eq:5.12new} respectively. Set

\begin{equation}\label{eq:5.14bis}
E(r)=\frac{\mu^2(r)}{32 L_1^2 L_0}.
\end{equation}
\begin{prop}\label{thm:PS}
Fix $r>0$. Let $\theta(r)$ and $\mu(r)$ be as in Proposition
\ref{thm:51new}. Then there exist $\rho(r) >0$ with the following property.

For all $x\in\mathfrak M_0$ and for all  $[a,b]
\subset [0,1]$ satisfying \eqref{inOmega} \eqref{eq:4.22f} and
\eqref{eq:4.23f}, let $V_x$ be as in  Proposition \ref{thm:51new}.
Extend it to $[0,1]$ so that $V_x(s)=V(x(a))$ on $[0,a]$ and $V_x(s)=V(x(b))$ on $[b,1]$.
Then there exists $\Delta(x) = \Delta(x,[a,b]) > 0$ such that:
\begin{enumerate}

\item\label{itm:PS4f}
for any $z\in\mathfrak M_0$ satisfying $\Vert x-z\Vert_{L^\infty\big([0,1],\R^{N}\big)}<\rho(r)$
the following conditions hold:

\begin{enumerate}
\item\label{itm:PS15-bis}
for any $\delta \in ]0,\delta_0]$, for any  $[c_z,d_z]\subset[a-\Delta(x),b+\Delta(x)] \cap [0,1]$, $\delta$--interval for $z$,
and for all $s\in[c_z,d_z]$ such that
$\phi(z(s)) = -\delta$,
it is  \[g(\nabla\phi(z(s)),V_x(s))\ge \tfrac{\theta(r)}{2}\Vert V_x\Vert_{a,b};\]
\item\label{itm:PS15a}
$g\big(\nabla\phi(z(s)),V_x(s)\big)\ge\tfrac{\theta(r)}2\Vert V_x\Vert_{a,b}$
for all $s\in [a-\Delta(x),b+\Delta(x)] \cap [0,1]$ with $\phi(z(s)) = 0$;
\end{enumerate}

\item\label{itm:PS5f}  for all $z\in\mathfrak M_0$ and
for all $[a_z,b_z] \in \mathcal I_z$, such that
$\Vert x-z\Vert_{a_z,b_z}<\rho(r)$,
$[a_z,b_z] \cap [a,b] \neq \emptyset$,
 and $\mathcal D(x,a_z,b_z,a,b)\le E(r)$, it is
$$\int_{a_z}^{b_z}g\big(\dot z,\Dds V_x\big)\,\mathrm ds\le-\frac{\mu(r)}2\Vert V_x\Vert_{a,b}.$$
\end{enumerate}
\end{prop}

\begin{rem}\label{rem:E(r)}
By property (\ref{itm:PS5f}) of Proposition \ref{thm:PS} and the definition of $D(x,a_z,b_z,a,b)$ (cf.
\eqref{eq:defDxalfabeta}) we see that the number $E(r)$
gives a bound on
the admissible difference between the energy of  $x_{\vert[a,b]}$ and $x_{\vert[a_z,b_z]}$,
to obtain a rate of decrease $\mu(r)/2$ for
the quantity $\tfrac12\int_{a_z}^{b_z}g(\dot z,\dot z)\text ds$, when $\Vert x - z \Vert_{a_z,b_z} < \rho(r)$.
\end{rem}

\begin{rem}\label{rem:equivarianza}
By Proposition \ref{thm:PS} we see that we can choose $V_{{\mathcal R}x}$ so that $V_{{\mathcal R}x} = {\mathcal R}V_x$.
\end{rem}

\begin{proof}[Proof of Proposition \ref{thm:PS}]
Let $\kappa(r)$ given by Proposition \ref{thm:51new}. In order to prove
\eqref{itm:PS15-bis}, let us consider  $[c_z,d_z]$, $\delta$--interval for $z$.

Suppose $\delta \leq \frac{\kappa(r)}2$.
Since $\phi(x(a)=\phi(x(b))=0$ it is immediately seen that we can choose $\Delta_1(x) > 0$ and $\rho_{1}(r) >0$ such that
\begin{multline*}
[c_z,d_z] \subset [a-\Delta_1(x),b+\Delta_1(x)] \cap [0,1] \text{ and }
\Vert x -z \Vert_{L^\infty\big([0,1],\R^{N}\big)}<\rho_{1}(r) \Longrightarrow \\
\phi(x([c_z,d_z]) \subset [-\kappa(r),\kappa(r)].
\end{multline*}
Therefore  by \eqref{itm:51new-a} of Proposition \ref{thm:51new} we have
\begin{equation}\label{eq:star1}
g(\nabla\phi(x(s)),V_{x}(s))\geq \theta(r)\Vert V_x \Vert_{a,b} \text{ for all }s \in [c_z,d_z],
\end{equation}
so there exists $\rho_{2}(r) \in ]0,\rho_{1}(r)]$ such that if $\Vert x -z \Vert_{L^\infty\big([0,1],\R^{N}\big)}<\rho_{2}(r)$ then
\begin{equation}\label{eq:star2}
g(\nabla\phi(z(s)),V_x)s))\geq \frac{\theta(r)}2 \Vert V_x \Vert_{a,b} \text{ for all }s \in [c_z,d_z].
\end{equation}

Now assume $\delta \in ]\frac{\kappa(r)}2,\delta_0]$. Then there exists $\rho_{3}(r)\in ]0,\rho_{2}(r)]$
and $\Delta_2(x)\in ]0,\Delta_1(x)]$ such that
\begin{multline*}
[c_z,d_z] \subset [a-\Delta_2(x),b+\Delta_2(x)] \cap [0,1],
\Vert x -z \Vert_{L^\infty\big([0,1],\R^{N}\big)}<\rho_{3}(r)  \text{ and } s \in [c_z,d_z]\\
\text{ imply }
\phi(x(s)) \in [-\kappa(r),\kappa(r)] \text{ or } \\
s \in [\alpha,\beta] \subset [a,b], \text{ where } [\alpha,\beta] \text{ is a } (\delta_1,\frac{\kappa(r)}8)\text{--interval for $x$ with }
\delta_1 \geq \frac{\kappa(r)}4.
\end{multline*}

Then by \eqref{itm:51new-bis} and \eqref{itm:51new-a} of Proposition \ref{thm:51new}, \eqref{eq:star1} is satisfied
for any $[c_z,d_z] \subset [a-\Delta_2(x),b+\Delta_2(x)] \cap [0,1]$, $\delta$--interval for $z$ and for any $\delta \in ]0,\delta_0]$. Hence
there exists
$\rho_{4}(r) \in ]0,\rho_3(r)]$ such that also \eqref{eq:star2} is satisfied for any $[c_z,d_z] \subset [a-\Delta_2(x),b+\Delta_2(x)] \cap [0,1]$, $\delta$--interval for $z$ and $\delta \in ]0,\delta_0]$. This proves
\eqref{itm:PS15-bis} with $\rho(r) = \rho_4(r)$ and $\Delta(x) = \Delta_2(x)$. Clearly we can also choose $\rho(r)$ and $\Delta(x)$
so that
$\Vert x -z \Vert_{L^\infty\big([0,1],\R^{N}\big)}<\rho(r)$, $s \in [a-\Delta(x),b+\Delta(x)]$, and $\phi(z(s))=0$ imply $\phi(x(s) \in [-\kappa(r),\kappa(r)]$, from which we immediately
deduce that we can choose $\rho(r)$ and $\Delta(x)$ so that also \eqref{itm:PS15a} holds.

\medskip

Finally, to show \eqref{itm:PS5f} observe that, if $z\in\mathfrak M_0$,
$\Vert x-z\Vert_{a_z,b_z}<\rho(r)$,
$[a_z,b_z] \cap [a,b] \neq \emptyset$ and $\Dcal(x,a_z,b_z,a,b)\le E(r)$,
($E(r)$ given by \eqref{eq:5.14bis}), by Lemma
\ref{thm:lem5.2new}, we have
\begin{multline*}
\int_{a_z}^{b_z}g\big(\dot z,\Dds V_x\big)\,\mathrm ds\le\int_a^bg\big(\dot x,\Dds V_x\big)\,
\mathrm ds+\\
\sqrt{2}\left(\sqrt{L_{0}E(r)}+G_0\Vert x-z\Vert_{a_z,b_z}\left(1 + \sqrt{\frac{M_0+E(r)}
{\ell_0}}\right)\right)L_1\Vert V_x\Vert_{0,1}.
\end{multline*}
Since $\Vert V_x\Vert_{0,1}=\Vert V_x\Vert_{a,b}$,
applying \eqref{itm:51new-b}
of Proposition \ref{thm:51new} we have
\begin{multline*}
\int_{a_z}^{b_z}g\big(\dot z,\Dds V_x\big)\,\mathrm ds\le\\\left[-\mu(r)+
\sqrt{2}\left(\sqrt{L_{0}E(r)}+G_0\Vert x-z\Vert_{a_z,b_z}\left(1 + \sqrt{\tfrac{M_0+E(r)}{\ell_0}}\right)\right)\right]L_1
\Vert V_x\Vert_{a,b},
\end{multline*}
therefore we can choose $\rho(r)$ such that property \eqref{itm:PS5f}
is satisfied, because $L_1\sqrt{2L_{0}E(r)}=\tfrac{\mu(r)}4$ (see \eqref{eq:5.14bis}).
\end{proof}

\section{The Palais-Smale for $\mathcal V^-$--critical portions and the notion of topological non essential interval}\label{V-}

Before to describe the class of admissible homotopies we need to give the Palais-Smale  version for "${\mathcal V}^{-}$--critical portions". Towards this goal we need some preliminary notations and results.

Throughout this section we will
denote by \[\pi:\phi^{-1}\big([-\delta_0,0]\big)\longrightarrow\phi^{-1}(0)\] the
"projection" onto $\partial\Omega$ along orthogonal geodesics (cf. \eqref{eq:flusso+}). Thanks to Remark \ref{thm:rem4.11bis}
a simple contradiction argument
shows that the following properties are satisfied by irregular variationally critical portions of first type
(see also Corollary \ref{thm:cor4.11bis}):

\begin{lem}\label{thm:lem4.18}
There exists $\bar\gamma>0$ and $\delta_1\in\left]0,\delta_0\right[$ such that, for all $\delta\in\left]0,\delta_1\right]$, for any $x\in{\mathfrak M}_0$
such that $x_{\vert[a,b]}$ is an irregular variationally critical portion of first type with
$\frac12\int_a^bg(\dot x,\dot x)\,\mathrm ds<M_0$ and for any interval $[\alpha,\beta] \subset [a,b]$ including a cusp interval
$[t_1,t_2]$ satisfying \eqref{eq:alphabeta}
the following inequality holds:
\begin{equation}\label{eq:4.31}
\max\Big\{\Vert x(\beta)-\pi(x(\alpha))\Vert_{E},\,
\Vert x(\alpha)-\pi(x(\beta))\Vert_{E}\Big\}\ge(1+2\bar\gamma)
\Vert\pi(x(\beta))-\pi(x(\alpha))\Vert_{E},
\end{equation}
\end{lem}
\noindent (recall that $\Vert \cdot \Vert_{E}$ denotes the Euclidean norm).

\bigskip

The following Lemma says that curves satisfying \eqref{eq:4.31} are
far away from  those that  satisfy \eqref{eq:eq5.6}: in other words, they are far to
be critical with respect to $\mathcal V^-$. In particular,
the set of irregular variationally critical portions of first type consists of curves
at which the value of the functional can be decreased by deforming in the directions
of $\mathcal V^-$, as we shall see later.

Let $\bar\gamma$ be as in Lemma \ref{thm:lem4.18}.
\begin{lem}\label{thm:lem4.19} There exists $\delta_2\in\left]0,\delta_0\right[$ with the following property: for any $\delta\in\left]0,\delta_2\right]$,
for any $[a,b]\subset\R$ and for any $y\in H^1([a,b],\overline\Omega)$ satisfying \eqref{eq:eq5.6} and such that
\[
\phi(y(a))=\phi(y(b))=-\delta,
\phi(y(\bar t))=0 \text{ for some }\bar t\in ]a,b[,
\]
the following inequality holds:
\begin{equation}\label{eq:4.32}
\max\Big\{\Vert y(b)-\pi(y(a))\Vert_{E},\,\Vert y(a)-\pi(y(b))\Vert_{E}\Big\}
\le \left(1+\frac{\bar\gamma}2\right)\Vert\pi(y(b))-\pi(y(a))\Vert_{E}.
\end{equation}
\end{lem}

\begin{proof}
Let $y\in H^1([0,1],\overline\Omega)$ be fixed. If $y$ satisfies
\eqref{eq:eq5.6} then it is of class $C^1$. Moreover $\phi(y(s))\le
0$ on $[a,b]$ so that $g\big(\nabla\phi(y(s)),\dot y(s)\big)=0$ for
all $s$ with $\phi(y(s))=0$. Thanks to the conservation law $g(\dot
y,\dot y)=\text{constant}$ (see Lemma~\ref{thm:lem4.16}), we have
that $\dot y(s)\ne 0$ for all $s\in [a,b]$, since
$\phi(x(a))=-\delta$ and $\phi(x(\bar t))=0$. Denote by $\tilde y$ the
arc-length reparameterization of $y$ and let $[\tilde a,\tilde b]$
be the interval such that $\phi(\tilde y(\tilde a))=\phi(\tilde
y(\tilde b))=-\delta$.  Note that $[\tilde a,\tilde b]$ is uniquely
determined, (up to translations of the affine parameter), since the
strong concavity condition implies $g\big(\nabla\phi(y(a)),\dot
y(a)\big)>0$ and $g\big(\nabla\phi(y(b)),\dot y(b)\big)<0$. In
addition,  the strong concavity condition also ensures the existence
of $\tilde c>0$ such that $|\tilde b-\tilde a|\le\tilde c$ for any
$\tilde y$ as above. Then, for any $y$, we can assume $g(\dot y,\dot
y)=1\,\forall s\in[a,b]$ without loss of generality (observe that
\eqref{eq:4.32} is a condition independent on the parameterization
of $y$).

By contradiction, assume now that the claim does not hold. Then, there exists a  sequence of  intervals $[a_n,b_n]$
of bounded length and a sequence of functions $y_n$ such that ${y_n}_{\vert[a_n,b_n]}$ satisfies \eqref{eq:eq5.6},
$g(\dot y_n,\dot y_n)=1$ on $[a_n,b_n]$, $-\delta_n=\phi(y_n(a_n))=\phi(y_n(b_n))$, ($\delta_n\to 0^+$),
$\phi(y_n(t_n))=0$ for some $t_n\in\left]a_n,b_n\right[$
and (possibly reversing the orientation of $y_n$):
\begin{equation}\label{eq:8.7bis}
\Vert y_n(a_n)-\pi(y_n(b_n))\Vert_{E}>\left(1+\frac{\bar\gamma}2\right)\Vert\pi(y_n(a_n))-\pi(y_n(b_n))\Vert_{E}.
\end{equation}
Since $|b_n-a_n|$ is bounded, up to a translation of the parameter, we can assume $a_n$ and $b_n$ bounded.
Moreover $\overline\Omega$ is compact, so by Lemma \ref{thm:lem4.16}, using local coordinates expression of the covariant derivative
(involving Christoffel symbols, cf. \cite{docarmo}) we obtain the existence of a constant
$C > 0$ independent of $n$ such that
\[
\Vert\ddot y_n(s)\Vert_{E}\le C g(\dot y_n(s), \dot y_n(s)) \text{ for a.e. } s \in [a_n,b_n].
\]
But $g(\dot y_n(s), \dot y_n(s) = 1$ for any $s \in [a_n,b_n]$, then
\[
\Vert\ddot y_n(s)\Vert_{E}\le C   \text{ for a.e. } s \in [a_n,b_n].
\]
Now, let $t_n^-\in\left]a_n,b_n\right[$ be the first instant at which the function $t\mapsto\phi(y_n(t))$ vanishes.
By Taylor expansion we have
\begin{equation}\label{eq:4.33}
\Vert y_n(s)-y_n(t_n^-)-\dot y_n(t_n^-)(s-t_n^-)\Vert_{E}\le\frac C2(s-t_n^-)^2
\end{equation}
for any $s\in[a_n,t_n^-]$. The argument is analogous in the interval $[t_n^+,b_n]$, where $t_n^+$ denotes the last
instant when $\phi(y_n(\cdot))$ vanishes.
Now, the same estimates, applied to $\pi(y_n)$, give the existence of a constant
$C_1>0$ such that:
\[\Vert\pi(y_n(s))-\pi(y_n(t_n^-))-\mathrm d\pi(y_n(t_n^-))[\dot y_n(t_n^-)](s-t_n^-)\Vert_{E}\le C_1(s-t_n^-)^2.\]
But for all $z\in\partial\Omega$, $\pi(z)=z$ and $\mathrm d\pi(z)$ is the identity map, and therefore
\begin{equation}\label{eq:4.34}
\Vert\pi(y_n(s))-y_n(t_n^-)-\dot y_n(t_n^-)(s-t_n^-)\Vert_{E}\le C_1(s-t_n^-)^2
\end{equation}
for all $s\in[a_n,t_n^-]$, and, analogously, for all $s$ in $[t_n^+,b_n]$.
Now
\begin{multline*}
\Vert\pi(y_n(a_n))-\pi(y_n(b_n))\Vert_{E}=\\
\Vert[\pi(y_n(a_n))-y_n(t_n^-)-\dot y_n(t_n^-)(a_n-t_n^-)]+\left[y_n(t_n^-)+\dot y_n(t_n^-)(a_n-t_n^-)\right]\\
-\left[y_n(t_n^+)+\dot y_n(t_n^+)(b_n-t_n^+)\right]+[y_n(t_n^+)+\dot y_n(t_n^+)(b_n-t_n^+)-\pi(y_n(b_n))]\Vert_{E},
\end{multline*}
and, similarly,
\begin{multline*}
\Vert y_n(a_n)-\pi(y_n(b_n))\Vert_{E}=\\
\Vert [y_n(a_n)-y_n(t_n^-)-\dot y_n(t_n^-)(a_n-t_n^-)]+\left[y_n(t_n^-)+\dot y_n(t_n^-)(a_n-t_n^-)\right]\\
-\left[y_n(t_n^+)+\dot y_n(t_n^+)(b_n-t_n^+)\right]+[y_n(t_n^+)+\dot y_n(t_n^+)(b_n-t_n^+)-\pi(y_n(b_n))]\Vert_{E}.
\end{multline*}
The strong concavity assumption implies that $(t_n^--a_n)\to 0$ and $(b_n-t_n^+)\to 0$ as $n\to+\infty$ (since $\delta_n \to 0$).
Then, by \eqref{eq:4.33} and \eqref{eq:4.34} (and their analogues in $[t_n^+,b_n]$) we have
\begin{equation}\label{eq:4.35}
\lim_{n\to+\infty}\frac{\Vert y_n(a_n)-\pi(y_n(b_n))\Vert_{E}}{\Vert\pi(y_n(a_n))-\pi(y_n(b_n))\Vert_{E}}=1
\end{equation}
unless (up to subsequences)
\begin{equation}\label{eq:4.36}
\lim_{n\to+\infty}(y_n(t_n^-)-y_n(t_n^+))=0,\quad\lim_{n\to+\infty}(\dot y_n(t_n^-)+\dot y_n(t_n^+))=0,
\end{equation}
(recall that $a_n-t_n^-<0$ and $b_n-t_n^+>0$)
that is the only case when \eqref{eq:4.35} above cannot be immediately deduced.
Therefore, to conclude the proof ---recalling \eqref{eq:8.7bis}--- it suffices to prove that \eqref{eq:4.36} cannot hold, using the fact
that $g(\dot y_n,\dot y_n)\equiv 1$ on $[a_n,b_n]$. Suppose by contradiction that \eqref{eq:4.36} holds. First, observe that
\[
\Vert\dot y_n(t_n^+)-\dot y_n(t_n^-)\Vert_{E}\le\int_{t_n^-}^{t_n^+}\Vert\ddot y_n(s)\Vert_{E}\,\text ds\le C (t_n^+-t_n^-)
\]
and therefore, by the second property in \eqref{eq:4.36} and the fact
that $g(\dot y_n,\dot y_n)\equiv 1$ on $[a_n,b_n]$, we deduce that
\[
\inf_{n\in\N}(t_n^+-t_n^-)> 0.
\]
 Since the sequence $y_n$
is bounded in $H^{2,\infty}([a_n,b_n],\overline\Omega)$,
$[t_n^-,t_n^+]\subset [a_n,b_n]$ and $b_n$, $a_n$ are bounded, up to
subsequence we deduce the existence of an interval $[t^-,t^+]$ ($t^- < t^+$) and
$y\in H^{2,\infty}([t^-,t^+],\overline\Omega)$ such that (recall we
are assuming  that \eqref{eq:4.36} holds)
\begin{align*}
&y_{\vert[t^-,t^+]}\ \text{satisfies\ } \eqref{eq:eq5.6},\\
&g(\dot y,\dot y)\equiv 1\quad\text{on\ } [t^-,t^+],\\
&y(t^-)=y(t^+), \qquad\text{(by first property in \eqref{eq:4.36})},\\
&\dot y(t^+)=-\dot y(t^-),\qquad\text{(by second property in \eqref{eq:4.36})}.
\end{align*}
Then by the uniqueness of the solution of the Cauchy problem for the "free" geodesic equation ($\Ddt\dot y= 0$)
and for the "constrained" geodesic equation \eqref{eq:4.30bis}, we have:
\[
y((t^++t^-)-s)=y(s)\qquad\forall s\in [t^-,t^+]
\]
since the functions on the two side of the equality are solutions of the same differential equation
with the same initial data. But $y$ is of class $C^1$ while
\[
\dot y\left(\frac{t^++t^-}2\right)=-\dot y\left(\frac{t^++t^-}2\right),
\]
from which we deduce $\dot y\left(\frac{t^++t^-}2\right)=0$, in contradiction with $g(\dot y,\dot y)\equiv 1$ on
$[t^-,t^+]$.
\end{proof}
Using vector fields in $\mathcal V^-(x),\,x\in\mathfrak M$, we shall build a flow $H_0$ moving away from
the set of irregular variationally critical portions of first type,
without increasing the energy functional. To this aim let $\pi,\,\bar\gamma,\,\delta_1,\,\delta_2$ be chosen
as in Lemma \ref{thm:lem4.18} and \ref{thm:lem4.19}, and set
\begin{equation}\label{eq:4.37}
\bar\delta=\min\{\delta_1,\delta_2\}.
\end{equation}
Let us give the following:
\begin{defin}\label{thm:defcostapp}
Let $x\in\mathfrak M_0$, $[a,b]\in \mathcal J_x^{0}$ and $[\alpha,\beta]\subset[a,b]$.
We say that \emph{$x$ is $\bar\delta$-close to $\partial\Omega$ on $[\alpha,\beta]$} if the following situation occurs:
\begin{enumerate}
\item\label{itm:app1} $\phi(x(\alpha))=\phi(x(\beta))=-\bar\delta$;
\item\label{itm:app2} $\phi(x(s))\ge-\bar\delta$ for all $s\in[\alpha,\beta]$;
\item\label{itm:app3}  there exists $s_0\in\left]\alpha,\beta\right[$ such that $\phi(x(s_0))>-\bar\delta$;
\item $[\alpha,\beta]$ is minimal with respect to properties \eqref{itm:app1}, \eqref{itm:app2} and \eqref{itm:app3}.

\end{enumerate}
If $x$ is $\bar\delta$-close to $\partial\Omega$ on $[\alpha,\beta]$, the \emph{maximal proximity} of
$x$ to $\partial\Omega$ on $[\alpha,\beta]$ is defined to be the quantity
\begin{equation}\label{eq:maxprox}
\mathfrak p^x_{\alpha,\beta}=\max_{s\in[\alpha,\beta]}\phi(x(s)).\end{equation}
\end{defin}
Given an interval $[\alpha,\beta]$  where $x$ is $\bar\delta$-close to $\partial\Omega$, we define
the following constant, which is a sort of measure of how much the curve $x_{\vert[\alpha,\beta]}$
fails to flatten along $\partial\Omega$:
\begin{defin}\label{thm:defflat}
The \emph{bending constant} of $x$ on $[\alpha,\beta]$ is defined by:
\begin{equation}\label{eq:bendconst}
\mathfrak b^x_{\alpha,\beta}=
\frac{\max\big\{
\Vert x(\beta)-\pi(x(\alpha))\Vert_{E},\Vert x(\alpha)-\pi(x(\beta))\Vert_{E}\big\}}{\Vert \pi(x(\alpha))-\pi(x(\beta))\Vert_{E}}\in\R^+\cup\{+\infty\},
\end{equation}
where $\pi$ denotes the projection onto $\partial\Omega$ along orthogonal geodesics.
\end{defin}

We observe that $\mathfrak b^x_{\alpha,\beta}=+\infty$ if and only if $x(\alpha)=x(\beta)$.

Let $\bar\gamma$ be as in Lemma~\ref{thm:lem4.18}.
Our next result is the counterpart of the result in Propositions~\ref{thm:51new} and \ref{thm:PS}, and it
will be used to define a flow that averts from paths in
the set of irregular
variationally critical portions of first type: if the bending constant of a path $y_{\vert[\alpha,\beta]}$
is greater than or equal to $1+\bar\gamma$, then the energy functional $f_{\alpha,\beta}$ can be decreased in a
neighborhood of $y_{\vert[\alpha,\beta]}$ keeping the endpoints $y(\alpha)$ and $y(\beta)$
fixed and moving far away from $\partial \Omega$.

First we need the following
\begin{defin}\label{def:summary-interval}
An interval $[\tilde\alpha,\tilde\beta]$ is called summary interval for $x \in \mathfrak M_0$ if it is the smallest interval included in
$[a,b]\in\mathcal I_{x}^{0}$ and including all the intervals $[\alpha,\beta]$ such that
\begin{itemize}
\item $x$ is $\bar\delta$--close to $\partial\Omega$ on $[\alpha,\beta]$,
\item $b^{x}_{\alpha,\beta} \geq 1 + \bar \gamma$.
\end{itemize}
\end{defin}

\begin{prop}\label{thm:prop4.20}
There exist positive constants $\sigma_0\in\left]0,{\bar\delta}/2\right[$,
$\varepsilon_0 \in ]0,\bar \delta - 2\sigma_0[$,
$\rho_0,\,\theta_0$ and $\mu_0$ such that for all
$y\in\mathfrak M_0$, for all $[a,b]\in \mathcal I_y^{0}$ satisfying $\frac12\int_{a}^{b}g(\dot y, \dot y)ds \leq M_0$
and for all $[\tilde\alpha,\tilde\beta]$ summary interval for $y$ including an interval $[\alpha,\beta]$ such that :
\[
y \text{ is $\bar\delta$--close to $\partial\Omega$ on }[\alpha,\beta], \,
\mathfrak b^y_{\alpha,\beta}\ge1+\bar\gamma, \, \mathfrak p^y_{\alpha,\beta}\ge-2\sigma_0,
\]
there exists $V_y\in H_0^1\big([\tilde\alpha,\tilde\beta],\R^N\big)$ with the following property:

for all $z\in H^1([\tilde\alpha,\tilde\beta],\R^N)$ with $\Vert z-y\Vert_{\tilde\alpha,\tilde\beta}\le\rho_0$ it is:
\begin{enumerate}
\item\label{itm:teruno} $V_y(s)=0$ for all $s\in[\tilde\alpha,\tilde\beta]$ such that $\phi(z(s))\le-\bar\delta + \varepsilon_0$;
\item\label{itm:terdue}
$g\big(\nabla\phi(z(s)),V_y(s)\big)\le-\theta_0\Vert V_y\Vert_{\tilde\alpha,\tilde\beta}$, if $s\in[\tilde\alpha,\tilde\beta]$
and $\phi(z(s))\in[-2\sigma_0,2\sigma_0]$
\item\label{itm:tertre} $\int_{\tilde\alpha}^{\tilde\beta} g(\dot z,\Ddt V_y)\,\text dt\le-\mu_0\Vert V_y\Vert_{\tilde\alpha,\tilde\beta}.$
\end{enumerate}
\end{prop}

\begin{proof}
We argue by contradiction, assuming the existence of sequences
$\sigma_n\to0^+ (\sigma_n<\bar\delta/2),\,
\varepsilon_n \to0^+, \varepsilon_n < \bar \delta -2\sigma_n,
\rho_n\to 0^+,\,\theta_n\to0^+,\,\mu_n\to 0^+,
\,y_n\in\mathfrak M_{0},\,[a_n,b_n]
\in \mathcal I_{y_n}^{0}$ satisfying the property
\[\frac12\int_{a_n}^{b_n}g(\dot y_n,\dot y_n)ds \leq M_0,\]
and
$[\alpha_n,\beta_n]\subset[a_n,b_n]$ such that:
\begin{equation*}
\begin{cases}
&\phi(y(\alpha_n))=\phi(y(\beta_n))=-\bar\delta,\quad\phi(y_n(s))\ge-\bar\delta\,\text{\ in\ }[\alpha_n,\beta_n],\\
&\exists s_n\in[\alpha_n,\beta_n]\,:\,\phi(y_n(s_n))\ge-\sigma_n,\\
&\phi(y_n(s))\le 0,\quad\text{for all}\ s\in[\alpha_n,\beta_n]\,\text{(since\ } [\alpha_n,\beta_n]\subset[a_n,b_n]
\in {\mathcal J}_{y_n}^{0}\text{),}\\
&\max\{\Vert y_n(\alpha_n)-\pi(y_n(\beta_n))\Vert_{E},\,\Vert\pi(y_n(\alpha_n))-y_n(\beta_n)\Vert_{E}\}\ge\\
&\hspace{6cm}(1+\bar\gamma)\Vert\pi(y_n(\alpha_n))-\pi(y_n(\beta_n))\Vert_{E},
\end{cases}
\end{equation*}
with $[\alpha_n,\beta_n]$ minimal with respect to the above properties
and such that the following holds true:

\begin{center}
\begin{minipage}{11cm}
for all $V\in H^1_{0}\big([\alpha_n,\beta_n],\R^N\big)$, there exists $z=z(y_n,V)$
such that\hfill\break $\Vert z-y_n\Vert_{\alpha_n,\beta_n}\le\rho_n$ and the conditions
\begin{enumerate}
\item\label{itm:proofuno} $V(s)=0$ for all $s\in[\alpha_n,\beta_n]$ such that $\phi(z(s))\le-\bar\delta + \varepsilon_n$;
\item\label{itm:proofdue}
$g(\nabla\phi(z(s)),V(s))\le-\theta_n\Vert V\Vert_{\alpha_n,\beta_n}$, if $s\in[\alpha_n,\beta_n]$ and $\phi(z(s))\in[-2\sigma_n,2\sigma_n]$
\item\label{itm:prooftre} $\int_{\alpha_n}^{\beta_n} g(\dot z,\Ddt V)\,\text dt\le-\mu_n\Vert V\Vert_{\alpha_n,\beta_n},$
\end{enumerate}
do not hold true at the same time. (Note that $V$ can be consider equal to $0$ outside $[\alpha_n,\beta_n]$).
\end{minipage}
\end{center}

In the following, we will first build a vector field $V$
satisfying \eqref{itm:proofuno}--\eqref{itm:proofdue} \emph{for
every} $z$ such that $\Vert z-y_n\Vert_{\alpha_n,\beta_n}\le\rho_n$.
Therefore, there must exist $\widehat z=\widehat z(y_n,V)$
such that property \eqref{itm:prooftre} does not hold, namely
\begin{equation}\label{eq:4.40}
\int_{\alpha_n}^{\beta_n}g(\dot{\hat z},\Ddt V)\,\text dt>-\mu_n\Vert V\Vert_{\alpha_n,\beta_n}.
\end{equation}

Towards this goal first note that
\[\int_{\alpha_n}^{\beta_n}g(\dot y_n,\dot y_n)\,\text ds\le
\int_{a_n}^{b_n}g(\dot y_n,\dot y_n)\,\text ds \leq 2M_0,\]
while $|\beta_n-\alpha_n|\ge
|\beta_n-s_n|$, which is bounded away from 0 because
$\phi(x(\beta_n))=-\bar\delta$ and $\phi(x(s_n))=-\sigma_n\to 0$.
Then up to subsequences there exists an interval
$[\alpha,\beta]\subset[0,1]$ with $\alpha<\beta$ and a curve $y\in
H^1([\alpha,\beta],\R^N)$ such that
\[
{y_{n}}_{\vert[\alpha_n,\beta_n]}\to y_{\vert[\alpha,\beta]}\quad \text{\ weakly in\ }H^1 \text{\ and uniformly}.
\]
By the properties satisfied by ${y_{n}}_{\vert[\alpha_n,\beta_n]}$ and uniform convergence, we have
\begin{equation}\label{eq:88}
\begin{cases}
&\phi(y(\alpha))=\phi(y(\beta))=-\bar\delta,\,\phi(y(s))\ge-\bar\delta,\,\forall s\in[\alpha,\beta],\\
&\text{there exists}\ s_0\in\left]\alpha,\beta\right[:\,\phi(y(s_0))=0,\\
&\phi(y(s))\le 0,\quad\forall s\in [\alpha,\beta],\\
&\max\{\Vert y(s_1)-\pi(y(s_2))\Vert_{E},\,\Vert\pi(y(s_1))-y(s_2)\Vert_{E}\}\ge\\
&\hspace{6cm}(1+\bar\gamma)\Vert\pi(y(s_1))-\pi(y(s_2))\Vert_{E},\\
\end{cases}
\end{equation}
for all $s_1\in C_{\alpha}$ and all $s_2\in C_{\beta}$, where $C_\alpha$ (resp.: $C_\beta$) is the connected
component of $(\phi\circ y)^{-1}(-\bar\delta)$ containing $\alpha$ (resp.: $\beta$).

Take $\alpha'=\sup\{s_1\,:\,s_1\in C_\alpha\}$ and $\beta'=\inf\{s_2\,:\,s_2\in C_\beta\}$.
Clearly $\phi(y(s))>-\bar\delta$ for all $s\in\left]\alpha',\beta'\right[$, while $\phi(y(\alpha'))=\phi(y(\beta'))=-\bar\delta$.

Now fix a sequence $V_k\in H_0^1([0,1],\R^N)$  such that (taking $k$ sufficiently large) $V_k(s)=0$
for all $s\not\in[\alpha'+\tfrac1k,\beta'-\tfrac1k]$, $V_k\not\equiv 0$ and  $g\big(V_k(s),\nabla\phi(y(s))\big)\le 0$ for all $s\in[\alpha',\beta']$ with $\phi(y(s))=0$. A simple contradiction argument shows that,
for sufficiently large $n$, if
$\Vert z-y_n\Vert_{\alpha_n,\beta_n}\le\rho_n$, then $V=V_k$ satisfies property \eqref{itm:teruno} with  $[\alpha,\beta]=[\alpha_n,\beta_n]$, because
$\varepsilon_n \rightarrow 0$ and $\phi(y(\delta)) > 0$ in $[\alpha'+\frac1k,\beta'-\frac1k]$.
We can now modify $V_k$ so that also property \eqref{itm:terdue} will be satisfied.

To this aim, set:
\[\begin{split}
\lambda_n=\sup\Big\{\max\big\{ &\frac{g\big(V_k(s),\nabla\phi(w(s))\big)+\theta_n\Vert V_k\Vert_{\alpha_n,\beta_n}}{
g\big(\nabla\phi(y_n(s)),\nabla\phi(w(s))\big)-\theta_n\Vert\nabla\phi(y_n)\Vert_{\alpha_n,\beta_n}},0\big\}:
\\ &  w\in H^1\big([\alpha_n,\beta_n],\R^N\big),\, \Vert y_n-w\Vert_{\alpha_n,\beta_n}\le\rho_n,\\ &
s\in[\alpha_n,\beta_n],\,\phi(w(s))\in\left[-2\sigma_n,2\sigma_n\right]\Big\}.
\end{split}
\]
The supremum above is finite for sufficiently large $n$, due to the facts that
$\nabla\phi\ne0$ in $\partial\Omega$, and that $\theta_n\to 0$ and $\Vert y_n\Vert_{\alpha_n,\beta_n}$ is bounded. Moreover, since $y_n$ is uniformly
convergent to $y$,
$\Vert y_n-w\Vert_{L^\infty[\alpha_n,\beta_n],\R^m)}\to 0$ and $g(V_k(s),\nabla\phi(y(s)))\le 0,\,\forall s\in [\alpha,\beta]$ such that $\phi(y(s))=0$,
we see that
\begin{equation}\label{eq:eq5.23}
\lim_{n\to\infty}\lambda_n=0.
\end{equation}
For all $k$ sufficiently large, consider the piecewise affine function
$\chi_k:[0,1]\to\R$ such that:
\[\chi_k(s)=\left\{\begin{array}{cl}0,&\text{if $s\le\alpha'+\frac1k$;
\vphantom{$\displaystyle\int$}}
\\
1,&\text{if $s\in\left[\alpha'+\frac2k,\beta'-\frac2k\right]$;}\\
0&\text{if $s\ge\beta'-\frac1k$,\vphantom{$\displaystyle\int$}}
\end{array}\right.\]
and the vector field $V_{n,k}\in H^1_0\big([\alpha_n,\beta_n],\R^N\big)$:
\[V_{n,k}(s)=V_k(s)-\lambda_n\chi_{k}(s)\nabla\phi(y_n(s))\quad\text{for $s\in[\alpha_n,\beta_n]$}.\]
Clearly, $\sup\limits_n\Vert V_{n,k}\Vert_{\alpha_n,\beta_n}<+\infty$, and $V_{n,k}\not\equiv0$
if $n$ is sufficiently large.
Note that for $k$ and $n$ sufficiently large it is
\[\phi(z_n(s)) \in [-2\sigma_n,2\sigma_n]\quad\Rightarrow\quad s\in\left[\alpha'+\tfrac2k,\beta'-\tfrac2k\right],\]
therefore (by the choice of $\lambda_n$ and $\chi_k$) $V_{n,k}$ satisfies property \eqref{itm:terdue}.
Moreover, since $V_k$ satisfies \eqref{itm:teruno} (for any $k$ sufficiently large) this is true also for $V_{n,k}$.

Then there exists $\hat z$ such that $\Vert\hat z-y_n\Vert_{\alpha_n,\beta_n}\le\rho_n\to 0$ and
\begin{equation}\label{eq:replace}
\int_{\alpha_n}^{\beta_n}g\big(\dot{\hat z},\Ddt V_{n,k}\big)\,\mathrm dt>
-\mu_n\Vert V_{n,k}\Vert_{\alpha_n,\beta_n}
\end{equation}
and since $\Vert V_{n,k}\Vert_{\alpha_n,\beta_n}$ is bounded,
\[
\liminf_{n\to+\infty}\int_{\alpha_n}^{\beta_n}g\big(\dot{\hat z},\Ddt V_{n,k}\big)\,\text dt\ge 0.
\]
Now, $\Vert\hat z-y_n\Vert_{\alpha_n,\beta_n}\to 0$, then we obtain
\[\liminf_{n\to\infty}\int_{\alpha_n}^{\beta_n}g\big(\dot y_n,\Ddt V_{n,k}\big)\,\mathrm dt\ge0,\]
and in addition $\lambda_n\to 0$, while $\nabla\phi(y_n)$ is bounded in $H^1([\alpha_n,\beta_n],\R^N)$; therefore $V_{n,k}$ tends to
$V_k$ in $H^1([\alpha_n,\beta_n],\R^N)$, so
\[
\liminf_{n\to\infty}\int_{\alpha_n}^{\beta_n}g(\dot y_n,\Ddt V_k)\,\text dt\ge 0.
\]
Moreover $V_k=0$ outside $[\alpha'+\tfrac1k,\beta'-\tfrac1k]$ and $[\alpha_n,\beta_n]\supset[\alpha'+\tfrac1k,\beta'-\tfrac1k]$. Therefore
\[
\liminf_{n\to\infty}\int_{\alpha'+1/k}^{\beta'-1/k}g(\dot y_n,\Ddt V_k)\,\text dt\ge 0
\]
and, since $\dot y_n\rightharpoonup\dot y$ in $L^2\big([\alpha'+\tfrac1k,\beta'-\tfrac1k],\R^N\big)$, we get
\[
0\le\lim_{n\to\infty}\int_{\alpha'+1/k}^{\beta'-1/k}g(\dot y_n,\Ddt V_k)\,\text dt=\int_{\alpha'+1/k}^{\beta'-1/k}g(\dot y,\Ddt V_k)\,\text dt=\int_{\alpha'}^{\beta'}g(\dot y,\Ddt V_k)\mathrm dt
\]
because $V_k=0$ on $[\alpha',\alpha'+\tfrac1k]\cup[\beta'-\tfrac1k,\beta']$.
Finally, by arbitrariness of the interval $[\alpha'+\tfrac1k,\beta'-\tfrac1k]$, we have
\[
\int_{\alpha'}^{\beta'}g(\dot y,\Ddt V)\,\text dt\ge 0,
\]
for all $V\in H_0^1([\alpha',\beta'],\R^N)$ such that
\[g(V(s),\nabla\phi(y(s)))\le 0\,\text{ for all $s$ satisfying }\phi(y(s))=0.\]
But this is contradiction with \eqref{eq:88} and Lemma \ref{thm:lem4.19}, and the proof is complete.
\end{proof}

\begin{rem}\label{rem:sigma1}
To obtain flows moving far from curves having topologically non essential intervals (that we define below) it is crucial to fix a constant $\sigma_1 \in ]0,\sigma_0]$
such that
\[
\sigma_1 \leq \frac27\rho_0\theta_0,
\]
where $\rho_0, \, \theta_0$ are given by Proposition  \ref{thm:prop4.20}.
\end{rem}

\medskip
Proposition \ref{thm:prop4.20} is a crucial ingredient to define the class of the admissible homotopies, whose elements will avoid
irregular variationally critical
points of first type. The description of this class is based of the notion of topologically non essential interval given below.

\medskip
Let $\bar\delta$ be as in \eqref{eq:4.37}, $\bar\gamma$ as in Lemma \ref{thm:lem4.18}
and $\sigma_1$ as in Remark \ref{rem:sigma1}.
\begin{defin}\label{def:4.21}
Let $y\in \mathfrak M$ be fixed. An interval $[\alpha,\beta]\subset[a,b]\in{\mathcal I}_{y}^0$, is called
\emph{topologically not essential interval (for $y$)} if  $y$ is $\bar\delta$-close to $\partial\Omega$ on
$[\alpha,\beta]$, with $\mathfrak p^y_{\alpha,\beta}\ge-\sigma_1$ and
$\mathfrak b^y_{\alpha,\beta}\ge(1+\tfrac32\bar\gamma)$.
\end{defin}

\begin{rem}\label{rem:4.22}
By Lemma \ref{thm:lem4.18} the intervals $[\alpha,\beta]$  containing cusp intervals
$[t_1,t_2]$ of curves $x$, which are
irregular variationally critical portion of first type, and satisfying $\Theta_x(t_1,t_2)\ge d_0$ are topologically not essential intervals with
$\mathfrak p_{\alpha,\beta}^x=0$
and $\mathfrak b_{\alpha,\beta}^x\ge 1+2\bar\gamma$.
This fact will allow us to use  Proposition \ref{thm:prop4.23} in section \ref{sec:firstdeflemma} and move away from the set of irregular variationally critical portions of first type
without increasing the value of the energy functional.
\end{rem}

\section{The admissible homotopies}\label{sec:homotopies}

In the present section
we shall list the properties of the admissible homotopies used
to prove existence result.
This choice is clearly crucial for the proof: the notion of topological critical level,
that we shall use also in this paper to obtain the existence result,
depends indeed by the choice of the admissible homotopies.

We shall consider continuous homotopies
$h:[0,1]\times\Dcal\to\mathfrak M$ where
$\mathcal D$ is a closed $\Rcal$--invariant subset of $\mathfrak C$. (However the following notions can be done at the same way when $h$ is defined
in any $\mathcal R$--invariant subset of $\mathfrak M$, not necessarily included in $\mathfrak C$).

Recall that $\mathfrak C$ and $\mathfrak C_0$ are defined in \eqref{eq:2.6bis}. The following basic properties for our admissible homotopies
holds:

\begin{equation}\label{eq:numero1}
h(0,\cdot) \text{ is the inclusion of } \Dcal
\text{ in }\mathfrak M;
\end{equation}

\begin{equation}\label{eq:numero2}
h(\tau,\gamma) \text{ is a constant curve for all }
\tau\in[0,1], \text{ for all } \gamma\in\mathfrak C_0\cap\Dcal;
\end{equation}

\begin{equation}\label{eq:numero3}
h \text{ is $\Rcal$--equivariant, namely }
h(\tau,\Rcal\gamma) = \Rcal h(\tau,\gamma) \text{ for all }\tau \in [0,1], \gamma \in \mathcal D.
\end{equation}

The homotopies that we shall use are of 3 types: outgoing homotopies, reparameterizazions and ingoing homotpies.
They can be described in the following way.

\begin{defin}\label{thm:tipoA}
Let $0 \leq \tau' < \tau'' \leq 1$. We say that $h$ is of type $A$
in $[\tau',\tau'']$ if it satisfies the following properties for all $\tau_0 \in [\tau',\tau'']$, for all $s_0 \in [0,1]$, for all $x \in \mathcal D$:
\begin{enumerate}
\item\label{unico}
if $\phi(h(\tau_0,x)(s_0)) = 0$, then
$\tau \mapsto \phi(h(\tau,x)(s_0))$ is strictly increasing in a neighborhood of $\tau_0$.
\item\label{doppio} for all $\delta \in ]0,\delta_0]$, if $[\alpha_{\tau_0},\beta_{\tau_0}]$ is a $\delta$--interval for $h(\tau_0,x)$,
$s \in [\alpha_{\tau_0},\beta_{\tau_0}]$
and $\phi(h(\tau_0,x)(s_0))=-\delta$, then $\tau \mapsto \phi(h(\tau,x)(s_0))$ is strictly increasing in a neighborhood of $\tau_0$.
\end{enumerate}
\end{defin}

\begin{rem}\label{rem:monotonia-intervalli}
It is relevant to observe that, by property above of
Definition \ref{thm:tipoA}, if $[a_{\tau},b_{\tau}]$ denotes any
interval in $\mathcal I_{h(\tau,\gamma)}$ we have:
\begin{equation*}\label{eq:4.35f}
\tau' \le\tau_1<\tau_2\le\tau''
\text{ and }
[a_{\tau_1},b_{\tau_1}]\cap
[a_{\tau_2},b_{\tau_2}]\ne\emptyset \Longrightarrow
[a_{\tau_2},b_{\tau_2}] \subset
[a_{\tau_1},b_{\tau_1}].
\end{equation*}
\end{rem}

\begin{defin}\label{thm:tipoB}
Let $0 \leq \tau' < \tau'' \leq 1$. We say that $h$ is of type $B$
in $[\tau',\tau'']$ if it satisfies the following property:
\begin{multline*}
\exists \Lambda : [\tau',\tau''] \times \mathcal H_0^1([0,1],[0,1]) \rightarrow [0,1] \text{ continuous and such that }\\
\Lambda(\tau,\gamma)(0)=0,\,\Lambda(\tau,\gamma)(1)=1, \, \forall \tau \in [\tau',\tau''],\,\forall \gamma \in \mathcal D,\\
s\mapsto\Lambda(\tau,\gamma)(s) \text{is strictly increasing in }[0,1], \; \forall \tau \in [\tau',\tau''],\forall \gamma \in \mathcal D, \\
\Lambda(0,\gamma)(s) = s \text{ for any }\gamma \in \mathcal D, s \in [0,1] \text{ and }\\
h(\tau,\gamma)(s) = (\gamma \circ \Lambda(\tau,\gamma))(s) \;
\forall \tau \in [\tau',\tau''], \forall s \in [0,1], \forall \gamma
\in \mathcal D.
\end{multline*}
\end{defin}

\begin{defin}\label{thm:tipoC}
Let $0 \leq \tau' < \tau'' \leq 1$. We say that $h$ is of type $C$
in $[\tau',\tau'']$ if it satisfies the following properties:
\begin{enumerate}
\item\label{eq:CI}
$h(\tau',\gamma)(s) \neq \overline\Omega \Rightarrow h(\tau,\gamma)(s) =
h(\tau',\gamma)(s)$ for any $\tau \in [\tau',\tau'']$;
\item\label{eq:CIbis}
$h(\tau',\gamma)(s) \in \Omega
\Rightarrow h(\tau,\gamma)(s) \in \Omega$
for any $\tau \in [\tau',\tau'']$;
\end{enumerate}
\end{defin}

\bigskip
The interval $[0,1]$ where $\tau$ lives will be partitioned in the
following way:
\begin{multline}\label{eq:partizione}
\text{There exists a partition of the interval } [0,1],\; 0=\tau_{0} < \tau_1 < \ldots <\tau_k =1 \text{ such that}\\
\text{ on any interval } [\tau_i, \tau_{i+1}], i=0,\ldots,k-1,
\text{ the homotopy $h$ is of type A,B, or C.}
\end{multline}

Moreover, in order to move far from topologically not essential
intervals (cf. Definition \ref{def:4.21}) we need the
following further property:

\begin{multline}\label{eq:numero8}
\text{for any }[a,b]\in {\mathcal J}_{h(\tau,\gamma)} \text{ and for
all }
[\alpha,\beta]\subset[a,b] \text{ topologically non essential, }\\
\phi(h(\tau,\gamma)(s))\le-\frac{\sigma_1}2 \text {for all }s
\in[\alpha,\beta],
\end{multline}
where ${\sigma_1}$ is defined in Remark \ref{rem:sigma1}.
\bigskip

We finally define the following classes of admissible homotopies:

\begin{multline}\label{eq:class0}
\mathcal H = \{(\mathcal D, h): \mathcal D \text{ is closed, $\mathcal R$--invariant subset of $\mathfrak C$ and } \\
h:[0,1] \times \mathcal D \rightarrow \mathfrak M \text{ satisfies
\eqref{eq:numero1}-\eqref{eq:partizione}}\},
\end{multline}

\begin{equation}\label{eq:class1}
{\mathcal H}_1 = \{(\mathcal D, h) \in \mathcal H \text{ such that }
h:[0,1] \times \mathcal D \rightarrow \mathfrak M \text{ satisfies also \eqref{eq:numero8} }\}.
\end{equation}

\begin{rem}\label{rem:6.6}
Observe that, thanks to the properties of $\mathfrak C$, denoting by
$h_0$ the constant identity homotopy it is $(\mathfrak C,h_0) \in
\mathcal H_1$.
\end{rem}

\bigskip

Homotopies of type A will be used far from variationally critical portions, homotopies of type B nearby variationally critical portions
of II type, while homotopies of type C will be used nearby variationally critical portions
of I type.

We shall use the following functional defined
for any $(\mathcal D,h) \in \mathcal H$:
\begin{equation}\label{eq:funzionepreparatoria}
{\mathcal F}(\mathcal D,h) =
\sup\{\frac{b-a}2\int_{a}^{b}g(\dot y,\dot y)\,\mathrm ds:
y=h(1,x), x\in{\mathcal D}, [a,b] \in {\mathcal I}_y\}.
\end{equation}

\begin{rem}\label{rem:7.0}
It is to pointed
out that any $\tfrac{(b-a)}2\int_a^bg(\dot y,\dot y)\,\text dt$ coincides with  the integral $\tfrac12\int_0^1g(\dot y_{a,b},\dot y_{a,b})\,
\mathrm dt$, where $y_{a,b}$ is the affine reparameterization of $y$ on the interval $[0,1]$.
\end{rem}

Given continuous maps $h_1:[0,1]\times F_1\to\mathfrak M$ and
$h_2:[0,1]\times F_2\to\mathfrak M$ such that $h_1(1,F_1)\subset
F_2$, then we define the \emph{concatenation} of $h_1$ and $h_2$ as
the continuous map $h_2\star h_1:[0,1]\times F_1\to\Lambda$ given by
\begin{equation}\label{eq:defconcatenationhomotop}
h_2\star h_1(t,x)=
\begin{cases}
h_1(2t,x),&\text{if\ } t\in[0,\tfrac12],\\
h_2(2t-1,h_1(1,x)),&\text{if\ } t\in[\tfrac12,1].
\end{cases}
\end{equation}

\section{A deformation lemma with homotopies of type A}\label{sec:first}
In this section we deal with homotopy of type A. From a technical point of view this is the more difficult case. For this reason
our approach to deformation lemmas starts with the function al $\mathcal F$ defined just on the class $\mathcal H$, instead of $\mathcal H_1$.

The proof of Proposition \ref{thm:prop4.17f} below is based on the construction of integral flows of vector fields in $\mathcal V^{-}$,
that can be carried on thanks to the results of Propositions
\ref{thm:51new} and \ref{thm:PS}.

It is important to point out that among the energy integrals used to define the functional $\mathcal F$
it is possible to have some energy integral that is not decreasing along the flow. This is the
reason of the use of Lemma \ref{thm:differenze}. It will allow us to obtain a global flow where $\mathcal F$ realizes
a smaller value.

\bigskip

Now for all $[a,b]\subset [0,1]$ let $\mathcal Z_{a,b}$ be as in \eqref{eq:defZab}.
\begin{prop}\label{thm:prop4.17f}
Fix $c_1 \in [\frac12(\frac{3\delta_0}{4K_{0}})^2,M_0[$ and $r > 0$. Then, there exists $\varepsilon_1=\varepsilon_1(r,c_1)\in ]0,\frac{c_1}4[$
such that for any $\varepsilon \in ]0,\varepsilon_1]$, for any $(\mathcal D,h) \in \mathcal H$  satisfying:
\begin{equation}\label{eq:4.28f}
\mathcal F(\mathcal D,h) \leq c_1 + \varepsilon
\end{equation}
and
\begin{multline}\label{eq:4.28ff}
\big\Vert y-x_{|[a,b]}\big\Vert_{a,b} \geq r \\
\text{ for any }y\in\mathcal Z_{a,b}, \text{ for any }
x = h(1,\xi),  \xi \in {\mathcal D}, \text{ such that }\\
\frac{b-a}2\int_{a}^{b}g(\dot x,\dot x)\,\mathrm ds \geq c_1-\varepsilon, \text{ where }[a,b] \in \mathcal I_x.
\end{multline}
there exists a continuous map
$H_\varepsilon:[0,1]\times h(1,\mathcal D)\to\mathfrak M$ of type A in $[0,1]$ with the following properties:
\smallskip

\begin{enumerate}
\item\label{itm:1lemdef1} $(\mathcal D,H_\varepsilon \star h)\in\mathcal H$;

\item\label{itm:1lemdef2}  $\mathcal F(\mathcal D,H_\varepsilon \star h)\le c_1 - \varepsilon$;

\item\label{itm:1lemdef3} there exists $T_\varepsilon>0$, with
$T_\varepsilon\to0$ as $\varepsilon\to0$, such
that
for any $z \in h(1,\mathcal D)$
$\Vert H_\varepsilon(\tau,z)-z\Vert_{a,b}
\le\tau T_\varepsilon$ for all $\tau\in[0,1]$, for all $[a,b] \in {\mathcal I}_z$.
\end{enumerate}
\end{prop}

For the proof of Proposition \ref{thm:prop4.17f} some preliminary results are needed. To state the first Lemma observe that, because of the compactness
$\phi^{-1}(]-\infty,\delta_0])$, analogously to \eqref{eq:5.11new}, there exists constant $K_1, K_2, K_3 \in \mathbb{R}^+$ such that
\begin{multline}\label{eq:kappas}
\vert g_x(v_1,v_1)-g_z(v_2,v_2)\vert\le K_1\Vert x-z\Vert_E\,{\Vert v_1\Vert_E}^2+K_2\Vert v_1-v_2\Vert_E\,\Vert v_1\Vert_E\\
+K_3\Vert v_1-v_2\Vert_E\,\Vert v_2 \Vert_E \text{ for any }x,z\in\phi^{-1}(]-\infty,\delta_0]) \text{ and }v_1,v_2 \in \mathbb{R}^N.
\end{multline}

Set

\begin{equation}\label{eq:kappa0}
\widehat K = \frac{K_1}{\ell_0}(M_0+1) + \frac{K_2}{\sqrt{\ell_0}}\sqrt{M_0+1} + \frac{K_3}{\sqrt{\ell_0}}\sqrt{M_0},
\end{equation}
where $\ell_0$ is given in
\eqref{eq:5.10bisnew}. We have

\begin{lem}\label{thm:differenze}
Fix $K_{*} > 0$ and $\rho_{*} = \frac{c_{1}K_{*}}{32M_0\widehat K}$. Fix $x \in \mathfrak M$ and $[a_x,b_x] \in \mathcal I_x$ (and therefore
$\int_{a_x}^{b_x}\frac12g(\dot x, \dot x)ds < M_0)$.
Take $[\alpha_x,\beta_x] \supset [a_x,b_x]$ and included in $[0,1]$ such that:
\begin{equation}\label{eq:seconda}
\int_{\alpha_x}^{a_x}\frac12g(\dot x, \dot x)ds \leq \frac{K_{*}}4, \int_{b_x}^{\beta_x}\frac12g(\dot x, \dot x)ds \leq \frac{K_{*}}4,
\int_{\alpha_x}^{\beta_x}\frac12g(\dot x, \dot x)ds \leq M_0 + 1
\end{equation}
\begin{equation}\label{eq:terza}
(a_x - \alpha_x)(M_0 + \frac{K_{*}}4) \leq  \frac{c_{1}K_{*}}{16M_0}, \; (\beta_x - b_x)(M_0 + \frac{K_{*}}4) \leq  \frac{c_{1}K_{*}}{16M_0}.
\end{equation}
Take $z \in \mathfrak M$ and $[a_z,b_z]$ satisfying:
\begin{equation}\label{eq:primaz}
[a_z,b_z] \subset [\alpha_x,\beta_x] \text{ and }[a_z,b_z] \cap [a_x,b_x] \neq \emptyset,
\end{equation}
\begin{equation}\label{eq:secondaz}
D(x,a_z,b_z,a_x,b_x) > K_{*} \text{ (cf. \eqref{eq:defDxalfabeta})},
\end{equation}
\begin{equation}\label{eq:terzaz}
\Vert z - x\Vert_{a_z,b_z} < \rho_{*},
\end{equation}
(so by \eqref{eq:1.10bis} we have also $\Vert z - x\Vert_{L^{\infty}([a_z,b_z],\mathbb{R}^N)} < \rho_{*}$).

Then the following property is satisfied:
\begin{equation}\label{eq:7.1ter}
\frac12\int_{a_z}^{b_z}g(\dot z,\dot z)\,\mathrm ds\le\frac12\int_{a_x}^{b_x}g(\dot x,\dot x)\,\mathrm ds -\frac{K_*}8.
\end{equation}
Moreover if
\begin{equation}\label{eq:prima}
(b_x - a_x)\int_{a_x}^{b_x}\frac12g(\dot x, \dot x)ds \geq \frac{c_1}2,
\end{equation}
we also have
\begin{equation}\label{eq:7.1bis}
\frac{b_z-a_z}2\int_{a_z}^{b_z}g(\dot z,\dot z)\,\mathrm ds\le\frac{b_x-a_x}2\int_{a_x}^{b_x}
g(\dot x,\dot x)\,\mathrm ds - \frac{c_1K_{*}}{32M_0}.
\end{equation}
\end{lem}

\begin{proof} By  \eqref{eq:5.10bisnew}, \eqref{eq:kappa0}, assumption \eqref{eq:terzaz} and H\"{o}lder inequality
\begin{multline*}
\left|\frac{1}{2}\int_{a_z}^{b_z}g(\dot x, \dot x)\mathrm{d}s -
\frac{1}{2}\int_{a_z}^{b_z}g(\dot z, \dot z)\mathrm{d}s\right| \leq \\
K_{1}\rho_{*}\int_{a_z}^{b_z}\Vert \dot x\Vert_{E}^2 + K_{2}\rho_{*}\sqrt{\int_{a_z}^{b_z}\Vert \dot x\Vert_{E}^2} +
K_{3}\rho_{*}\sqrt{\int_{a_z}^{b_z}\Vert \dot z\Vert_{E}^2} \leq \\
\frac{K_{1}\rho_{*}}{\ell_0}\int_{a_z}^{b_z}g(\dot x, \dot x) +
\frac{K_{2}\rho_{*}}{\sqrt{\ell_0}}\sqrt{\int_{a_z}^{b_z}g(\dot x, \dot x)}+
\frac{K_{3}\rho_{*}}{\sqrt{\ell_0}}\sqrt{\int_{a_z}^{b_z}g(\dot z, \dot z)}.
\end{multline*}
But $[a_z,b_z] \subset [\alpha_x,\beta_x]$ and $[a_z,b_z] \in \mathcal I_z$. Therefore by \eqref{eq:seconda} and \eqref{eq:kappa0},
\begin{equation}\label{eq:stimax-z}
\left|\frac{1}{2}\int_{a_z}^{b_z}g(\dot x, \dot x)\mathrm{d}s -
\frac{1}{2}\int_{a_z}^{b_z}g(\dot z, \dot z)\mathrm{d}s\right| \leq \widehat K\rho_*,
\end{equation}
from which we deduce
\begin{equation}\label{eq:stimaintermedia}
\frac{1}{2}\int_{a_z}^{b_z}g(\dot z, \dot z)\mathrm{d}s \leq
\frac{1}{2}\int_{a_z}^{b_z}g(\dot x, \dot x)\mathrm{d}s + \widehat K\rho_*.
\end{equation}
Now assumption \eqref{eq:secondaz} says that
\[
\frac{1}2\int_{I_{a_x,a_z}\cup I_{b_z,b_x}}g(\dot
x,\dot x)\,\mathrm ds > K_*,
\]
where $I_{a,b}$ denotes the interval $[a,b]$ if $b \geq a$ and the interval $[b,a]$ if $b < a$. Just to fix our ideas we can assume
\begin{equation}\label{eq:fixideas}
\frac{1}2\int_{I_{a_x,a_z}}g(\dot x,\dot x)\,\mathrm ds > \frac{K_*}2.
\end{equation}
In this case, thanks to \eqref{eq:seconda} and \eqref{eq:primaz} we deduce that $a_z \in [a_x,b_x]$. Then
\begin{multline*}
\frac{1}2\int_{a_z}^{b_z}g(\dot x,\dot x)\,\mathrm ds = \frac{1}2\int_{a_x}^{b_x}g(\dot x,\dot x)\,\mathrm ds + \\
- \frac{1}2\int_{a_x}^{a_z}g(\dot x,\dot x)\,\mathrm ds + \frac{1}2\int_{b_x}^{b_z}g(\dot x,\dot x)\,\mathrm ds \text { if }b_z \geq b_x
\end{multline*}
and
\begin{multline*}
\frac{1}2\int_{a_z}^{b_z}g(\dot x,\dot x)\,\mathrm ds = \frac{1}2\int_{a_x}^{b_x}g(\dot x,\dot x)\,\mathrm ds + \\
- \frac{1}2\int_{a_x}^{a_z}g(\dot x,\dot x)\,\mathrm ds - \frac{1}2\int_{b_z}^{b_x}g(\dot x,\dot x)\,\mathrm ds \text { if }b_z < b_x.
\end{multline*}
Then
\begin{equation}\label{eq:stimafinale}
\frac{1}2\int_{a_z}^{b_z}g(\dot x,\dot x)\,\mathrm ds \leq
\frac{1}2\int_{a_x}^{b_x}g(\dot x,\dot x)\,\mathrm ds
- \frac{1}2\int_{I_{a_x,a_z}}g(\dot x,\dot x)\,\mathrm ds +
\frac{1}2\int_{b_x}^{\beta_x}g(\dot x,\dot x)\,\mathrm ds,
\end{equation}
so combining \eqref{eq:stimaintermedia}--\eqref{eq:stimafinale} and using assumptions \eqref{eq:seconda} and \eqref{eq:terza} gives
\[
\frac{1}2\int_{a_z}^{b_z}g(\dot z,\dot z)\,\mathrm ds \leq \frac{1}2\int_{a_x}^{b_x}g(\dot x,\dot x)\,\mathrm ds -\frac{K_*}4 + \widehat K\rho_*,
\]
obtaining \eqref{eq:7.1ter} because of the choice of $\rho_*$ (recall that $c_1 < M_0$.)

In order to prove \eqref{eq:7.1bis} first observe that, by \eqref{eq:stimaintermedia},
\begin{equation}\label{eq:12.16bis}
(b_z-a_z)\frac{1}{2}\int_{a_z}^{b_z}g(\dot z, \dot z)\mathrm{d}s \leq
(b_z-a_z)\frac{1}{2}\int_{a_z}^{b_z}g(\dot x, \dot x)\mathrm{d}s + (b_z-a_z)\widehat K\rho_*.
\end{equation}
Moreover recall that by assumptions \eqref{eq:seconda} and \eqref{eq:primaz}, since we are assuming \eqref{eq:fixideas}, we have
\[
a_z \in [a_x,b_x].
\]
Then if $b_z \leq b_x$ we have $b_z-a_z \leq b_x-a_x$, so by \eqref{eq:12.16bis},
\[
(b_z-a_z)\frac{1}{2}\int_{a_z}^{b_z}g(\dot z, \dot z)\mathrm{d}s \leq (b_x-a_x)\frac{1}{2}\int_{a_z}^{b_z}g(\dot x, \dot x)\mathrm{d}s +
(b_x-a_x)\widehat K\rho_*,
\]
therefore by \eqref{eq:seconda} and \eqref{eq:stimafinale}
\[
(b_z-a_z)\frac{1}{2}\int_{a_z}^{b_z}g(\dot z, \dot z)\mathrm{d}s \leq (b_x-a_x)\frac{1}{2}\int_{a_x}^{b_x}g(\dot x, \dot x)\mathrm{d}s
+(b_x-a_x)\big(-\frac{K_*}2+ \frac{K_*}4 + \widehat K\rho_*\big),
\]
from which we deduce \eqref{eq:7.1bis} thanks to the choice of $\rho_*$ and property \eqref{eq:prima} that implies $b_x-a_x \geq \frac{c_1}{2M_0}$.

Now assume $b_z > b_x$. In this case we have
\begin{multline*}
(b_z-a_z)\frac{1}{2}\int_{a_z}^{b_z}g(\dot x, \dot x)\mathrm{d}s = (b_z-b_x)\frac{1}{2}\int_{a_z}^{b_z}g(\dot x, \dot x)\mathrm{d}s
+ (b_x-a_z)\frac{1}{2}\int_{a_z}^{b_z}g(\dot x, \dot x)\mathrm{d}s \leq\\
(\beta_x -b_x)\left(\frac{1}{2}\int_{a_x}^{b_x}g(\dot x, \dot x)\mathrm{d}s + \frac{1}{2}\int_{b_x}^{\beta_x}g(\dot x, \dot x)\mathrm{d}s\right)+
(b_x-a_x)\frac{1}{2}\int_{a_z}^{b_z}g(\dot x, \dot x)\mathrm{d}s \leq\\
\frac{c_1K_*}{16M_0} + (b_x-a_x)\frac{1}{2}\int_{a_z}^{b_z}g(\dot x, \dot x)\mathrm{d}s
\end{multline*}
by \eqref{eq:seconda} and \eqref{eq:terza}.

Then by \eqref{eq:12.16bis} we have
\[
(b_z-a_z)\frac{1}{2}\int_{a_z}^{b_z}g(\dot z, \dot z)\mathrm{d}s \leq \frac{c_1K_*}{16M_0} +
(b_x-a_x)\frac{1}{2}\int_{a_z}^{b_z}g(\dot x, \dot x)\mathrm{d}s + \widehat K\rho_*,
\]
so by \eqref{eq:fixideas}, \eqref{eq:stimafinale} and assumption \eqref{eq:seconda},
\[
(b_z-a_z)\frac{1}{2}\int_{a_z}^{b_z}g(\dot z, \dot z)\mathrm{d}s \leq \frac{c_1K_*}{4M_0} + \widehat K\rho_* +
(b_x-a_x)\left(\frac{1}{2}\int_{a_x}^{b_x}g(\dot x, \dot x)\mathrm{d}s - \frac{K_*}4\right).
\]
Therefore, recalling that $b_x - a_x \geq \frac{c_1}{2M_0}$, it is
\[
(b_z-a_z)\frac{1}{2}\int_{a_z}^{b_z}g(\dot z, \dot z)\mathrm{d}s \leq \frac{c_1K_*}{16M_0} + \widehat K\rho_* - \frac{c_1K_*}{8M_0}+
(b_x-a_x)\frac{1}{2}\int_{a_x}^{b_x}g(\dot x, \dot x)\mathrm{d}s,
\]
giving \eqref{eq:7.1bis} because of the choice of $\rho_*$.
\end{proof}

\begin{rem}\label{rem:spiegazione}
In the proof of Proposition \ref{thm:prop4.17f} we shall use the above Lemma with $K_{*}=E(r)$, in order to have that
conditions \eqref{eq:7.1ter} and \eqref{eq:7.1bis} are satisfies if
\[
D(x,a_z,b_z,a_x,b_x) > E(r).
\]
In this case it may happen that
$\frac{1}{2}\int_{a_z}^{b_z}g(\dot z, \dot z)\mathrm{d}s$ and
$(b_z-a_z)\frac{1}{2}\int_{a_z}^{b_z}g(\dot z, \dot z)\mathrm{d}s$ are not decreasing along the flow that we shall construct by means of the
vector filed $V_x$ given by Proposition \ref{thm:PS}. But condition \eqref{eq:7.1ter}  says that
$\frac{1}{2}\int_{a_z}^{b_z}g(\dot z, \dot z)\mathrm{d}s$ is much less than $M_0$, while condition \eqref{eq:7.1bis} says that
$(b_z-a_z)\frac{1}{2}\int_{a_z}^{b_z}g(\dot z, \dot z)\mathrm{d}s$ is much less then $(b_x-a_x)\frac{1}{2}\int_{a_x}^{b_x}g(\dot x, \dot x)\mathrm{d}s$,
and this will be useful to get a flow in the space $\mathfrak M$ where $\mathcal F$ reaches smaller values.
\end{rem}

\begin{defin}\label{eq:collezione-monotona}
Fix $\tau_0 \in [0,\delta]$. The set
\[\mathcal C_{\tau_0,\delta} = \{[a_\tau,b_\tau]\subset[0,1] : \tau_0\leq\tau_1\leq\tau_2\leq\delta \Longrightarrow [a_{\tau_2},b_{\tau_2}] \subset [a_{\tau_1},b_{\tau_1}]\}
\]
is called {\sl non increasing intervals collection} (NIIC).
\end{defin}

\begin{lem}\label{thm:controllocrescita}
Let $W$ be a locally Lipschitz continuous vector field in $H^1([0,1],\mathbb{R}^N)$, $\delta > 0$ and $\eta_z:[0,\delta] \rightarrow \mathfrak M_0$ be the solution
of the Cauchy problem
\begin{equation}\label{eq:testCauchy}
\left\{\begin{array}{l}\frac{\mathrm d}{\mathrm d\tau}\eta_z=
W(\eta_z) \\[0.5cm]
\eta_z(0)=z, z\in \mathfrak M_0. \end{array} \right.
\end{equation}
Fix $\tau_0 \in [0,\delta]$ and a NIIC $\mathcal C_{\tau_0,\delta} = \{[a_\tau,b_\tau] : \tau \in [\tau_0,\delta] \}$ such that
\begin{equation}\label{eq:ipotesi-energia}
\int_{a_\tau}^{b_\tau}\frac12g(\dot{\eta}_{z}(\tau),\dot{\eta}_{z}(\tau))ds < M_0, \text{ for any }[a_\tau,b_\tau] \in C_{\tau_0,\delta}
\end{equation}
where $\dot{\eta}_{z}(\tau)$ denotes the derivative with respect to $\tau$ of the curve $\eta_{z}(\tau)$, and
\begin{equation}\label{eq:prima-ipotesi}
\Vert W(\eta_{z}(\tau))\Vert_{a_\tau,b_\tau} \leq 1 \text{ for any }\tau \in [\tau_0,\delta].
\end{equation}
Define
\begin{equation}\label{eq:Fflusso}
F_z(\tau,\mathcal C_{\tau_0,\delta})= \frac{1}{2}\int_{a_\tau}^{b_\tau}
g(\dot{\eta}_{z}(\tau),\dot{\eta}_{z}(\tau))\,\text ds, \,
[a_\tau,b_\tau]\in  C_{\tau_0,\delta},
\end{equation}
and
\begin{equation}\label{eq:Gflusso}
G_z(\tau,\mathcal C_{\tau_0,\delta})=  \frac{b_\tau-a_\tau}{2}\int_{a_\tau}^{b_\tau}
g(\dot{\eta}_{z}(\tau),\dot{\eta}_{z}(\tau))\,\text ds, \,
[a_\tau,b_\tau]\in  C_{\tau_0,\delta}.
\end{equation}

Then for any $\tau \in ]\tau_0,\delta]$ we have:
\begin{equation}\label{eq:stimaFflusso}
F_z(\tau,C_{\tau_0,\delta}) \leq F_z(\tau_0,C_{\tau_0,\delta}) + L_1(\tau-\tau_0)\sqrt{2M_0L_0}
\end{equation}
and
\begin{equation}\label{eq:stimaGflusso}
G_z(\tau,C_{\tau_0,\delta}) \leq G_z(\tau_0,C_{\tau_0,\delta})  + L_1(\tau-\tau_0)\sqrt{2M_0L_0}
\end{equation}
where $L_0,L_1$ are defined by \eqref{eq:5.10bisnew} and \eqref{eq:5.12new}
respectively.
\end{lem}

\begin{proof} We shall proof the statement for the map $G$, since the proof for the map $F$ is analogous and simpler.

Take  $\tau > \tau_{0}$, and set $y_{\tau} = \eta_{z}(\tau)$ and $y_{\tau_{0}} =
\eta_{z}(\tau_{0})$. Since $b_{\tau_{0}} - a_{\tau_{0}} \leq 1$, and $[a_\tau,b_\tau] \subset [a_{\tau_0},_{\tau_0}]$,
using  \eqref{eq:fab} we have, for some
$\theta \in ]\tau_{0},\tau[$):
\begin{multline*}
\frac{b_{\tau} - a_{\tau}}{2}\int_{a_{\tau}}^{b_{\tau}}g({\dot y}_{\tau},{\dot y}_{\tau})\mathrm ds \leq
\frac{b_{\tau_{0}} - a_{\tau_{0}}}{2}\int_{a_{\tau}}^{b_{\tau}}g({\dot y}_{\tau},{\dot y}_{\tau})\mathrm ds
= (b_{\tau_{0}} - a_{\tau_{0}})f_{a_{\tau},b_{\tau}}\big(\eta_{z}(\tau)\big) \\
=(b_{\tau_{0}} - a_{\tau_{0}}) \left[(\tau-\tau_{0})\mathrm df_{a_{\tau},b_{\tau}}\big(\eta_{z}(\theta)\big)\left[\frac{\mathrm d\eta_{z}(\theta)}
{\mathrm d\tau}\right] + f_{a_{\tau},b_{\tau}}(\eta_{z}(\tau_{0})) \right] \\
\leq (b_{\tau_{0}} - a_{\tau_{0}})f_{a_{\tau_0},b_{\tau_0}}(\eta_{z}(\tau_0)) +
(\tau-\tau_{0})
\Big|\mathrm df_{a_{\tau},b_{\tau}}\big(\eta_{z}(\theta)\big)\left[\frac{\mathrm d\eta_{z}(\theta)}
{\mathrm d\tau}\right]\Big| \\
= \frac{b_{\tau_{0}} - a_{\tau_{0}}}{2}\int_{a_{\tau_0}}^{b_{\tau_0}}g({\dot y}_{\tau_0},{\dot y}_{\tau_0})\mathrm ds +
(\tau-\tau_{0})\Big|\mathrm df_{a_{\tau},b_{\tau}}(\eta_{z}(\theta))\left[\frac{\mathrm d\eta_{z}(\theta)}
{\mathrm d\tau}\right]\Big|.
\end{multline*}

Note that, by
\eqref{eq:testCauchy} and H\"{o}lder's inequality:
\begin{multline*}
\Big|\mathrm df_{a_{\tau},b_{\tau}}\big(\eta_{z}(\theta)\big)\left[\frac{\mathrm d\eta_{z}(\theta)}
{\mathrm d\tau}\right]\Big| \leq \int_{a_{\tau}}^{b_{\tau}}\Big|g\Big(\frac{\mathrm d\eta_{z}(\theta)}
{\mathrm ds}, \tfrac{\mathrm D}{\mathrm ds}W\Big)\Big| \mathrm ds \\
\leq \left(\int_{a_{\tau}}^{b_{\tau}}g\Big(\frac{\mathrm d\eta_{z}(\theta)}{\mathrm ds},
\frac{\mathrm d\eta_{z}(\theta)}{\mathrm ds}\Big)\mathrm ds\right)^{\tfrac12}
\left(\int_{a_{\tau}}^{b_{\tau}}g\big(\tfrac{\mathrm D}{\mathrm ds}W,\tfrac{\mathrm D}{\mathrm ds}W\big)\mathrm ds\right)^{\tfrac12}.
\end{multline*}
Now
\[[a_{\tau},b_{\tau}] \subset [a_{\theta},b_{\theta}] \subset [a_{\tau_{0}},b_{\tau_{0}}], \]
therefore by \eqref{eq:ipotesi-energia} we have

\begin{multline*}
\left(\int_{a_{\tau}}^{b_{\tau}}g\Big(\frac{\mathrm d\eta_{z}(\theta)}{\mathrm ds},
\frac{\mathrm d\eta_{z}(\theta)}{\mathrm ds}\Big)\mathrm ds\right)^{\tfrac12}\\
\leq
\left(\int_{a_{\theta}}^{b_{\theta}}g\Big(\frac{\mathrm d\eta_{z}(\theta)}{\mathrm ds},
\frac{\mathrm d\eta_{z}(\theta)}{\mathrm ds}\Big)\mathrm ds\right)^{\tfrac12} < \sqrt{2M_{0}},
\end{multline*}
while by \eqref{eq:5.10bisnew}, \eqref{eq:5.12new} and \eqref{eq:prima-ipotesi},
\begin{multline*}
\left(\int_{a_{\tau}}^{b_{\tau}}g\big(\tfrac{\mathrm D}{\mathrm ds}W,\tfrac{\mathrm D}{\mathrm ds}W\big)\mathrm ds\right)^{\tfrac12}
\leq
\left(\int_{a_{\tau}}^{b_{\tau}}L_{0}\Vert\tfrac{\mathrm D}{\mathrm ds}W\Vert_{E}^{2}\right)^{\tfrac12}\\
\leq L_1\sqrt{L_0}\Vert W\Vert_{a_{\tau},b_{\tau}}
\leq L_{1}\sqrt{L_0}.
\end{multline*}

Then
\[
\Big| \mathrm df_{a_{\tau},b_{\tau}}\big(\eta_{z}(\theta)\big)\left[\frac{\mathrm d\eta_{z}(\theta)}
{\mathrm d\tau}\right]\Big| \leq L_1\sqrt{2M_0L_0}
\]
from which we finally deduce \eqref{eq:stimaGflusso}.

\end{proof}

\begin{lem}\label{thm:palle}
Fix $x,z \in H^{1}([0,1],\mathbb{R}^N)$, $\hat \rho > 0$ and $[a,b]\subset [0,1]$ such that
\[
\Vert z-x \Vert_{a,b} < \hat \rho.
\]
Then there exists $\rho = \rho(x,z,a,b)$ such that, for any $w \in H^{1}([0,1],\mathbb{R}^N)$ such that $\Vert z-w \Vert_{a,b} < \rho$,
and for any $[\tilde a,\tilde b] \subset [a,b]$
it is $\Vert w-x \Vert_{\tilde a,\tilde b} < \hat \rho$.
\end{lem}
\begin{proof}
\[
\Vert w-x \Vert_{\tilde a,\tilde b}  \leq \Vert w-x \Vert_{a,b} \leq \Vert w-z \Vert_{a,b} + \Vert z-x \Vert_{a,b} < \rho + \Vert z-x \Vert_{a,b},
\]
from which we deduce the thesis.
\end{proof}

\begin{proof}[Proof of Proposition \ref{thm:prop4.17f}]
The proof is divided into two parts: first part requires the construction of the vector field (with the help of Propositions
\ref{thm:51new} and \ref{thm:PS}) generating the flow that we shall study in the second part of the proof.

\bigskip
{\bf Part A: Construction of the vector field.}

Fix $\varepsilon \in ]0,\varepsilon_1], \varepsilon_1 < \frac{c_1}4$.
Let $(\mathcal D,h) \in \mathcal H$, $x = h(1,\xi) $ with $\xi \in \mathcal D$. By assumption \eqref{eq:4.28ff} it is:
\begin{equation}\label{eq:P1}
\Vert y - x_{|[a_x,b_x]}\Vert_{a_x,b_x} \geq r \text{ for any }y \in \mathcal Z_{a_x,b_x},
\end{equation}
for any $[a_x,b_x] \in \mathcal I_x$ such that,
\begin{equation}\label{eq:P2}
(b_x - a_x)\int_{a_x}^{b_x}\frac12g(\dot x, \dot x)ds \geq c_1-\varepsilon,
\end{equation}

Note that by the definition of $\mathfrak M$
we also have
\begin{equation}\label{eq:numeroter}\int_{a_x}^{b_x}g(\dot x,\dot x) ds < M_0, \text{ for any }[a_x,b_x] \in \mathcal I_x.
\end{equation}

For any $x \in h(1,\mathcal D)$ we shall consider the following class of intervals:
\[\mathcal I_{x,r,\varepsilon} = \{[a,b] \in \mathcal I_x: [a,b] \text{ satisfies }\eqref{eq:P1} \text{ and }\eqref{eq:P2}\}.\]

Now for any $x \in h(1,\mathcal D)$ we choose a vector field $V_x$ as follows.
Consider the quantities $\mu(r),\theta(r)$ and $\kappa(r)$ given in
Proposition~\ref{thm:51new}. If $x = h(1,\xi)$ on any $[a_x,b_x] \in \mathcal I_{x,r,\varepsilon}$
we choose as $V_x$  the vector field defined in $[a_x -\Delta(x,[a_x,b_x]),b_x + \Delta(x,[a_x,b_x])]$ and given by Proposition \ref{thm:PS}
(satisfying \eqref{itm:51new-bis}, \eqref{itm:51new-a} and \eqref{itm:51new-b} of Proposition \ref{thm:51new} with $[a,b]$
replaced by $[a_x,b_x]$) and such that
\begin{equation}\label{eq:perorasichia}
\Vert V_x\Vert_{a_x,b_x}=\tfrac12.
\end{equation}
If $[a_x,b_x] \in \mathcal I_x \setminus \mathcal I_{x,r,\varepsilon}$ we choose $V_x=0$ on $[a_x,b_x]$. Clearly we can choose $V_{\mathcal R x}$ so that
\begin{equation}\label{eq:Rscelta}
V_{\mathcal R x}(s) = V_x(1-s) = \mathcal R V_x(s) \text{ for any }s \in [0,1].
\end{equation}
Indeed we have $a_{\Rcal x} = 1-b_x$ and $b_{\Rcal x} = 1-a_x$.
\medskip

Consider also $\rho(r)$ be as in Proposition \ref{thm:PS}. Choose $K_*=E(r)$,
$\rho_{*} = \frac{c_{1}E(r)}{32M_0\widehat K}$ and
\[
\rho_0(r) = \min\{\rho(r),\rho_{*}\}.
\]
Now by Lemma \ref{thm:lem2.3} the number of intervals $[a_x,b_x] \in \mathcal I_{x,r,\varepsilon}$ is finite and bounded by a constant independent of $x$. Then
for any $x \in h(1,\mathcal D)$ and $[a_x,b_x] \in \mathcal I_{x,r,\varepsilon}$, we can choose $\Delta(x)$ in Proposition \ref{thm:PS},
independently by $[a_x,b_x] \in \mathcal I_{x,r,\varepsilon}$,
intervals $[\alpha_x,\beta_x]$, $[\alpha_x',\beta_x']$
corresponding to $[a_x,b_x] \in \mathcal I_{x,r,\varepsilon}$
verifying
\[
\alpha_{\Rcal x} = 1-\beta_x, \alpha_{\Rcal x}' = 1-\beta_x', \beta_{\Rcal x} = 1-\alpha_x, \beta_{\Rcal x}' = 1-\alpha_x',
\]
and
\[
[a_x-\Delta(x),b_x+\Delta(x)]\cap[0,1] \supset [\alpha_x,\beta_x] \supset  [\alpha_x',\beta_x'] \supset [a_x,b_x]
\]
and a number
\[
\rho(x) \in ]0,\frac12\rho_{0}(r)]
\]
such that the following properties are satisfied:
\begin{itemize}
\item
if $[a_{x}^{i},b_{x}^{i}]$ are the intervals in $\mathcal I_{x,r,\varepsilon}$, and $[\alpha_{x}^{i},\beta_{x}^{i}]$
are the corresponding intervals above including $[a_{x}^{i},b_{x}^{i}]$, it is
\[
i \neq j \Longrightarrow [\alpha_{x}^{i},\beta_{x}^{i}] \cap [\alpha_{x}^{j},\beta_{x}^{j}] = \emptyset,
\]
\item $\alpha_x$ and $\beta_x$ satisfy \eqref{eq:seconda} and \eqref{eq:terza} of Lemma \ref{thm:differenze} with $K_{*}=E(r)$,
\item if $a_x > 0$, then $\alpha_x < \alpha_x' < a_x$, while if $a_x=0$ then $\alpha_x = \alpha_x' = 0$,
\item if $b_x < 1$, then $\beta_x > \beta_x' > b_x$, while if $b_x=1$ then $\beta_x = \beta_x' = 0$,
\item $\sup\{\phi(x(s): s \in [\alpha_x,\beta_x]\} \leq \frac{\delta_0}4$ (recall that $\sup\{\phi(x(s): s \in [a_x,b_x]\}=0$),
\item $\inf\{\phi(x(s): s \in [\alpha_x,\beta_x]\setminus[a_x,b_x]\} \geq -\frac{\kappa(r)}2$,
\item $\phi(x(s)) > 0$ for any $s \in [\alpha_x,\beta_x]\setminus [\alpha_x',\beta_x']$,
\end{itemize}
and for any $z \in h(1,\mathcal D)$ also the following properties are satisfied:
\begin{itemize}
\item $\Vert x - z \Vert_{0,1} < \rho(x)$ implies $\int_{a_z'}^{b_z'}\frac12g(\dot z, \dot z)ds < \frac{c_1}2 < c_1 - \varepsilon$,
for any $[a_z',b_z'] \subset [a_z,b_z]\in \Ical _z,  [a_z',b_z'] \subset  [\alpha_x,\beta_x]\setminus]a_x,b_x[$,
\item if $\Vert x - z \Vert_{0,1} < \rho(x)$ and $x = h(1,\xi)$
has intervals in $\mathcal I_{x,r,\varepsilon}$, then $z$ is not constant,
\item $\Vert x - z \Vert_{L^{\infty}([0,1],\mathbb{R}^N)} < \rho(x)$ implies
$\sup\{\phi(z(s): s \in [\alpha_x,\beta_x]\} \leq \frac{\delta_0}2$,
\item $\Vert x - z \Vert_{L^{\infty}([0,1],\mathbb{R}^N)} < \rho(x)$ implies
$\inf\{\phi(x(s): s \in [\alpha_x,\beta_x]\setminus[a_x,b_x]\} \geq -\kappa(r)$,
\item $\Vert x - z \Vert_{L^{\infty}([0,1],\mathbb{R}^N)} < \rho(x)$ implies $\phi(z(s)) > 0$ for any $s \in [\alpha_x,\beta_x]\setminus
[\alpha_x',\beta_x']$, whenever $[\alpha_x,\beta_x] \supset [a_x,b_x]$ with $[a_x,b_x] \in \mathcal I_{x,r,\varepsilon}$.
\end{itemize}

Observe that, by the last property above, it is
\begin{multline}\label{eq:propertyintersezione}
\Vert x-z\Vert_{L^\infty\big([0,1],\R^N\big)}<\rho(x),
[a_z,b_z]\in{\mathcal I}_{z}, [a_z,b_z]\cap[\alpha_x,\beta_x]\ne\emptyset,  \text{ and } [a_x,b_x] \in \mathcal I_{x,r,\varepsilon} \\
\text{ imply } [a_z,b_z]\subset[\alpha_x',\beta_x'].
\end{multline}

\bigskip

Define now a new vector field $\widehat V_x$ as follows.
Suppose $\alpha_x < a_x < b_x < \beta_x$ Let $\varphi_{x}:\R\to[0,1]$ be the continuous piecewise affine function which is zero
outside any $[\alpha_x,\beta_x]$ as above and it is equal to $1$ in any $[\alpha_x',\beta_x']$ as above.
We set
\begin{equation}\label{eq:Vcap}
\widehat V_x(s)=\varphi_{x}(s) V_x(s) \text{ for any }s \in [\alpha_x,\beta_x] \text{ as above }, \text { and $0$ otherwise}.
\end{equation}
If $a_x = 0$ we choose $\varphi_x$ so that $\varphi_x(0)=1$, while if $b_x=1$ we choose $\varphi_x$ so that $\varphi_x(1)=1$.

A straightforward computation shows that there exists $C_x > 0$ such that
\begin{equation}\label{eq:4.32fnew}
\Vert  \widehat V_x(s) \Vert_{0,1}\le C_x,
\end{equation}
and thanks to \eqref{eq:1.10bis}, \eqref{eq:perorasichia} and  the definition of $\varphi_x$ we see that
\begin{equation}\label{eq:elleinfinito}
\quad\Vert \widehat V_x(s)
\Vert_{L^\infty\big([0,1],\R^N\big)}\le \frac12.
\end{equation}

\smallskip
For any $x \in h(1,\mathcal D)$ consider the open subset of $h(1,\Dcal)$ given by
\begin{equation}\label{eq:Dintorno}
\Ucal _{x} =\{z \in h(1,\Dcal) : \Vert z - x\Vert_{0,1} < \rho(x)\}.
\end{equation}

\smallskip

Consider the open covering $\{\Ucal_{x},\Ucal_{\Rcal x}\}_{x \in h(1,\Dcal)}$ of $h(1,\Dcal)$.
Since $h(1,\Dcal)$ is compact there exists a finite covering $\{\Ucal_{x_i},\Ucal_{\Rcal x_i}\}_{i=1,\ldots,k}$ of $h(1,\Dcal)$.

Define, for any $z \in h(1,\Dcal)$,
\[
\varrho_{x_i}(z) = \dist_{0,1}(z,h(1,\Dcal) \setminus \Ucal_{x_i}).
\]
where $\dist_{0,1}$ is the distance induced by the norm $\Vert \cdot \Vert_{0,1}$. Since $\Rcal h(1,\Dcal) = \Dcal$,
$\Rcal \Ucal_{x_i} =\Ucal_{{\Rcal x}_i}$ and $\Rcal \circ \Rcal$ is the identity map, it is
\begin{equation}\label{eq:invarianzadirho}
\varrho_{x_i}(\Rcal z) = \varrho_{\Rcal x_i}(z).
\end{equation}
Finally set
\[
\beta_{x_i}(z) =
\frac{\varrho_{x_i}(z)}{\sum_{j=1,\ldots,k}(\rho_{x_j}(z)+\varrho_{\Rcal x_j}(z))}.
\]
Note that by \eqref{eq:invarianzadirho}
\begin{equation}\label{eq:invarianzadibeta}
\beta_{x_i}(\Rcal z) = \beta_{\Rcal x_i}(z)
\end{equation}
and that $\{\beta_{x_i},\beta_{\Rcal x_i}\}_{i=1,\ldots,k}$ is a continuous partition of the unity. Indeed
\begin{equation}\label{eq:betapartizione}
\sum_{j=1,\ldots,k}(\beta_{x_j}(z)+\beta_{\Rcal x_j}(z)) = 1.
\end{equation}
The vector field that we want to construct is defined as
\begin{equation}\label{eq:vectorfield}
W^{\lambda}(z,\eta) = \left(\sum_{j=1,\ldots,k}(\beta_{x_j}(z)\widehat V_{x_j}+\beta_{\Rcal x_j}(z)\widehat V_{\Rcal x_j}) +
\lambda\nabla\phi(\eta)\right)\chi(\phi(\eta)),
\end{equation}
where $\lambda \in ]0,1], z \in h(1,\Dcal), \eta \in \mathfrak M_0$, and
\begin{equation}\label{eq:chi}
\chi \in C{^2}(\mathbb{R},[0,1]) \text{ satisfies: } \chi(t)=1 \, \forall t \leq \frac{\delta_0}2, \, \chi(t)=0 \, \forall t \geq \frac{3\delta_0}4.
\end{equation}
Note that by the definition of ${\widehat V_{x}}$ , the choice
of $\rho(x)$, the definition of $\chi$,  \eqref{eq:1.10bis}, \eqref{eq:1.9fabio}, \eqref{eq:elleinfinito} and \eqref{eq:5.10bisnew} we have

\begin{equation}\label{eq:7.11ter}
\Vert W^{\lambda}(z,\eta)\Vert_{L^{\infty}([0,1],\R^{N})} \leq \frac12 + \lambda\frac{K_0}{\sqrt{\ell_0}} \leq 1,
\forall \lambda \in ]0,\frac{\sqrt{\ell_0}}{2K_0}],
\end{equation}
and there exists constants $C_1,C_2 \in \mathbb{R}^+$ such that
\begin{equation}\label{eq:7.15bis}
\Vert W^{\lambda}(z,\eta) \Vert_{0,1} \leq C_1+\lambda C_2\Vert \eta \Vert_{0,1}.
\end{equation}

Moreover, by \eqref{eq:Rscelta}, the $\Rcal$--equivariant choice of $\alpha_x,\alpha_x',\beta_x',\beta_x$
and the definition of the map $\varphi_x$, we have ${\widehat V}_{\Rcal x} = \Rcal {\widehat V}_x$, so by
\eqref{eq:invarianzadibeta}
\begin{equation}\label{eq:invarianzadiW}
\Rcal W^{\lambda}(z,\eta) = W^{\lambda}(\Rcal z,\Rcal\eta).
\end{equation}

The existence of the homotopy $H_\varepsilon$ will be proved studying the flow given by
the solution of the initial value problem
\begin{equation}\label{eq:7.12}
\left\{\begin{array}{l}\frac{\mathrm d}{\mathrm d\tau}\eta=
W^{\lambda}(z,\eta) \\[0.5cm]
\eta(0)=z \in h(1,\Dcal). \end{array} \right.
\end{equation}
Note that by \eqref{eq:7.15bis} the solution $\eta(\tau,z)$ of \eqref{eq:7.12} is defined for all $\tau \in \mathbb{R}^+$,
taking values on $H^1([0,1],\mathbb{R}^N$) for any $z \in h(1,\Dcal)$.

\bigskip

{\bf Part B: Study of the flow $\eta(\tau,z)$.}

\noindent
{\bf Step 1.} The flow $\eta$ is $\Rcal $--equivariant, namely
\[
\eta(\tau,\Rcal z) = \Rcal \eta(\tau,z).
\]
\begin{proof}
It is sufficient to observe that $\Rcal \eta(\tau,z)$ solves the Cauchy problem
\[
\left\{\begin{array}{l}\frac{\mathrm d}{\mathrm d\tau}\eta=
W^{\lambda}(\Rcal z,\eta) \\[0.5cm]
\eta(0)=\Rcal z. \end{array} \right.
\]
But this is a simple consequence of \eqref{eq:invarianzadiW} and the $\Rcal$--equivariance of $h$.
\end{proof}
\noindent
{\bf Step 2.} If $z$ is a constant curve, then
$\eta(\tau,z)$ is a constant curve for all
$\tau \geq 0$.
\begin{proof}
If $z$ is a constant curve, by
the choice of $\rho(x)$  and \eqref{eq:Vcap} we have:
\[
z \in \Ucal_{x_i} \Longrightarrow \widehat  V_{x_i} \equiv 0, \text{ and }\widehat V_{\Rcal x_i} \equiv 0.
\]
Then $\eta$ solves
\[
\left\{\begin{array}{l}\frac{\mathrm d}{\mathrm d\tau}\eta=
\lambda\nabla\phi(\eta)\chi(\phi(\eta)) \\[0.5cm]
\eta(0)=z. \end{array} \right.
\]
Denote by $\eta_\tau$ the integral flow in $\mathbb{R}^N$ of the above Cauchy problem, where the constant curve $z$ can be identified with a vector
in $\mathbb{R}^N$. By the local uniqueness property we deduce that, for any $\tau \geq 0$, $\eta(\tau,z)$ is the constant curve $\eta_\tau$.
\end{proof}

{\bf Step 3.} Fix $\lambda \in ]0,\frac{\sqrt{\ell_0}}{2K_0}]$. Then for any $z \in \Ucal_{x_i}$ if $\tau_* \in \mathbb{R}^+$ is the first instant
such that $\Vert \eta(\tau,z) - x_i \Vert_{L^{\infty}([0,1],\R^{N})} = \rho(r)$ (where $\rho(r)$ is given by Proposition
\ref{thm:PS}), then $\tau_* \geq \frac{\rho(r)}2$.
\begin{proof}
Since $\rho(x_i) \leq \frac{\rho(r)}2$ if $z \in \Ucal_{x_i}$, by \eqref{eq:7.11ter} and
\eqref{eq:7.12} we have
\begin{multline*}
\frac{\rho(r)}2 \leq \Vert \eta(\tau_*,z) - x_i \Vert_{L^{\infty}([0,1],\R^{N})} -
\Vert \eta(0,z) - x_i \Vert_{L^{\infty}([0,1],\R^{N})} \leq \\
\Vert \eta(\tau_*,z) - \eta(0,z)\Vert_{L^{\infty}([0,1],\R^{N})}\leq
\int_{0}^{\tau_*} \Vert \frac{\mathrm d}{\mathrm d\sigma}\eta(\sigma,z)\Vert_{L^{\infty}([0,1],\R^{N})}\mathrm d\sigma \leq \tau_*.
\end{multline*}
\end{proof}

{\bf Step 4.} Fix $\lambda \in ]0,\frac{\sqrt{\ell_0}}{2K_0}]$. Then, for any $\tau \in [0,\frac{\rho(r)}2]$, $\eta(\tau,z)$ satisfies
\eqref{unico} and \eqref{doppio}of Definition \ref{thm:tipoA}.
\begin{proof}
To prove \eqref{unico} of Definition \ref{thm:tipoA}, first of all observe that the curve
$[0,1]\ni s\mapsto\eta(\tau,z)(s)\in\R^N$ is of class $H^1$, and the map
$\tau\mapsto\eta(\tau,z)(\cdot)$
is of class $C^1$ and taking values in $H^1\big([0,1],\R^N\big)$.

Therefore, for any $s \in [0,1]$ we have:

\begin{multline*}
\frac\partial{\partial\tau}\,\phi(\eta(\tau,z)(s))\;
=g\big(\nabla\phi(\eta(\tau,z)(s)),\frac\partial{\partial\tau}\,
\eta(\tau,z)(s)\big)= \\
g\big(\nabla\phi(\eta(\tau,z)(s)),W^{\lambda}(\eta(\tau,z)(s))\big) =\\
g\big(\nabla\phi(\eta(\tau,z)(s),
\sum_{j=1}^{k}\beta_{x_j}(z)\varphi_{x_j}(s)V_{x_j}(s)\big)\chi(\phi(\eta(\tau,z)(s))+\\
g\big(\nabla\phi(\eta(\tau,z)(s),
\sum_{j=1}^{k}\beta_{\Rcal x_j}(z)\varphi_{\Rcal x_j}(s)V_{\Rcal x_j}(s)\big)
\chi(\phi(\eta(\tau,z)(s))) +\\
\lambda g\big(\nabla\phi(\eta(\tau,z)(s)),\nabla\phi(\eta(\tau,z)(s))\big)\chi(\phi(\eta(\tau,z)(s))).
\end{multline*}

Now $\phi = 0$ implies $g(\nabla \phi, \nabla \phi) > 0$ and $\chi(\phi)=1$. Therefore, since $\lambda > 0$, we have
\begin{multline*}
\phi(\eta(\tau,z)(s)) = 0 \Longrightarrow \frac\partial{\partial\tau}\,\phi(\eta(\tau,z)(s)) >\\
\sum_{j=1}^{k}\beta_{x_j}(z)\varphi_{x_j}(s)g(\nabla\phi(\eta(\tau,z)(s)),V_{x_j}(s))+\\
\sum_{j=1}^{k}\beta_{\Rcal x_j}(z)\varphi_{\Rcal x_j}(s)g(\nabla\phi(\eta(\tau,z)(s)),V_{\Rcal x_j}(s)).
\end{multline*}

Then by Step 3, \eqref{itm:PS15a} of Proposition \eqref{thm:PS} and the choice of $V_{x_i}$ and $V_{\Rcal x_i}$, we obtain the proof
since, for any $j=1,\ldots,k$, $\varphi_{x_j}, \varphi_{\mathcal R x_{j}} \geq 0$ and  $\beta_{x_j}, \beta_{\Rcal x_j} \geq 0$. The proof of property \eqref{doppio} is completely analogous.
\end{proof}

Set
\begin{equation}\label{eq:No}
N_0 =  \sup\{\Vert H^{\phi}(y)\Vert : y \in \phi^{-1}(]-\infty,\delta_0])\},
\end{equation}
where $\Vert H^{\phi}(y)\Vert$ is the norm induced by the metric $g$ of the hessian of the map $\phi$ at the point $y$.
\begin{rem}
Let $[a_\tau,b_\tau] \in \mathcal I_{\eta(\tau,z)}$ and $\eta(\tau,z) \in \mathfrak M$. Since, by step 4, $[a_\tau,b_\tau] \subset
[a_z,b_z]$, by the definition of $W^\lambda$ we have
\[
\Vert W^{\lambda}(z,\eta(\tau,z)) \Vert_{a_\tau,b_\tau} \leq \frac12 + \lambda \Vert\nabla\phi(\eta(\tau,z))\Vert_{a_\tau,b_\tau}.
\]
Now by \eqref{eq:1.8fabio}, \eqref {eq:1.9fabio}, \eqref{eq:5.10bisnew} and \eqref{eq:No} (recalling that we are assuming that
$\eta(\tau,z) \in \mathfrak M$)
\[
\Vert\nabla\phi(\eta(\tau,z)\Vert_{a_\tau,b_\tau} < \sqrt{\frac{K_0^2}{2\ell_0} + \frac{M_0N_0^2}{\ell_0}},
\]
so
\[
\Vert W^{\lambda}(z,\eta(\tau,z)) \Vert_{a_\tau,b_\tau} \leq \frac12 + \lambda\sqrt{\frac{K_0^2}{2\ell_0} + \frac{M_0N_0^2}{\ell_0}}.
\]
\end{rem}

\medskip

Then if we set
\begin{equation}\label{eq:lambda1}
\lambda_1 = \min(\frac{\sqrt{\ell_0}}{2K_0},\frac{\sqrt{\ell_0}}{\sqrt{2K_0^2+4M_0N_0^2}}),
\end{equation}
we have (recall \eqref{eq:7.11ter}),
\begin{equation}\label{eq:normaW}
\Vert W^{\lambda}(z,\eta(\tau,z)) \Vert_{L^{\infty}([0,1],\mathbb{R}^{N})} \leq 1, \Vert W^{\lambda}(z,\eta(\tau,z))\Vert_{a_\tau,b_\tau} \leq 1
\end{equation}
for any $\lambda \in ]0,\lambda_1]$, for any $\tau$ such that $\eta(\tau,z) \in \mathfrak M$.

{\bf Step 5.} Fix $\lambda \in ]0,\lambda_1]$ and $z \in \Ucal_{x_i}$ for some $i=1,\ldots,k$. Suppose that
$\tau_* \in ]0,\frac{\rho(r)}2]$ is such that:
\begin{itemize}
\item $\Vert \eta(\tau_*,z) - x_i) \Vert_{a_{\tau_*},b_{\tau_*}} \geq \rho(r)$,
\item $\Vert \eta(\tau,z) - x_i) \Vert_{a_{\tau},b_{\tau}} < \rho(r)$, for any $\tau \in [0,\tau_*[$,
\item $\eta(\tau_*,z) \in \mathfrak M$, for any $\tau \in [0,\tau_*[$,
\end{itemize}
where $[a_\tau,b_\tau] \in \Ical_{\eta(\tau,z)}$ for any $\tau \in [0,\tau_*]$. Then $\tau_* \geq \frac12 \rho(r)$.
\begin{proof}
Since $\rho(x_i) \leq \frac12 \rho(r)$ and (by Step 4) $[a_{\tau_*},b_{\tau_*}] \subset [a_0,b_0] = [a_z,b_z]$, it is
\begin{multline*}
\frac12 \rho(r) \leq \Vert \eta(\tau_*,z) - x_i) \Vert_{a_{\tau_*},b_{\tau_*}} - \Vert \eta(0,z) - x_i) \Vert_{a_{0},b_{0}} \leq\\
\Vert \eta(\tau_*,z) - x_i) \Vert_{a_{\tau_*},b_{\tau_*}} - \Vert \eta(0,z) - x_i) \Vert_{a_{\tau_*},b_{\tau_*}}.
\end{multline*}
Then
\[
\frac12 \rho(r) \leq \Vert \eta(\tau_*,z) - \eta(0,z)) \Vert_{a_{\tau_*},b_{\tau_*}} \leq
\int_{0}^{\tau_*} \Vert \frac{\mathrm d}{\mathrm d\sigma}\eta(\sigma,z)\Vert_{a_{\tau_*},b_{\tau_*}}ds.
\]
Now, by Step 4, $[a_{\tau_*},b_{\tau_*}] \subset [a_{\tau},b_{\tau}]$ for any $\tau \in [0,\tau_*]$. Then by
\eqref{eq:normaW} and \eqref{eq:lambda1} it is $\Vert \frac{\mathrm d}{\mathrm d\sigma}\eta(\sigma,z)\Vert_{a_{\tau_*},b_{\tau_*}} \leq 1$
for any $\sigma \in [0,\tau_*]$, from which we deduce the thesis.
\end{proof}

Define
\begin{equation}\label{eq:mhd}
M_{h,\Dcal} = \sup\{\frac12\int_{a}^{b}g(\dot z, \dot z)ds: [a,b] \in \mathcal I_{z}, z \in h(1,\Dcal) \}
\end{equation}
(and note that $M_{h,\Dcal} < M_0$ since $h(1,\Dcal) \subset \mathfrak M$ and it is compact).

{\bf Step 6.} Fix
\begin{equation}\label{eq:tau1}
\tau_1(r) = \min\left(\frac{\rho(r)}2, \frac{E(r)}{8L_1\sqrt{2M_0L_0}}, \frac{M_0-\frac{c_1}2}{L_1\sqrt{2M_0L_0}},1\right).
\end{equation}
Then for any
$\lambda \in ]0,\min(\lambda_1,\frac{M_0-M_{h,\Dcal}}{2N_0M_0})]$ and for any $\tau \in [0,\tau_1(r)]$ we have $\eta(\tau,z) \in \mathfrak M$.
\begin{proof}
Since $W^{\lambda}(z,\eta(\tau,z))(s)=0$ whenever $\phi(\eta(\tau,z))\geq \frac{3\delta_0}4$ (cf. \eqref{eq:vectorfield} and
\eqref{eq:chi}),
by step 4 $\eta(\tau,z) \in \mathfrak M_0$
for any $\tau \in [0,\frac{\rho(r)}2]$. We have therefore to prove that, for any
$\tau \in [0,\tau_1(r)]$:
\begin{equation}\label{eq:in_mathfrakM}
\int_{a_\tau}^{b_\tau}\frac12g(\dot y_\tau,\dot y_\tau)ds < M_0, \, y_\tau=\eta(\tau,z), \;
\forall [a_\tau,b_\tau] \in \Ical_{\eta(\tau,z)}.
\end{equation}

First of all
consider a NIIC $\mathcal C_{0,\tau_1(r)} = \{[a_\tau',b_\tau'] : \tau \in [0,\tau_1(r)] \}$ such that
\[
[a_{\tau}',b_{\tau}'] \subset [a_\tau,b_\tau] \in \Ical_{\eta(\tau,z)}
\]
and there exists $\tau_0 \in [0,\tau_1(r)[$ such that  $[a_{\tau_0}',b_{\tau_0}'] \in \mathcal C_{0,\tau_1(r)}$ satisfies
\begin{multline}\label{eq:intervallo-iniziale}
\exists i \in \{1,\ldots,k\}: [a_{\tau_0}',b_{\tau_0}']\subset [a_z,b_z]\in\Ical_{z}, z \in \Ucal_{x_i}, \\
\text{ and }[a_{\tau_0}',b_{\tau_0}'] \subset [\alpha_{x_i}',\beta_{x_i}'] \setminus ]a_{x_i},b_{x_i}[, \text{ for some }[a_{x_i},b_{x_i}]
\in {\mathcal I}_{x_i}.
\end{multline}
Note that, by the choice of $\rho(x_i)$ it is $\int_{a_{\tau_0}'}^{b_{\tau_0}'}\frac12g(\dot z, \dot z)ds < \frac{c_1}2$,
therefore since $\tau_1(r)$ satisfies
\[
\frac{c_1}2 + L_1\tau_1(r)\sqrt{2M_0L_0} \leq M_0,
\]
by Lemma \ref{thm:controllocrescita} we have:
\begin{multline}\label{eq:stima-intervalli-piccoli}
\text{if } [a_{\tau_0}',b_{\tau_0}'] \in \mathcal C_{0,\tau_1(r)} \text{ satisfies \eqref{eq:intervallo-iniziale} for some } i \in \{1,\ldots,k\} \text{ then }\\
\int_{a_\tau'}^{b_\tau'}\frac12g(\dot y_\tau,\dot y_\tau)ds < M_0, \, y_\tau=\eta(\tau,z), \;
\forall [a_\tau',b_\tau'] \in \mathcal C_{0,\tau_1(r)}.
\end{multline}

Now choose a NIIC $\widehat {\mathcal C}_{0,\tau_1(r)}$ consisting of intervals $[a_\tau.b_\tau]$ such that
\[
[a_\tau.b_\tau] \in \Ical_{\eta(\tau,z)} \text{ for any }\tau \in [0,\tau_1(r)],
\]
and consider the following partition of the set of indexes $\{1,\ldots,k\}$ (recalling that, by \eqref{eq:propertyintersezione}, any $[a_\tau,b_\tau]$ intersects
at most one interval in $\Ical_{x_i,r,\varepsilon}$):
\[
I_1^{\tau} = \{i=1,\ldots,k: [a_\tau,b_\tau]\cap[a_{x_i},b_{x_i}] \neq \emptyset \text{ for some } [a_{x_i},b_{x_i}] \in \mathcal I_{x_i,r,\varepsilon}\},
\]
\begin{multline*}
I_2^{\tau} = \{j=1,\ldots,k:  [a_\tau,b_\tau] \subset [\alpha_{x_j}',\beta_{x_j}'] \\
\text{ and }
[a_\tau,b_\tau]\cap[a_{x_j},b_{x_j}] = \emptyset \text{ for some } [a_{x_j},b_{x_j}] \in \mathcal I_{x_j,r,\varepsilon}\},
\end{multline*}
\begin{multline*}
I_3^{\tau} =\{m=1,\ldots,k:  [a_\tau,b_\tau] \cap [\alpha_{x_m},\beta_{x_m}]   = \emptyset \\
\text{ for any }[\alpha_{x_m},\beta_{x_m}] \text { having the corresponding }[a_{x_m},b_{x_m}] \in \mathcal I_{x_m,r,\varepsilon}\},
\end{multline*}
\[
J_1^{\tau} = \{i=1,\ldots,k: [a_\tau,b_\tau]\cap[a_{\Rcal x_i},b_{\Rcal x_i}] \neq \emptyset \text{ for some }
[a_{\Rcal x_i},b_{\Rcal x_i}] \in \mathcal I_{\Rcal x_i,r,\varepsilon}\},
\]
\begin{multline*}
J_2^{\tau} = \{j=1,\ldots,k:  [a_\tau,b_\tau] \subset [\alpha_{\Rcal x_j}',\beta_{\Rcal x_j}'] \\
\text{ and }
[a_\tau,b_\tau]\cap[a_{\Rcal x_j},b_{\Rcal x_j}] = \emptyset \text{ for some }
[a_{\Rcal x_j},b_{\Rcal x_j}] \in \mathcal I_{\Rcal x_j,r,\varepsilon}\},
\end{multline*}
\begin{multline*}
J_3^{\tau} =\{m=1,\ldots,k:  [a_\tau,b_\tau] \cap [\alpha_{\Rcal x_m},\beta_{\Rcal x_m}]   = \emptyset \\
\text{ for any }[\alpha_{\Rcal x_m},\beta_{\Rcal x_m}] \text { having the corresponding }[a_{\Rcal x_m},b_{\Rcal x_m}] \in \mathcal I_{\Rcal x_m,r,\varepsilon}\}.
\end{multline*}

\medskip
Note that, by \eqref{eq:stima-intervalli-piccoli}, we can just consider the case $I_2^{\tau} \cup J_2^{\tau} = \emptyset$ for any $\tau \in [0,\tau_1(r)]$,
studying first the case:
\begin{multline}\label{eq:I}
\text{ there exist }\tau_* \in [0,\tau_1(r)] \text{ and }i \in I_1^{\tau_*} \cup J_1^{\tau_*} \text{ such that:}\\
[a_{\tau_*},b_{\tau_*}] \cap [a_{x_i},b_{x_i}] \neq \emptyset \text{ for some }[a_{x_i},b_{x_i}] \in \Ical_{x_i,r,\varepsilon} \\
\text { and } D(x_i,a_{\tau_*},b_{\tau_*},a_{x_i},b_{x_i}) > E(r) \text{ or } \\
[a_{\tau_*},b_{\tau_*}] \cap [a_{\Rcal x_i},b_{\Rcal x_i}] \neq \emptyset \text{ for some }[a_{\Rcal x_i},b_{\Rcal x_i}] \in \Ical_{\Rcal x_i,r,\varepsilon} \\
\text{ and }
D(\Rcal x_i,a_{\tau_*},b_{\tau_*},a_{\Rcal x_i},b_{\Rcal x_i}) > E(r).
\end{multline}

In this case by Lemma \ref{thm:differenze} applied with $K_* = E(r)$ we deduce that
\[
\int_{a_{\tau_*}}^{b_{\tau_*}}\frac12g(\dot y_{\tau_*},\dot y_{\tau_*})ds < M_0 - \frac{E(r)}8, \, y_\tau=\eta({\tau_*},z).
\]
Then by Lemma \ref{thm:controllocrescita}, for any $\tau \geq \tau_*$ we have
\[
\int_{a_{\tau}}^{b_{\tau}}\frac12g(\dot y_{\tau},\dot y_{\tau})ds < M_0 - \frac{E(r)}8 + L_1(\tau-\tau_*)\sqrt{2M_0L_0},
\]
provided that $M_0 - \frac{E(r)}8 + L_1(\tau-\tau_*)\sqrt{2M_0L_0} \leq M_0$.
Then, since  $\tau_1(r) > 0$ satisfies (cf. \eqref{eq:tau1})
\begin{equation}\label{eq:scelta1pertau1}
L_1\tau_1(r)\sqrt{2M_0L_0} - \frac{E(r)}8 \leq 0,
\end{equation}
we have the following property:
\begin{multline*}
\text{if there exist }\tau_* \in [0,\tau_1(r)]  \text{ and }i \in I_1^{\tau_*} \cup J_1^{\tau_*} \text{ such that \eqref{eq:I} is satisfied, then}\\
\int_{a_{\tau}}^{b_{\tau}}\frac12g(\dot y_{\tau},\dot y_{\tau})ds < M_0 \text{ for any }\tau \in [\tau_*,\tau_1(r)],
\text{ for any }[a_\tau,b_\tau]\in \widehat {\mathcal C}_{\tau_*,\tau_1(r)}.
\end{multline*}

Therefore, recalling that we have just to consider the case
$I_2^{\tau} \cup J_2^{\tau} = \emptyset$ for any $\tau \in [0,\tau_1(r)]$,
to conclude the proof of step 6 it will be sufficient to prove that
\[
\int_{a_{\tau}}^{b_{\tau}}\frac12g(\dot y_{\tau},\dot y_{\tau})ds < M_0 \text{ for any }\tau \in [0,\tau_1(r)],
\text{ for any }[a_\tau,b_\tau]\in \widehat {\mathcal C}_{0,\tau_1(r)},
\]
under the assumptions
\begin{multline}\label{eq:ipotesi-finaleper-step6}
\forall \hat \tau \in ]0,\tau_1(r)], \forall \tau \in [0,\hat \tau],\forall i \in I_1^{\tau} \cup J_1^{\tau} \text{ it is }\\
[a_{\tau_*},b_{\tau_*}] \cap [a_{x_i},b_{x_i}] \neq \emptyset \text{ with }[a_{x_i},b_{x_i}] \in \Ical_{x_i,r,\varepsilon} \\
\text { and } D(x_i,a_{\tau},b_{\tau},a_{x_i},b_{x_i}) \leq E(r) \text{ and } \\
[a_{\tau_*},b_{\tau_*}] \cap [a_{\Rcal x_i},b_{\Rcal x_i}] \neq \emptyset \text{ for some }[a_{\Rcal x_i},b_{\Rcal x_i}] \in \Ical_{\Rcal x_i,r,\varepsilon}\\
\text{ and }D(\Rcal x_i,a_{\tau},b_{\tau},a_{\Rcal x_i},b_{\Rcal x_i}) \leq E(r),
\end{multline}
and
\begin{equation}\label{eq:caso2}
I_2^{\tau} \cup J_2^{\tau} = \emptyset \text{ for any }\tau \in [0,\tau_1(r)].
\end{equation}

Towards this goal, for all $z \in h(1,\Dcal)$, set:
\[
\Phi_z(\tau)= F_{h(1,z)}(\tau,\widehat{\mathcal C}_{0,\tau_1(r)}).
\]

Since $z \in \mathfrak M$, $\Phi_z(0) < M_0$. Then by
Lemma \ref{thm:palle}, there exists $\tilde \tau > 0$ such that
$\Phi_z(\tau) < M_0$ for any $\tau \in [0,\tilde\tau]$. Working on $\tau \in [0,\tilde\tau]$, we have,
for any $l$ sufficiently small,

\begin{multline*}
\Phi_{z}(\tau+l)-\Phi_{z}(\tau)=\\
 f_{a_{\tau+l},b_{\tau+l}}(\eta(\tau+l,z))-
f_{a_{\tau},b_{\tau}}(\eta(\tau+l,z))+f_{a_{\tau},b_{\tau}}(\eta(\tau+l,z))
-f_{a_{\tau},b_{\tau}}(\eta(\tau,z))\\
\qquad\stackrel{\text{by Step 4}}{\le}
f_{a_{\tau},b_{\tau}}(\eta(\tau+l,z))
-f_{a_{\tau},b_{\tau}}(\eta(\tau,z)) \\
\qquad=l\,\mathrm df_{a_\tau,b_\tau}(\eta(\tau,z))
\Big[\frac{\mathrm d}{\mathrm d\tau}\eta(\tau,z)\Big]+o(l)\quad\text{as $l\to0$.}
\end{multline*}

Then setting $y_\tau = \eta(\tau,z)$
\begin{multline}\label{eq:limsupPhi}
\limsup_{l\to0^+} \frac{\Phi_z(\tau+l)-\Phi_z(\tau)}{l} \leq
\int_{a_\tau}^{b_\tau}g(\dot y_{\tau},\Dds W^{\lambda}(z,y_\tau))ds = \\
\sum_{i=1}^{k}\left(\beta_{x_i}(z)\int_{a_\tau}^{b_\tau}g(\dot y_{\tau},\Dds \widehat V_{x_i})ds+
\beta_{\Rcal x_i}(z)\int_{a_\tau}^{b_\tau}g(\dot y_{\tau},\Dds \widehat V_{\Rcal x_i})ds\right)+\\
\lambda\int_{a_\tau}^{b_\tau}H^{\phi}(y_\tau)[\dot y_\tau,\dot y_\tau]ds
\end{multline}

Now if $i \in I_3^{\tau}$, $\widehat V_{x_i}(s)=0$ for any
$s \in [a_\tau,b_\tau]$ and if $j \in J_3^{\tau}$, $\widehat V_{\Rcal x_j}(s)=0$ for any
$s \in [a_\tau,b_\tau]$. Therefore, by \eqref{eq:ipotesi-finaleper-step6}, \eqref{eq:caso2}, \eqref{itm:PS5f}
of Proposition \ref{thm:PS} and Step 5 (recalling that $\tau_1(r) \leq \frac{\rho(r)}2$), we have
\[
\limsup_{l\to0^+} \frac{\Phi_z(\tau+l)-\Phi_z(\tau)}{l} \leq \lambda\int_{a_\tau}^{b_\tau}H^{\phi}(y_\tau)[\dot y_\tau,\dot y_\tau]ds
< 2\lambda N_0M_0
\]
for any $\tau \in [0,\tau_1(r)]$ such that $\Phi_z(\tau) < M_0$. Then
\[
\Phi_z(\tau) < M_0 \Longrightarrow \Phi_z(\tau) < \Phi_z(0) + 2N_0M_0\lambda\tau \leq M_{h,\Dcal} + 2N_0M_0\lambda\tau,
\]
where $M_{h,\Dcal}$ is defined by \eqref{eq:mhd}.
Therefore, since $\tau_1(r) \leq 1$ and $\lambda$ satisfies
\[
M_{h,\Dcal} + 2N_0M_0\lambda \leq M_0,
\]
the proof of Step 6 is concluded.
\end{proof}

{\bf Step 7.} Fix $z \in h(1,\mathcal D)$ and a NIIC ${\mathcal C}_{0,\tau_1(r)}$ consisting of intervals $[a_\tau,b_\tau] \in \Ical_{\eta(\tau,z)}$ such that
\begin{multline}\label{eq:step7}
[a_0,b_0] \in \mathcal I_z \text{ and }z \in \mathcal U_{x_i} \Longrightarrow [a_0,b_0] \cap [a_{x_i},b_{x_i}] \neq \emptyset,\\
\text {for some }[a_{x_i},b_{x_i}] \in \mathcal I_{x_i,r,\varepsilon}.
\end{multline}
Set
\[
\Psi_z(\tau)= G_z(\tau,\widehat{\mathcal C}_{0,\tau_1(r)}),
\]
cf. \eqref{eq:Gflusso}. Then there exists $\tau_2(r) \in ]0,\tau_1(r)]$ and $\lambda_2 \in ]0,\min(\lambda_1,\frac{M_0-M_{h,\Dcal}}{2N_0M_0})]$ such that
\begin{equation}\label{eq:limsupPsi}
\limsup_{l\to0^+} \frac{\Psi_z(\tau+l)-\Psi_z(\tau)}{l} \leq -\frac{\mu(r)c_1}{16M_0} \text{ for any } \tau\in [0,\tau_2(r)]
\end{equation}
or
\begin{multline}\label{eq:alternativa-limsupPsi}
\exists \tau_* \in [0,\tau_2(r)] \text{ such that \eqref{eq:limsupPsi} holds for any }\tau \in [0,\tau_*[ \text{ and }\\
\Psi_z(\tau) \leq \Psi_z(0) - \min \big(\frac{c_1E(r)}{64M_0},\frac{c_1}4 \big) \text{ for any }\tau \in [\tau_*,\tau_2(r)].
\end{multline}

\begin{proof} Fix $l > 0$. By Step 4
\[
\Psi_z(\tau + l) - \Psi_z(\tau) \leq (b_\tau - a_\tau)(\Phi_z(\tau + l) - \Phi_z(\tau)),
\]
where $[a_\tau,b_\tau] \in \widehat{\mathcal C}_{0,\tau_1(r)}$. Then setting $y_\tau = \eta(\tau,z)$ by \eqref{eq:limsupPhi} we have
\begin{multline}\label{eq:limsupPsi1}
\limsup_{l\to0^+} \frac{\Psi_z(\tau+l)-\Psi_z(\tau)}{l} \leq
(b_{\tau}-a_{\tau})\int_{a_\tau}^{b_\tau}g(\dot y_{\tau},\Dds W^{\lambda}(z,y_\tau))ds = \\
(b_{\tau}-a_{\tau})\sum_{i=1}^{k}\left(\beta_{x_i}(z)\int_{a_\tau}^{b_\tau}g(\dot y_{\tau},\Dds \widehat V_{x_i})ds+
\beta_{\Rcal x_i}(z)\int_{a_\tau}^{b_\tau}g(\dot y_{\tau},\Dds \widehat V_{\Rcal x_i})ds\right)+\\
\lambda(b_{\tau}-a_{\tau})\int_{a_\tau}^{b_\tau}H^{\phi}(y_\tau)[\dot y_\tau,\dot y_\tau]ds.
\end{multline}

Consider first the case
\[
\exists \tau_* \in [0,\tau_2(r)]: \Psi_z(\tau) < \frac{c_1}2.
\]
By Lemma \ref{thm:controllocrescita} for any $\tau \in [\tau_*,\tau_2(r)]$ it is
\begin{equation}\label{eq:Psi-non-inters}
\Psi_z(\tau) \leq \Psi_z(\tau_*) + L_1\tau_2(r)\sqrt{2M_0L_0} < \frac{c_1}2 + L_1\tau_2(r)\sqrt{2M_0L_0} \leq \frac34 c_1,
\end{equation}
provided that
\[
L_1\tau_2(r)\sqrt{2M_0L_0} \leq \frac14 c_1.
\]
Note that the same estimate made in \eqref{eq:Psi-non-inters}  allows to treat also the case
\[
\exists i \in \{1,\ldots,k\} \text{ and }\tau_* \in [0,\tau_2(r)] \text{ such that }[a_{\tau_*},b_{\tau_*}] \cap [a_{x_i},b_{x_i}] = \emptyset,
\]
where $[a_{x_i},b_{x_i}] \in \Ical_{x_i,r,\varepsilon}$ is the interval satisfying \eqref{eq:step7}, as done in Step 6 to prove
\eqref{eq:stima-intervalli-piccoli}.

Now suppose that
\begin{multline*}
\text{ there exist }\tau_* \in [0,\tau_2(r)] \text{ and }i \in \{1,\ldots,k\} \text{ such that: }
\Psi_z(\tau_*)\geq \frac{c_1}2 \text{ and } \\ [a_{\tau_{*}},b_{\tau_{*}}] \cap [a_{x_i},b_{x_i}] \neq \emptyset \text{ and }
D(x_i,a_{\tau_*},a_{\tau_*},a_{x_i},b_{x_i}) > E(r) \\
\text{ or }
[a_{\tau_{*}},b_{\tau_{*}}] \cap [a_{\Rcal x_i},b_{\Rcal x_i}] \neq \emptyset \text{ and }
D(\Rcal x_i,a_{\tau_*},a_{\tau_*},a_{\Rcal x_i},b_{\Rcal x_i}) > E(r),\\
\text{ where } [a_{x_i},b_{x_i}] \text{ is the interval satisfying \eqref{eq:step7}}.
\end{multline*}
In this case by Lemma \ref{thm:differenze} with $K_{*} = E(r)$ we deduce that
\[
\Psi_z(\tau_*) \leq \Psi_z(0) - \frac{c_1E(r)}{32M_0}.
\]
Then, by Lemma \ref{thm:controllocrescita} for any $\tau \in [\tau_*,\tau_2(r)]$ it is
\begin{multline*}
\Psi_z(\tau) \leq \Psi_z(\tau_*) + L_1\tau_2(r)\sqrt{2M_0L_0} \leq
\Psi_z(0)  - \frac{c_1E(r)}{32M_0} + L_1\tau_2(r)\sqrt{2M_0L_0} \leq \\
\Psi_z(0) - \frac{c_1E(r)}{64M_0},
\end{multline*}
provided that
\[
L_1\tau_2(r)\sqrt{2M_0L_0} \leq \frac{c_1E(r)}{64M_0}.
\]
Then it remains to consider only the case:
\begin{multline*}
\text{for any }\tau \in [0,\tau_2(r)], \text{ for any }i \in \{1,\ldots,k\} \text{ it is:  }\\
\Psi_z(\tau)\geq \frac{c_1}2,
[a_\tau,b_\tau] \cap [a_{x_i},b_{x_i}] \neq \emptyset, \, [a_\tau,b_\tau] \cap [a_{\Rcal x_i},b_{\Rcal x_i}] \neq \emptyset,\\
D(x_i,a_{\tau_*},b_{\tau_*},a_{x_i},b_{x_i}) \leq E(r) \text{ and } D(\Rcal x_i,a_{\tau_*},b_{\tau_*},a_{\Rcal x_i},b_{\Rcal x_i}) \leq E(r),\\
\text{ where }[a_{x_i},b_{x_i}] \in \mathcal I_{x_i,r,\varepsilon} \text{ satisfies \eqref{eq:step7}}.
\end{multline*}

Now recall that $V_{x_i}$ has been chosen so that $\Vert V_{x_i} \Vert_{a_{x_{i}},b_{x_{i}}}
=\frac12$ for any $i$, while
$c_1 \geq \frac12(\frac{3\delta_0}{4K_{0}})^{2}$ so
$\frac{c_1}2 \geq \frac14(\frac{3\delta_0}{4K_{0}})^{2}$. Then
by the choice of $\rho(x_i)$,  \eqref{eq:limsupPsi1}, Steps 5 and 6 and Propositions
\ref{thm:51new} and \ref{thm:PS} we have:

\begin{multline*}
\limsup_{l\to0^+} \frac{\Psi_z(\tau+l)-\Psi_z(\tau)}{l} \leq \\
(b_{\tau}-a_{\tau})\sum_{i=1}^{k}\left(\beta_{x_i}(z)(-\frac{\mu(r)}4) +
\beta_{\Rcal x_i}(z)(-\frac{\mu(r)}4) \right)+\\
\lambda(b_{\tau}-a_{\tau})N_0\int_{a_\tau}^{b_\tau}g(\dot y_\tau,\dot y_\tau)ds \leq
(b_{\tau}-a_{\tau})(-\frac{\mu(r)}4) + 2\lambda N_0M_0.
\end{multline*}
Now $b_{\tau}-a_{\tau} \geq \frac{c_1}{2M_0}$, therefore
\[
\limsup_{l\to0^+} \frac{\Psi_z(\tau+l)-\Psi_z(\tau)}{l} \leq -\frac{c_1}{2M_0}\frac{\mu(r)}4 + 2\lambda N_0M_0 \leq -\frac{\mu(r)c_1}{16M_0},
\]
choosing $\lambda_2$£ such that $2\lambda_2 N_0M_0 \leq \frac{\mu(r)c_1}{16M_0}$.
\end{proof}
We can finally conclude the proof of Proposition \ref{thm:prop4.17f} defining
\[
H_{\varepsilon}(\tau,z) = \eta(\tau T_{\varepsilon},z),
\]
which is of type $A$ in $[0,1]$, $T_{\varepsilon} = 2\frac{\varepsilon 16M_0}{c_1\mu(r)}$, $\varepsilon \in ]0,\varepsilon_1(r)]$,
and $\varepsilon_1(r)$ such that $T_{\varepsilon_{1}(r)} \leq \tau_2(r)$.

By Steps 1,2,4 and 6, $(\mathcal D, H_{\varepsilon} \star h) \in \mathcal H$, so \eqref{itm:1lemdef1} of Proposition \ref{thm:prop4.17f} is satisfied.
By Step 7  it follows  that also \eqref{itm:1lemdef2} of Proposition \ref{thm:prop4.17f} is satisfied.
Finally property \eqref{eq:normaW} gives \eqref{itm:1lemdef3}.
\end{proof}

\section{Deformation results for homotopies of type B}\label{sec:riparametrizzazione}
In this section we study how to proceed when there are curves nearby variationally critical portions
of second type. On these curves we use a suitable reparameterization flow to push down the functional $\mathcal F$. The construction
of the reparameterizations is based on the following simple idea. Fix a curve $x:[a,b] \rightarrow H^{1}([a,b],\mathbb{R}^N)$ and fix $c \in ]a,b[$.
Consider the integral average of the energy in the intervals $[a,c]$ and $[c,b]$. Suppose that they are different and take an affine reparameterization
of $[a,c]$ on $[a,c+\tau]$ and of $[c,b]$ on $[c+\tau,b]$. In such a way we can move $c$ towards $a$ if $\tau < 0$ and $c$ towards $b$
if $\tau > 0$, and consider the corresponding reparameterization $x_\tau$ of the curve $x$. Moving $c$ towards the extreme point of the interval
where the integral average is smaller, the energy functional decreases. This fact is rigorously explained in Remark \ref{rem:contophi}.

The above property is clearly satisfied for any irregular variationally critical portion $x$ of II type if we choose as $c$ an instant nearby an
interval where $x$ is constant. We can therefore to push down $\mathcal F$ in a neighborhood of any irregular variationally critical portion of II type,
obtaining Proposition \ref{thm:prop4.19f}. We shall do this trying to follow as much as possible the frame of the previous section.

Note that using reparameterizations there is the possibility that pieces of curves outside
$\Omega$ move inside $\Omega$. Anyway this is allowed for homotopies of type B
and it is not a problem since we are just dealing with reparametrization of intervals.

\medskip

Let $x\in\mathfrak M$ be fixed; for all $[a,b]\in {\mathcal I}_x$ such that
$x_{\vert[a,b]}$ is an irregular variationally critical portion of
second type, we set:
\begin{eqnarray}\label{eq:4.45f}
\ell_-(x)=\max\big\{s\in\left]0,b-a\right[:x_{\vert[a,a+s]}\ \text{is constant}\big\},
\\ \label{eq:4.46f}
\ell_+(x)=\max\big\{s\in\left]0,b-a\right[:x_{\vert[b-s,b]}\ \text{is constant}\big\}.
\end{eqnarray}
\begin{rem}\label{thm:rem4.18f}
Note that, by Proposition~\ref{thm:regcritpt}
if $x_{\vert[a,b]}$ is of first type, it may be $\ell_-(x)+\ell_+(x)=0$,
while if $x_{\vert[a,b]}$ is of second type we have $\ell_-(x)+\ell(x)>0$, and
such a quantity is, in a sense, a measure of the distance from $x_{\vert[a,b]}$
to the set of OGC's affinely reparameterized in $[a+\ell_-(x),b-\ell_+(x)]$.
\end{rem}

Instead of considering only reparameterization flows we prefer to
give the analogous of Proposition~\ref{thm:prop4.17f}
when we are far only from
regular variationally critical portions and
from irregular variationally critical portions of first type. Nearby irregular variationally critical portion of II type we use reparameteriztions
of the type described in Remark \ref{rem:contophi}, while far from them we use the results of Proposition~\ref{thm:prop4.17f}.

To this aim, we set:
\begin{equation}\label{eq:4.50f}
\begin{split}
\mathcal Z_{a,b}^1=\Big\{y\in
H^1([a,&b],\overline\Omega):\text{either}\ y_{\vert[a,b]}\ \text{is an OGC },\\
&\text{or $y_{\vert[a,b]}$ is an irregular variationally critical portion of first type.}\Big\}
\end{split}
\end{equation}

\begin{rem}\label{rem:contophi}
Assume $[a,b]\subset[0,1]$ and let $c\in\left]0,1\right[$ be fixed.
Let  $\tau_0>0$ be such that
$c+\tau\in\left]a,b\right[$  for all $\tau\in]-\tau_0,\tau_0[$, and for
all $x\in\mathfrak M$ denote by $x_{a,c,\tau}$ the (orientation preserving)
affine reparameterization on  the interval
$[a,c+\tau]$ of $x_{\vert[a,c]}$. Similarly, define $x_{c,\tau,b}$ to be the affine reparameterization
of $x_{\vert[c,b]}$ on $[c+\tau,b]$. Set:
\begin{equation}\label{eq:4.42f}
\varphi(\tau)=\frac12\int_a^{c+\tau}g\big(\dot x_{a,c,\tau},\dot x_{a,c,\tau}\big)\,\mathrm ds+
\int_{c+\tau}^bg\big(\dot x_{c,\tau,b},\dot x_{c,\tau,b}\big)\,\mathrm ds.
\end{equation}
A straightforward calculation gives:
\begin{equation}\label{eq:4.43f}
\varphi(\tau)=\frac{c-a}{2(c-a+\tau)}\int_a^cg(\dot x,\dot x)\,\mathrm ds+\frac{b-c}{2(b-c-\tau)}
\int_c^bg(\dot x,\dot x)\,\mathrm ds.
\end{equation}
Hence,
\[\varphi'(\tau)=-\frac{(c-a)\int_a^cg(\dot x,\dot x)\,\mathrm ds}{2(c-a+\tau)^2}+\frac{%
(b-c)\int_c^bg(\dot x,\dot x)\,\mathrm ds}{2(b-c-\tau)^2},\]
and in particular:
\begin{equation}\label{eq:eqLem4.18}
\varphi'(0)=-\frac{1}{2(c-a)}\int_a^cg(\dot x,\dot x)\,\mathrm ds+\frac1{2(b-c)}\int_c^bg(\dot x,\dot x)\,\mathrm ds.
\end{equation}

\end{rem}

\bigskip

\begin{prop}\label{thm:prop4.19f}
Fix $c_1 \in [\frac12(\frac{3\delta_0}{4K_{0}})^2,M_0[$ and $r > 0$. Then, there exists $\varepsilon_2=\varepsilon_2(r,c_1)\in ]0,\frac{c_1}4[$
such that for any $\varepsilon \in ]0,\varepsilon_2]$, for any $(\mathcal D,h) \in \mathcal H$ satisfying:
\begin{equation}\label{eq:4.28f-a}
\mathcal F(\mathcal D,h) \leq c_1 + \varepsilon
\end{equation}
and
\begin{multline}\label{eq:4.28ff-b}
\big\Vert y-x_{|[a,b]}\big\Vert_{a,b} \geq r \\
\text{ for any }y\in\mathcal Z_{a,b}^1, \text{ for any }
x = h(1,\xi),  \xi \in {\mathcal D}, \text{ such that }\\
\frac{b-a}2\int_{a}^{b}g(\dot x,\dot x)\,\mathrm ds \geq
c_1-\varepsilon, \text{ where }[a,b] \in \mathcal I_x.
\end{multline}
there exists a continuous map
$H_\varepsilon:[0,1]\times h(1,\mathcal D)\to\mathfrak M$, concatenation of a homotopy of type A in $[0,1]$ with another of type B,
with the following properties:
\smallskip

\begin{enumerate}
\item\label{itm:1lemdef1-a} $(\mathcal D,H_\varepsilon \star h)\in\mathcal H$;

\item\label{itm:1lemdef2-b}  $\mathcal F(\mathcal D,H_\varepsilon \star h)\le c_1 - \varepsilon$;

\item\label{itm:1lemdef3-c} there exists $T_\varepsilon>0$, with
$T_\varepsilon\to0$ as $\varepsilon\to0$, such
that
for any $z \in h(1,\mathcal D)$
$\Vert H_\varepsilon(\tau,z)-z\Vert_{a,b}
\le\tau T_\varepsilon$ for all $\tau\in[0,1]$, for all $[a,b] \in {\mathcal I}_z$.
\end{enumerate}
\end{prop}

In order to prove Proposition \ref{thm:prop4.19f} some preliminary results are needed. Thanks to Remark \ref{rem:rem4.12} a simple contradiction argument
gives
\begin{lem}\label{thm:stima-ell}
There exist $\rho_2 > 0$ and $\ell > 0$ such that,
if $x \in \mathfrak M$, $[a,b] \in \Ical_{x}$ and
\[
\Vert x_{|[a,b]} - y\Vert_{a,b} \geq r \text{ for any }y \in Z^1_{a,b},
\]
then
\begin{equation}\label{vicinanza}
\ell_-(w)+\ell_+(w)\ge\ell,
\end{equation}
\begin{equation}\label{eq:energiaIItype}
\frac12\int_{a}^{b}g(\dot w,\dot w)\,\mathrm ds \leq M_0 + 1,
\end{equation}
for any $w_{|[a,b]}$ irregular variational portion of II type satisfying
$\Vert x_{|[a,b]}-w_{|[a,b]} \Vert_{a,b} \leq \rho_2$.
\end{lem}

Now we have the following
\begin{lem}\label{thm:propertiesIItype}
Let $w_{|[a,b]}$ be an irregular critical variational portion of
second type satisfying \eqref{eq:energiaIItype} and
\begin{equation}\label{eq:hp-coda}
\ell_-(w) \geq \frac{\ell}2
\end{equation}
Then there exist $\theta(\ell) > 0$, $\mu(\ell) > 0$ and $\delta(\ell) > 0$ such that, setting
\[
c_{w} = a + \ell_{-}(w) + \theta(\ell) \in ]a,b[
\]
we have
\begin{equation}\label{eq:IItype_propr2}
\tfrac1{2(c_{w}- a)}\int_{a}^{c_{w}}g(\dot w,\dot w)\,\mathrm ds-\tfrac1{2(b-c_{w})}\int_{c_{w}}^{b}g(\dot w,\dot
w)\,\mathrm ds \leq -2\mu(\ell).
\end{equation}
\end{lem}
\noindent (The case $\ell_+(w) \geq \frac{\ell}2$ is completely analogous).

\begin{proof}
Denote by $E_w$
the energy of the OGC  $w$ on the interval $[a + \ell_-(w), b - \ell_+(w)]$,  so that
\[
g(\dot w(s), \dot w(s)) = E_w \text{ for any }s \in [a + \ell_-(w), b - \ell_+(w)].
\]
Now set $c_{w} = a + {\ell}_{-}(w) + \theta$ with $\theta > 0$.
Since $w$ is constant on $[a,b] \setminus [a + \ell_-(w), b - \ell_+(w)]$ it is
\begin{multline}\label{eq:stimeIItype}
\tfrac1{2({c_{w}}- a)}\int_{a}^{c_{w}}g(\dot w,\dot w)\,\mathrm ds-\tfrac1{2(b-c_{w})}\int_{c_{w}}^{b}g(\dot w,\dot
w)\,\mathrm ds= \\
\Big[\frac{\theta}{2(c_{w} - a)}
+ \frac{\theta}{2(b -c_{w})}
- \frac1{2(b - c_{w})}\Big((b-a) - (\ell_{-}(w) + \ell_{+}(w)\Big)\Big]E_w \leq\\
\Big[\frac{\theta}{\ell} +\frac{\theta}{2(b -c_{w})}-\frac12\Big((b-a) - (\ell_{-}(w) +
\ell_{+}(w)\Big)\Big] E_w,
\end{multline}
since $(b-c_{w})\leq 1$ (recall that $[c_{w}, b] \subset [0,1]$) and $c_{w} - a \geq {\ell}_{-}(w)
\geq \frac{\ell}2$.

Now $\Big((b-a) - (\ell_{-}(w) + \ell_{+}(w))\Big)$ is the length of the interval $[a + \ell_-(w), b - \ell_+(w)]$ where $w$ is an OGC of energy $E_w$.
Since $\frac{1}{2}\int_{a}^{b}g(\dot w,\dot w) \leq M_0 +1$ the strong concavity assumption \eqref{eq:1.1bis} and Lemma
\ref{thm:lem2.3}
imply that
\[(b-a) - (\ell_{-}(w) + \ell_{+}(w)) \geq \frac{\delta_0^2}{2K_0^2(M_0+1)}\]
and
\[
E_w \geq \frac{\delta_0^2}{K_0^2},
\]
because $b-a \leq 1$.

Now $b-c_{w} \geq (b-a)-(\ell_{-}(w) + \ell_{+}(w)) - \theta$ and ${c_{w}}- a \geq \frac{\ell}2$, therefore we can choose $\theta = \theta(\ell)$
(independently of $w$) sufficiently small such that
\begin{multline*}
\tfrac1{2({c_{w}}- a)}\int_{a}^{c_{w}}g(\dot w,\dot w)\,\mathrm ds-\tfrac1{2(b-c_{w})}\int_{c_{w}}^{b}g(\dot w,\dot w)\,\mathrm ds\leq \\
\Big(\frac{\theta}{\ell} + \frac12\frac{\theta}{\frac{\delta_0^2}{2K_0^2(M_0+1)}- \theta} -\frac12\frac{\delta_0^2}{2K_0^2(M_0+1)}\Big)
\frac{\delta_0^2}{K_0^2} \leq -2\mu(\ell),
\end{multline*}
where $\mu(\ell) > 0$.
\end{proof}

\bigskip
Now, by Remark  \ref{rem:rem4.12} and Lemma \ref{thm:propertiesIItype} we immediately obtain the following
\begin{lem}\label{thm:nearbyIItype}
There exist $\rho(\ell) \in ]0,\rho_2]$
such that for any $x \in h(1,\Dcal)$ and $[a,b]\in \Ical_{x}$, if there exists $w_{|[a,b]}$ irregular variationally critical
portion of II type such that
\[
\Vert x - w \Vert_{a,b} \leq \rho(\ell), \text{ and }\ell_{-}(w) \geq \frac{\ell}2,
\]
there exists $c_x \in\left]a,b\right[$ such that
\[
\frac1{2(c_{x}- a)}\int_{a}^{c_{x}}g(\dot x,\dot x)\,\mathrm ds-\frac1{2(b-c_{x})}\int_{c_{x}}^{b}g(\dot x,\dot x)\,\mathrm ds\leq -\mu(\ell).
\]
\end{lem}

\bigskip

Let $\rho(\ell)$ be as in Lemma \ref{thm:nearbyIItype}. We can state the analogous of Proposition \ref{thm:51new} as follows. Set
\begin{multline}\label{eq:aggiunta}
\Vcal^{-}_{\ell} = \Big\{x \in h(1,\Dcal): \, \exists [a_x,b_x] \in \Ical_{x},\\
\text{ and }
w_{|[a_x,b_x]} \text{ irregular variationally critical portion of II type satisfying }\\
\ell_{-}(w) \geq \frac{\ell}2 \text{ and }\Vert x - w \Vert_{a_x,b_x} < \rho(\ell)\Big\}.
\end{multline}
Let $x \in \Vcal^{-}_{\ell}$. We shall set
\begin{multline*}
\Ical_{x,\rho(\ell)} = \{[a_x,b_x] \in \Ical_{x}: \exists
w_{|[a_x,b_x]} \text{ irregular variationally critical portion of II type }\\
\text{ satisfying }
\ell_{-}(w) \geq \frac{\ell}2 \text{ and }\Vert x - w \Vert_{a_x,b_x} < \rho(\ell)\Big\}.
\end{multline*}
Let $[a_x,b_x]\in \Ical_{x,\rho(\ell)}$ and $c_x \in ]a_x,b_x[$ given by Lemma \ref{thm:nearbyIItype}.
Define for any $\tau \in ]0,c_x - a_x[$ and $s \in [0,1]$:
\begin{equation}\label{eq:rip-}
\varphi_{x,a_x,c_x,b_x,-}(\tau,s)\equiv
\varphi_{x,-}(\tau,s)=\begin{cases}
s,&\text{if\ } s \in [0,a_x]\\[.3cm]
\frac{c_x-a_x}{c_x-\tau-a_x}(s-a_x)+a_x,&\text{if\ } s\in[a_x,c_x-\tau]\\[.3cm]
\frac{b_x-c_x}{b_x+\tau-c_x}(s-c_x+\tau)+c_x,&\text{if\ } s\in[c_x-\tau,b_x]\\[.3cm]
s,&\text{if\ } s\in[b_x,1].
\end{cases}
\end{equation}

\bigskip

if $x \in {\mathcal V}_{\ell}^{+}$, we choose $c_x$ starting from $b - \ell_{+} -\theta(\ell)$ and define
\[
\varphi_{x,a_x,c_x,b_x,+}(\tau,s)\equiv
\varphi_{x,+}(\tau,s)=\begin{cases}
s,&\text{if\ } s \in [0,a_x]\\[.3cm]
\frac{c_x-a_x}{c_x+\tau-a_x}(s-a_x)+a_x,&\text{if\ } s\in[a_x,c_x+\tau]\\[.3cm]
\frac{b_x-c_x}{b_x-\tau-c_x}(s-c_x-\tau)+c_x,&\text{if\ } s\in[c_x+\tau,b_x]\\[.3cm]
s,&\text{if\ } s\in[b_x,1].
\end{cases}
\]

By Lemma \ref{thm:nearbyIItype} a straightforward computation (made in $\tau = 0$) gives the following result, (analogous of Proposition
\ref{thm:51new}).
\begin{prop}\label{thm:PSrip-}
Let $\mu(\ell)$ and $\theta(\ell)$ be as in Lemma \ref{thm:propertiesIItype}. There exists $\tau(\ell) \in]0,\theta(\ell)]$
such that, for any $\tau \in [0,\tau(\ell)]$, for any $x \in \Vcal^{-}_{\ell}$ and for any $[a_x,b_x]$ as in \eqref{eq:aggiunta} we have:
\[
\int_{a_x}^{b_x}g(\dot x(\sigma),\dot x(\sigma))\frac{\partial^2\varphi_{x,-}}{\partial\tau\partial s}(\tau,\varphi_{x,-}^{-1}(\tau,\sigma))d\sigma
\leq -\frac{3\mu(\ell)}4.
\]
\end{prop}

\begin{rem}\label{rem:precisazione-l+} An analogous results clearly holds for $x \in {\mathcal V}_{\ell}^{+}$ using the parameterization $\varphi_{x,+}$.

Observe also that
\[
x \in \Vcal^{-}_{\ell} \Longleftrightarrow \Rcal x \in \Vcal^{+}_{\ell}
\]
and we choose
\begin{equation}\label{eq:phiRx}
\varphi_{\Rcal x,1-b_x,1-c_x,1-a_x,+}(\tau,s) = \varphi_{x,a_x,c_x,b_x,-}(\tau,s).
\end{equation}
By $\Phi_x$ we shall denote the map $\varphi_{x,-}$ if $x \in \Vcal^{-}_{\ell} \setminus \Vcal^{+}_{\ell}$
and the map $\varphi_{x,+}$ if $x \in \Vcal^{+}_{\ell} \setminus \Vcal^{-}_{\ell}$, relatively to $[a_x,b_x]$.
Whenever $x \in \Vcal^{-}_{\ell} \cap \Vcal^{+}_{\ell}$ we shall denote by $\Phi_x$ just one of the two maps above. It will be only important
to obey to property \eqref{eq:phiRx} to have $\Rcal$--invariance.
\end{rem}

Now we state and proof the analogous of Proposition \ref{thm:PS}, whenever $x \in \Vcal^{-}_{\ell}$ (the case $x \in \Vcal^{+}_{\ell}$ being
completely analogous). Let $\mu(\ell)$ given by Lemma \ref{thm:propertiesIItype}.
\begin{prop}\label{thm:PS-rip1}
There exists $\tau_1(\ell) \in ]0,\tau(\ell)]$, $\rho_1(\ell) \in ]0,\min(1,\rho(\ell))]$ and $E_{1}(\ell) > 0$ such that
for any $x \in \Vcal^{-}_{\ell}$ and $[a_x,b_x]$ a in \eqref{eq:aggiunta}, if $\varphi_{x,-}$ is the parameterization \eqref{eq:rip-},
the following property holds:
\begin{multline*}\label{eq:z-intorno}
\text{for any }z \in h(1,\Dcal), \text{ and for all }[a_z,b_z] \in \Ical_z \text{ such that }\\
\Vert x - z \Vert_{a_z,b_z} < \rho_1(\ell), [a_z,b_z] \cap [a_x,b_x] \neq \emptyset, \text{ and }D(x,a_z,b_z,a_x,b_x)\leq E_{1}(\ell) \text{ it is }\\
\int_{a_z}^{b_z}g(\dot z(\sigma),\dot z(\sigma))\frac{\partial^2\varphi_{x,-}}{\partial\tau\partial s}(\tau,\varphi_{x,-}^{-1}(\tau,\sigma))d\sigma
\leq -\frac{\mu(\ell)}2 \text{ for any }\tau\in[0,\tau_1(\ell)].
\end{multline*}
\end{prop}
\begin{proof}
First note that there exists $M(\ell) > 0$ such that
\[
\Big|\frac{\partial^2\varphi_{x,-}}{\partial\tau\partial s}(\tau,\varphi_{x,-}^{-1}(\tau,s))\Big| \leq M(\ell) \text{ for any }\tau\in[0,\tau(\ell)],
\text{ for any }s \in [0,1].
\]
Then, the same estimate used to prove \eqref{eq:stimax-z} allows to obtain
\begin{multline*}
\Big|
\int_{a_z}^{b_z}g(\dot z(\sigma),\dot z(\sigma))\frac{\partial^2\varphi_{x,-}}{\partial\tau\partial s}(\tau,\varphi_{x,-}^{-1}(\tau,\sigma))d\sigma -\\
\int_{a_z}^{b_z}g(\dot x(\sigma),\dot x(\sigma))\frac{\partial^2\varphi_{x,-}}{\partial\tau\partial s}(\tau,\varphi_{x,-}^{-1}(\tau,\sigma))d\sigma\Big|
\leq \widehat K M(\ell)\rho_1(\ell).
\end{multline*}
where $\widehat K$ is defined by \eqref{eq:kappa0}. Finally
\begin{multline*}
\Big|
\int_{a_z}^{b_z}g(\dot x(\sigma),\dot x(\sigma))\frac{\partial^2\varphi_{x,-}}{\partial\tau\partial s}(\tau,\varphi_{x,-}^{-1}(\tau,\sigma))d\sigma -\\
\int_{a_x}^{b_x}g(\dot x(\sigma),\dot x(\sigma))\frac{\partial^2\varphi_{x,-}}{\partial\tau\partial s}(\tau,\varphi_{x,-}^{-1}(\tau,\sigma))d\sigma
\Big|\leq \\
M(\ell)\int_{I_{a_z,a_x}}g(\dot x,\dot x)\,\mathrm ds + M(\ell)\int_{I_{b_z,b_x}}g(\dot x,\dot x)\,\mathrm ds \leq \\
2M(\ell)D(x,a_z,b_z,a_x,b_x)\leq 2M(\ell)E_1(\ell).
\end{multline*}
Then, if $E_1(\ell)$ is sufficiently small (depending only by $\ell$) by Proposition \ref{thm:PSrip-} we have the thesis.
\end{proof}
\bigskip
\begin{proof}[Proof of Proposition \ref{thm:prop4.19f}.]
For sake of simplicity we shall assume that  any $x \in \Vcal^{-}_{\ell} \cup \Vcal^{+}_{\ell}$ has only one interval $[a_x,b_x] \in \Ical_{x,\rho(\ell)}$.
Indeed the general case can be carried on by small changes.

Starting from any $z \in \overline{\Vcal^{-}_{\ell}} \cup \overline{\Vcal^{+}_{\ell}}$, we apply the reparameterization flow
\begin{equation}\label{eq:rip-flow}
z_{\tau}(s) = z(\Phi_{z}(\tau,s)),
\end{equation}
where
\[
\Phi_{z}(\tau,s) = \sum_{j=1,\ldots,k}(\beta_{x_j}(z)\Phi_{x_j}(\tau,s) + \beta_{\Rcal x_j}(z)\Phi_{\Rcal x_j}(\tau,s)),
\]
$\{\beta_{x_j},\beta_{\Rcal x_j}\}$ is a partition of the unity (defined on
$\overline{\Vcal^{-}_{\ell}} \cup \overline{\Vcal^{+}_{\ell}}$),
as that one used in the proof of Proposition \ref{thm:prop4.17f}, and
$\Phi_{x_j}$ is described in Remark \ref{rem:precisazione-l+}. The partition of the unity is subordinated to the open finite covering
$\{\mathcal U_{x_i},\mathcal U_{\Rcal x_i}\}$ where
$\mathcal U_{x_i} = \{z \in \overline{\Vcal^{-}_{\ell}} \cup \overline{\Vcal^{+}_{\ell}}: \Vert z - x_i \Vert_{0,1} < \rho_1(x_i)\},$
with $\rho_1(x) < \rho_1(\ell)$ (where $x \in \overline{\Vcal^{-}_{\ell}} \cup \overline{\Vcal^{+}_{\ell}}$) that will be chosen later.

\medskip
Fix $[a_z,b_z] \in \Ical_{z}$, $z \in \overline{\Vcal^{-}_{\ell}} \cup \overline{\Vcal^{+}_{\ell}}$ and set
\[
\Gamma_{z}(\tau) = \int_{a_\tau}^{b_\tau}g(\dot z_{\tau},\dot z_{\tau})\,\mathrm ds
\]
where $z_\tau$ is defined in \eqref{eq:rip-flow},
$a_{\tau} = \Phi_{z}(\tau,a_z)$ and $b_{\tau} = \Phi_{z}(\tau,b_z)$.
Since $\Phi_{z}(0,s) = s$ for any $s \in [0,1]$, we have $z_0 = z$, while
\[
\dot z_{\tau}(s) = z(\Phi_{z}(\tau,s))\frac{\partial\Phi_{z}}{\partial s}(\tau,s)
\]
and
\[
\Gamma_{z}(\tau) = \int_{\Phi_{z}(\tau,a_z)}^{\Phi_{z}(\tau,b_z)}g(\dot z(\Phi_{z}(\tau,s)),\dot z(\Phi_{z}(\tau,s))
(\frac{\partial\Phi_{z}}{\partial s}(\tau,s))^2 \,\mathrm ds.
\]
Note that, since  $s \mapsto \Phi_{z}(\tau,s)$ is strictly increasing, the interval
$[\Phi_{z}(\tau,a_z),\Phi_{z}(\tau,b_z)] \in \Ical_{z_{\tau}}$, while setting $\Phi_{z}(\tau,s) = \sigma$ we have
\begin{multline*}
\Gamma_{z}(\tau) =  \int_{a_z}^{b_z}g(\dot z(\sigma),\dot z(\sigma))\frac{\partial\Phi_{z}}{\partial s}(\tau,s)d\sigma =\\
\sum_{j=1,\ldots,k}\beta_{x_j}(z)\int_{a_z}^{b_z}g(\dot z(\sigma),\dot z(\sigma))\frac{\partial\Phi_{x_j}}
{\partial s}(\tau,\Phi_{x_j}^{-1}(\tau,\sigma))d\sigma\\
+ \sum_{j=1,\ldots,k}\beta_{\Rcal x_j}(z)\int_{a_z}^{b_z}g(\dot z(\sigma),\dot z(\sigma))\frac{\partial\Phi_{\Rcal x_j}}
{\partial s}(\tau,\Phi_{\Rcal x_j}^{-1}(\tau,\sigma))d\sigma.
\end{multline*}
Then by Proposition \ref{thm:PS-rip1}, since $\sum_{j=1,\ldots,k}(\beta_{x_j}(z) + \beta_{\Rcal x_j}(z)) = 1$, if
\begin{multline}\label{eq:rip-D}
[a_z,b_z] \cap [a_{x_i},b_{x_i}] \neq \emptyset, \\
D(x_i,a_z,b_z,a_{x_i},b_{x_i})\leq E_{1}(\ell) \text{ and }D(\Rcal x_i,a_z,b_z,a_{\Rcal x_i},b_{\Rcal x_i})\leq E_{1}(\ell) \text{ for any }i
\end{multline}
we have
\begin{equation}\label{eq:rip-Gamma}
\frac{\partial \Gamma_{z}}{\partial \tau}(\tau) \leq -\frac{\mu(\ell)}2 \text{ for any }\tau\in[0,\tau_1(\ell)].
\end{equation}

Note that, (analogously to the proof of Proposition \ref{thm:prop4.17f})  we can choose $\rho_1(x)$ sufficiently small so that if, for some $i$,
$[a_z,b_z] \cap [a_{x_i},b_{x_i}] = \emptyset$ and $a_z$ is close to $a_{x_i}$ or $b_z$ is close to $b_{x_i}$, then the energy
$\frac12\int_{a_z}^{b_z}g(\dot z, \dot z)\,\mathrm ds$ is small. Then using the analogous of Lemmas \ref{thm:differenze} and \ref{thm:controllocrescita}, from
\eqref{eq:rip-Gamma} we can deduce that
\[
z_\tau \in \mathfrak M \, \forall \tau \in [0,\tau_2(\ell)], \text { for some } \tau_2(\ell) \in ]0, \tau_1(\ell)].
\]

Then, to conclude the proof it remains to study the
\[
\Delta_{z}(\tau) = (b_\tau - a_\tau)\Gamma_{z}(\tau).
\]
Since $b_\tau - a_\tau$ is not necessarily decreasing (as we deduce calculating the derivative with respect to $\tau$ of $\varphi_{x_i,-}$
and $\varphi_{x_i,+}$)
it is convenient to recall that
\[
\Gamma_{z}(\tau) < M_0 \text{ for any }\tau \in [0,\tau_1(\ell)]
\]
and observe that
\begin{multline*}
b_{\tau} - a_{\tau} = \sum_{j=1,\ldots,k}\beta_{x_j}(z)\Big[\Phi_{x_j}(\tau,b_z)-\Phi_{x_j}(\tau,a_z)\Big] +\\
\sum_{j=1,\ldots,k}\beta_{\Rcal x_j}(z)\Big[\Phi_{\Rcal x_j}(\tau,b_z)-\Phi_{\Rcal x_j}(\tau,a_z)\Big].
\end{multline*}

Now for any $x$ it is
\[
\frac{\partial \varphi_{x,-}}{\partial \tau}(\tau,s)=\begin{cases}
0,&\text{if\ } s \in [0,a_x]\\[.3cm]
\frac{c_x-a_x}{(c_x-\tau-a_x)^2}(s-a_x),&\text{if\ } s\in[a_x,c_x-\tau]\\[.3cm]
\frac{b_x-c_x}{(b_x+\tau-c_x)^2}(b_x-s),&\text{if\ } s\in[c_x-\tau,b_x]\\[.3cm]
0,&\text{if\ } s\in[b_x,1]
\end{cases}
\]
and analogously for $\varphi_{x,+}$. Then by \eqref{eq:rip-Gamma}, we can deduce the existence of $\sigma(\ell) > 0$ and
$\tau_3(\ell) \in ]0,\tau_2(\ell)]$ such that
\begin{multline}\label{eq:Deltaderiv}
a_z < a_{x_j} + \sigma(\ell),\;b_z > b_{x_j} - \sigma(\ell),\;a_z < a_{\Rcal x_j} + \sigma(\ell),\;b_z > b_{\Rcal x_j} - \sigma(\ell)\;
\forall j=1,\ldots,k \\ \text{ and }\tau \in [0,\tau_3(\ell)] \Longrightarrow
\frac{\partial \Delta_{z}}{\partial \tau}(\tau) \leq (b_\tau - a_\tau)(-\frac{\mu(\ell)}4) \leq -\frac{c_1\mu(\ell)}{8M_0},
\end{multline}
because $b_\tau - a_\tau \geq \frac{c_1}{2M_0}$.

Note that if there exists $i$ such that
\begin{multline*}
a_z \geq a_{x_i} + \sigma(\ell) \text{ or }b_z \leq b_{x_i} - \sigma(\ell), \text{ and }\\
a_z \geq a_{\Rcal x_i} + \sigma(\ell) \text{ or }b_z \leq b_{\Rcal x_i} - \sigma(\ell)
\end{multline*}
and $\rho_1(x_i)$ is chosen sufficiently small we have
\begin{equation}\label{eq:stimavalore}
z \in {\mathcal U}_{x_i} \Longrightarrow
(b_z-a_z)\frac{1}{2}\int_{a_z}^{b_z}g(\dot z, \dot z)\mathrm{d}s
\leq (b_{x_i}-a_{x_i})\frac{1}{2}\int_{a_{x_i}}^{b_{x_i}}g(\dot x, \dot x)\mathrm{d}s -\Delta(\ell),
\end{equation}
where $\Delta(\ell)$ is a positive constant depending only by $\ell$, as we can see arguing as in the proof of Lemma \ref{thm:differenze}.
Therefore, arguing as in the proof of Lemma \ref{thm:controllocrescita}, we see that $\Delta_{z}(\tau)$ is always under the level $c_1-\epsilon$
for any $\tau \in [0,\tau_3(\ell)]$, for a suitable choice of $\tau_3(\ell) \in ]0,\tau_2(\ell)]$.

To conclude the proof we can therefore repeat the proof of Proposition \ref{thm:prop4.17f} outside a small neighborhood of irregular variationally critical
portions of II type, using the flow \eqref{eq:rip-flow} is such a neighborhood. Using essentially the same approach of the proof of
Proposition \ref{thm:prop4.17f} we can arrive to the conclusion thanks in particular to properties \eqref{eq:rip-Gamma}, \eqref{eq:Deltaderiv}
and \eqref{eq:stimavalore}.
\end{proof}
\begin{rem}\label{rem:sez13:ossfinale}
In the proof above we use $\sigma(\ell)$ (involving lengths of intervals) instead of $E_1(\ell)$ because it may happens that
$D(x_i,a_z,b_z,a_{x_i},b_{x_i})$ is small and the differences $a_z - a_{x_i}$ or $b_{x_i} - b_z$ are positive and big, since we are closed to
curve that may be constant in such intervals. Note also that here $b_\tau - a_\tau$ may be increasing, differently from the situation of section
\ref{sec:first} where $b_\tau - a_\tau$ is always not increasing.
\end{rem}

\section{A deformation result for homotopies of type C and the First Deformation Lemma} \label{sec:firstdeflemma}

In this section we shall give statement and proof of the analogous of the classical First Deformation Lemma for the functional $\mathcal F$.
This is the key to obtain the existence result.

Before to do this we give a deformation result involving homotopies of type C used to move far from irregular
variationally critical pieces of I type.

Let $\rho_0$ be given by Proposition \ref{thm:prop4.20}.

\begin{prop}\label{thm:prop4.23} There exists $\bar r>0$ with the following property:
\smallskip

\begin{minipage}{12cm}
for all $(h,\mathcal D)\in {\mathcal H_1}$
there exists a continuous
homotopy $H_0:[0,1]\times h(1,\mathcal D)\to\mathfrak M$ of type C such that
\begin{enumerate}
\item\label{itm:quateruno} $(\mathcal D,H_0\star h)\in \mathcal H_1$;
\item\label{itm:quatertre}
$\mathcal F(\mathcal D,H_{0}\star h) \leq \mathcal F(\mathcal D,h)$;
\item\label{itm:quaterquattro} $\Vert H_0(\tau,x)-x\Vert_{0,1}\le
\tau\frac{\rho_0}2$, for all $x\in h(1,\mathcal D)$, for all $\tau\in[0,1]$;
\item\label{itm:quatercinque}
for every $x\in h(1,\mathcal D)$, and for every  $[a,b]\in \mathcal I_{x}$
it is
\[
\Vert H_0(1,x)-y\Vert_{\alpha,\beta}\ge\bar r
\]
for any $y\in\mathfrak M$, and for any $[\alpha,\beta]\subset[a,b]$ topologically not essential interval for $y$.
\end{enumerate}
\end{minipage}
\end{prop}

\begin{proof}
The proof is divided into two parts: first part requires the construction of the vector field (with the help of Proposition \ref{thm:prop4.20}) generating
the flow that we shall study in the second part of the proof.

{\bf Part A: Construction of the vector field.}

Set
\begin{multline*}
X = \{x \in h(1,\Dcal): \exists [\alpha,\beta] \subset [a,b] \in \Ical_x, \bar\delta\text{--close to }\partial \Omega\\
\text{and such that } \mathfrak p^x_{\alpha,\beta} \in [-2\sigma_1,-\frac{\sigma_1}3], \; \mathfrak b^x_{\alpha,\beta} \geq 1+ \bar\gamma \},
\end{multline*}
where $\bar\delta$ is given by \eqref{eq:4.37} and $\sigma_1$ by Remark \ref{rem:sigma1}.
Note that $X$ is a compact set. For any $x \in X$ denote by $[\alpha_x^i,\beta_x^i]$, $i=1,\ldots,k(x)$ the intervals $[\alpha,\beta]$
described in the definition above.
Note that, by Lemma \ref{thm:lem2.3}, their number, $k(x)$ is bounded by a constant independent of $x$, $\Dcal$ and $h$. Moreover due to their minimality
property, for any $i \neq j$, $[\alpha_x^i,\beta_x^i] \cap [\alpha_x^j,\beta_x^j]$ has at most one element. Denote by $[\tilde\alpha_x,\tilde\beta_x]$ the summary
interval for $x$ (cf. Definition \ref{def:summary-interval}) including the intervals $[\alpha_x^i,\beta_x^i]$.

\medskip

Let $\rho_0,\theta_0,\varepsilon_0$ and $\mu_0$ be as in Proposition \ref{thm:prop4.20}. Thanks to it, for any $x \in X$ there exist
$\rho(x) \in ]0,\frac{\rho_0}2]$ and $V_x \in H_0^1([0,1],\mathbb{R}^N$) (extended as null vector field outside $[\tilde\alpha_x,\tilde\beta_x]$) such that:
\begin{itemize}
\item $V_{\Rcal x}(s) = V_{x}(1-s) = \Rcal V_{x}(s)$ for all $s \in [0,1]$,
\item $\Vert V_x \Vert_{0,1} = 1$,
\item $V_x(s) = 0$, for any $s$ such that $\phi(x(s) \leq -\bar\delta + \varepsilon_0$,
\end{itemize}
and for any $z$ satisfying $\Vert z-x \Vert_{0,1} < \rho(x)$  it is:
\begin{itemize}
\item $\phi(z(s)) \leq -\frac{\sigma_1}4$ for all $s \in [\tilde\alpha_x,\tilde\beta_x]$,
\item $V_x(s) = 0$, for any $s$ such that $\phi(z(s) \leq -\bar\delta + \varepsilon_0$,
\item $g(\nabla\phi(z(s)),V_x(s)) \leq-\theta_0$ for all $s \in [\tilde\alpha_x,\tilde\beta_x]$ such that  $\phi(z(s)) \in [-2\sigma_0,0]$,
\item $\int_{\tilde\alpha_x^i}^{\tilde\beta_x^i}g(\dot z,\frac{\mathrm D}{\mathrm ds}V_x)\,\mathrm ds \leq -\mu_0$.
\end{itemize}

Now set
\[
\Ucal _x = \{z \in \mathfrak M : \Vert z-x\Vert_{0,1} < \rho(x)\}
\]
and consider the open covering $\{\Ucal_x,\Ucal_{\Rcal x}\}_{x\in X}$ of $X$. Since $X$ is compact there exists a finite covering
$\{\Ucal_{x_i},\Ucal_{\Rcal x_i}\}_{i=1,\ldots,k}$ of $X$. Set
\[
\Ucal _X = \bigcup_{i=1,\ldots,k}(\Ucal_{x_i} \cup \Ucal_{\Rcal x_i}).
\]

Define, for any $x \in \overline{\Ucal_x}$,
\[
\varrho_{x_i} (z) =  \dist_{0,1}(z,\overline{\Ucal_X}\setminus\Ucal_{x_i})
\]
where $\dist_{0,1}$ is the distance induced by the norm $\Vert \cdot \Vert_{0,1}$. Since $\Rcal X = X$,
$\Rcal \Ucal_{x_i} =\Ucal_{{\Rcal x}_i}$ and $\Rcal \circ \Rcal$ is the identity map, it is
\begin{equation*}
\varrho_{x_i}(\Rcal z) = \varrho_{\Rcal x_i}(z).
\end{equation*}
Finally set
\[
\beta_{x_i}(z) =
\frac{\varrho_{x_i}(z)}{\sum_{j=1,\ldots,k}(\varrho_{x_j}(z)+\varrho_{\Rcal x_j}(z))},
\]
which satisfies
\begin{equation*}
\beta_{x_i}(\Rcal z) = \beta_{\Rcal x_i}(z)
\end{equation*}
and
\[
\sum_{j=1,\ldots,k}(\beta_{x_j}(z)+\beta_{\Rcal x_j}(z)) = 1.
\]

The vector field that we shall use in this proof is defined as
\begin{equation}\label{eq:vectorfield1}
W(z)(s) = \sum_{j=1,\ldots,k}(\beta_{x_j}(z)V_{x_j}(s)+\beta_{\Rcal x_j}(z)V_{\Rcal x_j}(s)),
\end{equation}
which is well defined on all $\mathfrak M$.

Note that by the definition of $V_{x}$
\begin{equation}\label{eq:normaWx}
\Vert W(z)\Vert_{0,1} \leq 1.
\end{equation}

Moreover, by  the $\Rcal$--equivariant choice of $V_{x_i}$ and the $\Rcal$--invariant property of $\beta_{x_i}$ it is
\begin{equation}\label{eq:invarianzadiW1}
\Rcal W(z) = W(\Rcal z).
\end{equation}

The existence of the homotopy $H_0$ will be proved studying the flow $\eta(\tau,z)$ given by
the solution of the initial value problem
\begin{equation}\label{eq:flussoC}
\left\{\begin{array}{l}\frac{\mathrm d}{\mathrm d\tau}\eta (s)=
W(z)(s) \\[0.5cm]
\eta(0)=z \in \overline{\Ucal_X}, \end{array} \right.
\end{equation}
whose solution (for any $\tau \geq 0$) is given by
\[
\eta(\tau,z)(s) = z(s) + \tau W(z)(s).
\]

\medskip

{\bf Part B: Properties of the flow $\eta$.}

We shall move by $\eta$ only the open set $\Ucal_X$. More precisely let $\Vcal_X$ an $\Rcal$--equivariant open neighborhood of $X$ such that
\[
X \subset \Vcal_X \subset \overline{\Vcal_X} \subset \Ucal_X,
\]
and consider a $\Rcal$--invariant continuous map $\chi :\mathfrak M \rightarrow [0,1]$ such that
\[
\chi(z) = 1 \text{ if } z \in \overline{\Vcal_X}, \; \chi(z)=0 \text{ if }z \in \mathfrak M \setminus \Ucal_X.
\]
We shall consider
\[
H(\tau,z)=\eta(\chi(z)\tau,z).
\]

Arguing as in proving Step 1 of Proposition \ref{thm:prop4.17f} we see that $H$ is $\Rcal$--equivariant. Moreover choosing $\rho(x)$ sufficiently small $\Ucal_X$
is far from the constat curves, so $H$ does not move them.

Since $\rho(x_i) \leq \frac{\rho_0}2$, arguing as in step 5 of Proposition \ref{thm:prop4.17f}, se see that
\begin{multline*}
\text{ if } \tau_* \text{is the first instant such that }\Vert \eta(\tau,z)-x_i\Vert_{0,1} = \rho_0 \\
\text{ for some $i$ and }z \in \Ucal_{x_i},\text{ then }\tau_* \geq \frac{\rho_0}2.
\end{multline*}

Moreover, any $V_{x_i}$ satisfies
\[
g(\nabla\phi(z(s)),V_{x_i}(s)) \leq-\theta_0 \text{ for all }s \in [\tilde\alpha_{x},\tilde\beta_{x}] \text{ such that } \phi(z(s)) \in [-2\sigma_0,0],
\]
while $V_{x_i}=0$ outside the intervals $[\tilde\alpha_{x},\tilde\beta_{x}]$. Then,
thanks to condition
\eqref{eq:numero8} satisfied by $h$, if we choose $\rho(x) $ sufficiently small we see that $H$ is of type C in $[0,\frac{\rho_0}2]$
(and that any $[a_z,b_z] \in \Ical_z$ does not change along the flow).

Note also that, due to the properties of the $V_{x_i}$'s we have
\[
\frac12\int_{a_z}^{b_z}g(\dot z_{\tau},\dot z_{\tau})\,\mathrm ds \leq \frac12\int_{a_z}^{b_z}g(\dot z,\dot z)\,\mathrm ds,
\]
where $z_{\tau} = \eta(\tau,z)$, for any $\tau \in [0,\frac{\rho_0}2]$, for any $z\in\mathfrak M$.

Then, setting
\[
H_0(\tau,z) = H(\tau\frac{\rho_0}2,z),
\]
we see that  $\Fcal(\Dcal,h\star H_0) \leq \Fcal(\Dcal,h)$, while by \eqref{eq:normaWx}
\begin{equation}\label{eq:controllospeed}
\Vert H_{0}(\tau,x) - x \Vert_{0,1} \leq \tau\frac{\rho_0}2 \text{ for any }x \in h(1,\Dcal).
\end{equation}
To conclude also that $(\Dcal,h\star H_0)\in\Hcal_1$, it is to observe that if the radius of any $\mathcal U_{x_i}$ is sufficiently small
the endpoints of the restriction of the curve $z$ to a topologically not essential
interval are \emph{not} moved by the flow, which implies that the bending
constant is not affected by the flow. Moreover the flow do not creates news topologically non essential intervals since nothing moves nearby the level
$-\bar\delta$.
Then, along the flow, intervals which are far to be topologically not essential still remain far to be topologically not essential.

Finally by the choice of any $V_{x_i}$ we see also that, starting from $z \in  \overline{\Vcal_X}$, the
map $\phi$ decreases along the flow $\eta$ with derivative $\leq -\theta_0$
until the maximal proximity to $\partial\Omega$
gets to the value $-2\sigma_1$ and we have,
for any $\tau \in [0,\frac{\rho_0}2]$ and $z \in \overline{\Vcal_X}$ such that $\phi(\eta(\tau,z)) \geq -2\sigma_1$:
\[
\phi(\eta(\tau,z)) \leq \phi(\eta(0,z)) -\tau\theta_0 \leq -\frac{\sigma_1}4 - \tau\theta_0,
\]
while $\sigma_1$ has been chosen so that $\sigma_1 \leq \frac27\rho_0\theta_0$ (cf. Remark \ref{rem:sigma1}).
Then $\phi(\eta(\cdot,z))$ reaches surely the level $-2\sigma_1$ in a time $\tau \in [0,\frac{\rho_0}2]$ and
 we can finally deduce the existence of $\bar r >0$ satisfying \eqref{itm:quatercinque}.
\end{proof}

We are now ready for the following:

\begin{prop}[First Deformation Lemma]
\label{thm:firstdeflemma} Let $c \geq \frac12(\frac{3\delta_0}{4K_{0}})^2$ be a geometrically  regular
value (cf. Definition \ref{thm:defgeomcrit}). Then it is a topologically regular value of $\mathcal F$, namely
there exists $\varepsilon=\varepsilon(c)>0$ having the following property:
for all  $(\mathcal D,h)\in {\mathcal H}_1$ with
\[
\mathcal F(\mathcal D,h)\leq c+\varepsilon
\]
there exists a
continuous map $\eta\in C^0\big([0,1]\times h(1,\Dcal),\mathfrak M\big)$ such that $(\Dcal,\eta\star h)\in{\Hcal}_1$ and
\[
\mathcal F(\Dcal,\eta\star h)\leq c-\varepsilon.
\]
\end{prop}

\begin{proof}
Take $\bar\varepsilon>0$ such that there are no OGC's in $\overline\Omega$
with energy value in
$[c-\bar\varepsilon,c+\bar\varepsilon]$. A simple contradiction
argument shows the existence of $r_c>0$ such that the conditions
$x\in\mathfrak M$, $[a,b]\in{\mathcal I}_x$,
 and
$\frac{(b-a)}2\int_a^bg(\dot x,\dot x)\,\mathrm ds\in[c-\bar\varepsilon,c+\bar\varepsilon]$ imply $\Vert
x-y\Vert_{a,b}\ge r_c$ for all $y:[a,b]\to\overline\Omega$
affinely parameterized OGC in $\overline\Omega$.

Using the
homotopy $H_0$ of Proposition~\ref{thm:prop4.23} we can move
$h(1,\mathcal D)$ in such a way that
\[
\mathcal F(\mathcal D,H_0\star h)\le\mathcal F(\mathcal D,h)
\]
and it becomes far from the set of
paths having topologically not essential intervals.

Therefore, choosing $\varepsilon>0$ sufficiently small and applying Proposition~\ref{thm:prop4.19f},
we obtain the existence of a homotopy $H_\varepsilon$ such that
$(\mathcal D,H_\varepsilon\star H_0\star h) \in \mathcal H_1$ and
\[
\mathcal F(\mathcal D,H_\varepsilon\star H_0\star h) \leq
\mathcal F(\mathcal D,H_0\star h) - 2\varepsilon \leq \mathcal F(\mathcal D,h)-2\varepsilon
\]
which implies
\[
\mathcal F(\mathcal D,H_\varepsilon\star H_0\star h) \leq c - \varepsilon.
\]
Then $H_\varepsilon\star H_0\star h$ is the required
homotopy.

\end{proof}

\section{Proof of the main Theorem}

In view of multiplicity results
the topological invariant that will be employed in our proof is the relative category of a
topological pair $(X,Y)$ as defined in \cite[Definition 3.1]{FW}. For other definition of the relative category
and other relative cohomological indexes see e.g. \cite{FH} and references therein.

\begin{defin}\label{def:10.1}
Let $X$ be a topological space and $Y$ a closed subset of $X$. A closed subset $F$ of $X$ has relative category
equal to $k\in\N$ ($\cat_{X,Y}(F)=k$) if $k$ is the minimal positive integer such that $F\subset\cup_{i=0}^k
A_i$, any $A_i$ is open, $F\cap Y\subset A_0$, and for any $i=0,\ldots,k$ there exists $h_i\in C^0([0,1]\times
A_i,X)$ with the following properties:
\begin{enumerate}
\item\label{itm:def10.1-1} $h_1(0,x)=x,\,\forall x\in A_i,\,\forall i=0,\ldots,k$; \item\label{itm:def10.1-2}
for any $i=1,\ldots,k$:
\begin{enumerate}
\item\label{itm:def10.1-2a} there exists $x_i\in X\setminus Y$ such that $h_i(1,A_i)=\{x_i\}$;
\item\label{itm:def10.1-2b} $h_i:([0,1]\times A_i)\subset X\setminus Y$;
\end{enumerate}
\item\label{itm:def10.1-3} if $i=0$:
\begin{enumerate}
\item\label{itm:def10.1-3a} $h_0(1,A_0)\subset Y$; \item\label{itm:def10.1-3b} $h_0(\tau,A_0\cap Y)\subset
Y,\,\forall\tau\in[0,1]$.
\end{enumerate}
\end{enumerate}
\end{defin}

For any $X\in\Mcal$, we denote by $\widetilde X$ the quotient space with respect to the
equivalence relation induced by $\mathcal R$.
We shall use, as topological invariant, the number
$\cat_{\widetilde{\mathfrak C},\widetilde{\mathfrak C}_0}(\widetilde{\mathfrak C})$.
In Appendix~\ref{sec:app1} we will show that
\begin{equation}\label{eq:10.1}
\cat_{\widetilde{\mathfrak C},\widetilde{\mathfrak C_0}}(\widetilde{\mathfrak C})\ge N,
\end{equation}
using the topological properties of the $(N-1)$--dimensional real projective space.

Denote by $\Dgot$ the class of closed $\Rcal$--invariant subset of $\mathfrak C$. Define, for
any $i=1,\ldots,N$,
\begin{equation}\label{eq:10.2}
\Gamma_i=\{\Dcal\in\Dgot\,:\,\cat_{\widetilde{\mathfrak C},\widetilde{\mathfrak C}_0}(\widetilde\Dcal)\ge i\}.
\end{equation}
Set
\begin{equation}\label{eq:10.3}
c_i=\inf_{\stackrel{\Dcal\in\Gamma_i,}{(\mathcal D,h)\in\Hcal_1}}\Fcal(\mathcal D,h).
\end{equation}

Any number $c_i$ is called topologically critical level or $\Fcal$.
\begin{rem}\label{rem:10.2}
If $\mathrm I_{\mathfrak C}:[0,1]\times\mathfrak C$ denotes the map $\mathrm I_{\mathfrak C}(\tau,x)=x$ for
all $\tau$ and all $x$, the the pair $(\mathfrak C,\mathrm I_{\mathfrak C})\in \Hcal_1$.
Since $\widetilde{\mathfrak C}\in\Gamma_i$ for any $i$ (see \eqref{eq:10.1}),
we get:
\[
c_i\le\Fcal(\mathfrak C,\mathrm I_{\mathfrak C}) < M_0.
\]
Moreover $\Fcal\ge 0$, therefore $0\le c_i\le M_0$ for any $i$ (recall also the definition of $\mathcal F$ and $M_0$).
\end{rem}

We have the following lemmas involving the real numbers $c_i$.

\begin{lem}\label{thm:lem10.3} The following statements hold:
\begin{enumerate}
\item\label{itm:lem10.3-1} $c_1\ge \frac{1}{2}\left(\tfrac{3\delta_0}{4K_0}\right)^2$;
\item\label{itm:lem10.3-2} $c_1\le c_2\le\cdots\le c_N$.
\end{enumerate}
\end{lem}

\begin{lem}\label{thm:lem10.4} For all $i=1,\ldots,N$,  $c_i$ is a geometrically critical value.
\end{lem}

\begin{proof}[Proof of Lemma \ref{thm:lem10.3}]
Let us prove \eqref{itm:lem10.3-1}. Assume by contradiction $c_1<\frac12 \left(\tfrac{3\delta_0}{4K_0}\right)^2$,
and take $\varepsilon>0$ such that $c_1+\varepsilon<\frac12 \left(\tfrac{3\delta_0}{4K_0}\right)^2$. By
\eqref{eq:10.2}--\eqref{eq:10.3} there exists $\Dcal_\varepsilon\in\Gamma_1$, and $(\mathcal D_\varepsilon,h_\varepsilon)\in\Hcal_1$
such that
\[
\Fcal(\mathcal D_\varepsilon,h_\varepsilon)\le c_1+\varepsilon
<\frac12\left(\tfrac{3\delta_0}{4K_0}\right)^2.
\]
Then there exists an $\Rcal$--equivariant homotopy $h_0$ sending the curve $h_\varepsilon(1,\mathcal D_\varepsilon)$ first on $\partial \Omega$ (sending constant curves in
constant curves) and then on the constant curves of
$\partial \Omega$, moving the extreme points along the curves themselves. So $(h_0 \star h_\varepsilon)(1,\mathcal D_\varepsilon)$ consists of constant curves
in $\partial \Omega$
(and $h_0\star h_\varepsilon$ sends the constant curves of  $\mathcal D_{\varepsilon}$ in constant curves). Then there exist
a homotopy $K_\varepsilon:[0,1]\times\widetilde\Dcal_\varepsilon\to\widetilde{\mathfrak C}$ such that
$K_\varepsilon(0,\cdot)$ is the identity, $K_\varepsilon(1,\widetilde\Dcal_\varepsilon)\subset\widetilde{\mathfrak C}_0$
and
\[
K_\varepsilon(\tau,\widetilde\Dcal_\varepsilon\cap\widetilde{\mathfrak C}_0)\subset\widetilde{\mathfrak C}_0,\,\forall\tau\in[0,1].
\]
Then $\cat_{\widetilde{\mathfrak C},\widetilde{\mathfrak C}_0}(\widetilde\Dcal_\varepsilon)=0$, in contradiction with the
definition of $\Gamma_1$.

To prove \eqref{itm:lem10.3-2}, fix $i\in\{1,\ldots,N-1\}$ and consider $c_i$ and $c_{i+1}$. By \eqref{eq:10.3}
for any $\varepsilon>0$ there exists $\Dcal\in\Gamma_{i+1}$ and $(\mathcal D,h)\in\Hcal_1$ such that
\[
\Fcal(\mathcal D,h)\le c_{i+1}+\varepsilon.
\]
Since $\Gamma_{i+1}\subset\Gamma_i$ by definition of $c_i$ we deduce $c_i\le c_{i+1}+\varepsilon$, and
\eqref{itm:lem10.3-2} is proved, since $\varepsilon$ is arbitrary.
\end{proof}

\begin{proof}[Proof of Lemma \ref{thm:lem10.4}]
Assume by contradiction that $c_i$ is not a  geometrically critical value for some $i$.
Take $\varepsilon=\varepsilon(c_i)$ as in Proposition~\ref{thm:firstdeflemma}, and $(\Dcal_\varepsilon,h)\in\Hcal_1$ such that
\[
\Fcal(\Dcal_\varepsilon,h)\le c_i+\varepsilon.
\]
Now let $\eta$ as in Proposition~\ref{thm:firstdeflemma} and take $h_\varepsilon=\eta\star h$. Since
\[
\Fcal(\Dcal_\varepsilon,h_\varepsilon)\le c_i-\varepsilon,
\]
we get a contradiction with \eqref{eq:10.3} because $(\Dcal_\varepsilon,h_\varepsilon)\in\Hcal_1$.
\end{proof}

\begin{proof}[Proof of Theorem]
It follows from part \eqref{itm:lem10.3-1} of Lemma \ref{thm:lem10.3} and Lemma \ref{thm:lem10.4}.
\end{proof}

\begin{rem}
Note that we have obtained only an existence result; a multiplicity result would follow
from our construction if one provided additional arguments that show that the
$c_i$'s are distinct.
This could be done if one were able to determine open contractible sets of our path space,
containing curves having portions close to OGC's. The details of this construction are rather
involved, and they are the object of further studies.
\end{rem}

\appendix

\section{Computation of the relative category}\label{sec:app1}

Let $2\le N\in\N,\,\mathds D^N$ the unit disk in $\R^N$ with the Euclidean norm and $\mathds{S}^{N-1}=\partial\mathds
D^N$. Let $\Rcal$ be the reversing map  $\Rcal x(s)=x(1-s)$, defined in the set of curves
$x:[0,1]\to\mathds D^N$. With a slight abuse of notation we will denote by
$\Rcal$ also the equivalence relation induced. Fix $\sigma\in\left]0,1\right[$ and set
\begin{align}
&\mathcal C=\{\gamma:[0,1]\to \mathds D^N,\:\,\gamma(t)=(1-t) x_1+t x_2,\,x_1,x_2\in \mathds{S}^{N-1}\}/\Rcal,\label{eq:A.1}\\
&\mathcal C_\sigma=\{[\gamma]\in\mathcal C\,:\,|\gamma(1)-\gamma(0)|\le\sigma\},\label{eq:A.2},\\
\intertext{and} &\mathcal C_0=\{[\gamma]\in\mathcal C\,:\,\gamma(t)=x\in \mathds{S}^{N-1}\,\forall
t\in[0,1]\}.\label{eq:A.3}
\end{align}

Note that $\mathcal C$ is homeomorphic to $(\mathds{S}^{N-1}\times \mathds{S}^{N-1})/\mathds Z_2$ where the action of $\mathds
Z_2$ on $\mathds{S}^{N-1}\times \mathds{S}^{N-1}$ is given by  $\mathcal S(A,B)=(B,A)$.

\begin{rem}\label{rem:A.1}
Since there exists an homeomorphism $\Phi:\Ccal\to\mathfrak C$ such that
$\Phi(\Ccal_0)=\mathfrak C_0$, then
\begin{equation}\label{eq:A.4}
\cat_{\Ccal,\Ccal_0}(\Ccal)=\cat_{\widetilde{\mathfrak C},\widetilde{\mathfrak C_0}}
(\widetilde{\mathfrak C)},
\end{equation}
therefore, to prove \eqref{eq:10.1}, we need to show   that
$\cat_{\mathcal C,\mathcal C_0}(\mathcal C)\ge N$.
\end{rem}

\begin{rem}\label{rem:A.2}
Note that property \eqref{itm:def10.1-2b} of Definition \ref{def:10.1} is essential in our case,
to guarantee that the relative category is at least $N$. Indeed in \cite{G} it is proved that the
Ljusternik-Schnirelman category of $\Ccal$ is at most equal to three, and if we did not require
\eqref{itm:def10.1-2b} in the definition of relative category we would have
$\cat_{\Ccal,\Ccal_0}(\Ccal)\le\cat(\Ccal)$.
\end{rem}

Now we are going to prove the following result.

\begin{prop}\label{thm:propA.3} $\cat_{\mathfrak C,\mathfrak C_0}(\mathfrak C)\ge N$.
\end{prop}

The proof will be performed using singular cohomology theory and the cup product (see e.g. \cite{Spa})
with $\mathds Z_2$ coefficients.
For any topological pair
$(X,Y)$ it will be denoted by $H^q(X,Y)$ at any dimension $q\ge 0$.

The notion of relative cuplength, here recalled, will be also used.

\begin{defin}\label{def:A.4}
The number $\cl(X,Y)$ is the largest positive integer $k$ for which there exists $\alpha_0\in H^{q_0}(X,Y)$
($q_0\ge 0$) and $\alpha_i\in H^{q_i}(X)$, $i=1,\ldots,k$ such that
\[
q_i\ge 1\,\forall i=1,\ldots,k,
\]
and
\[
\alpha_0\cup\alpha_1\cup\ldots\cup\alpha_k\ne 0 \text{\ in\ } H^{q_0+q_1+\ldots+q_k}(X,Y),
\]
where $\cup$ denotes the cup product.
\end{defin}

Recall that, if $Y\ne\emptyset$, the absolute cuplenght of $X$ is the largest positive integer $k$ for which
there exists $\alpha_i\in H^{q_i}(X),\,i=1,\ldots,k$ such that
\[
q_i\ge 1,\,\forall i=1,\ldots,k,
\]
and
\[
\alpha_1\cup\ldots\alpha_k\ne 0\text{\ in\ }H^{q_1+\ldots+q_k}(X).
\]

\begin{proof}[Proof of Proposition \ref{thm:propA.3}]
The proof is divided into four steps.

\textbf{Step 1.} $\cat_{\mathcal C,\mathcal C_0}(\mathcal C)\ge\cl(\mathcal C\setminus\mathcal C_0,
\mathcal C_\sigma\setminus\mathcal C_0)+1$.

Assume that $\cat_{\mathcal C,\mathcal C_0}(\mathcal C)=k<+\infty$. Since $\mathcal C_0$ is not a retract of
$\mathcal C$, it is $k\ge 1$. Take $A_0,A_1,\ldots,A_k$ open subsets as in definition \ref{def:10.1}, and let
\[
\imath_r:A_r\to\mathcal C\setminus\mathcal C_0,\qquad\jmath_r:(\mathcal C\setminus\mathcal C_0,\emptyset)\to(\mathcal C\setminus\mathcal C_0,A_r)
\]
be inclusion maps. By property \eqref{itm:def10.1-2} of Definition \ref{def:10.1}, $\imath_r^*:H^q(\mathcal C\setminus\mathcal C_0)\to H^q(A_r)$ is the zero constant map for any $q\ge 1$ and any $r\ge 1$. Then, since
the sequence
\[
\ldots\xrightarrow{} H^{q_r}(\mathcal C\setminus\mathcal C_0,A_r)\xrightarrow{\jmath_r^*}H^{q_r}(\mathcal C\setminus\mathcal C_0)\xrightarrow{\imath_r^*}H^{q_r}(A_r)\xrightarrow{}\ldots
\]
is exact, then $\jmath_r^*$ is surjective if $q_r\ge 1$.
Then for any $\alpha_r\in H^{q_r}(\mathcal C\setminus\mathcal C_0)$, if $q_r\ge 1$,
there exists $\beta_r\in H^{q_r}(\mathcal C\setminus\mathcal C_0,A_r)$ such that
$\jmath_r^*(\beta_r)=\alpha_r$.

Since $A_0 \supset \mathcal C_{0}$, $A_0$ is open and $\mathcal C_0$ is closed, there exists $\sigma\in\left]0,1\right[$ such that
$A_0\supset \mathcal C_{\sigma}$.
Moreover by property \eqref{itm:def10.1-3b} of Definition \ref{def:10.1}, $\sigma$ can be chosen sufficiently
small so that, up to consider a projection on $\mathcal C_0$,
a homotopy $\hat h_0$ can be built such that $\hat h_0(\tau,\mathcal C_\sigma)\subset \mathcal C_\sigma,\,\forall
\tau\in[0,1]$, (while, obviously, $\hat
h_0(1,A_0)\subset\mathcal C_\sigma$).

Now consider the inclusion maps
\[
\jmath_0:(\mathcal C,\mathcal C_\sigma)\to(\mathcal C,A_0),\quad\imath_0:(A_0,\mathcal C_\sigma)\to(\mathcal C,\mathcal C_\sigma),
\]
and the (exact) sequence
\[
\ldots\xrightarrow{}H^{q_0}(\mathcal C,A_0)\xrightarrow{\jmath_0^*}H^{q_0}(\mathcal C,\mathcal C_\sigma)\xrightarrow{\imath_0^*}H^{q_0}(A_0,\mathcal C_\sigma)\xrightarrow{}\ldots
\]
Since $\imath_0^*:H^{q_0}(\mathcal C,\mathcal C_\sigma)\to H^{q_0}(A_0,\mathcal C_\sigma)$ is the constant
zero map then $\jmath_0^*$ is surjective and, for any $\alpha_0\in H^{q_0}(\mathcal C,\mathcal C_\sigma)$
there exists $\beta_0\in H^{q_0}(\mathcal C,A_0)$ such that $\jmath_0^*(\beta_0)=\alpha_0$. Now $\mathcal C_0$
is a strong deformation retract of $\mathcal C_\sigma$ since $\sigma<1$,
and by excision property (recalling
that $\mathcal C_\sigma\subset A_0$) we have that for any $\hat\alpha_0\in H^{q_0}(\mathcal C\setminus\mathcal C_0,\mathcal C_\sigma\setminus\mathcal C_0)$ there exists $\hat\beta_0\in
H^{q_0}(\mathcal C\setminus\mathcal C_0,A_0\setminus\mathcal C_0)$ such that
\[
\jmath_0^*(\hat\beta_0)=\hat\alpha_0,
\]
where $\jmath_0:(\mathcal C\setminus\mathcal C_0,\mathcal C_\sigma\setminus\mathcal C_0)\to(\mathcal C\setminus\mathcal C_0,A_0\setminus\mathcal C_0)$ is the inclusion map.

Finally we have, since $A_i$ are open,
\begin{multline*}
\hat\beta_0\cup\beta_1\cup\ldots\cup\beta_k\in H^{q_0+q_1+\ldots+q_k}(\mathcal C\setminus\mathcal C_0,(A_0\setminus\mathcal C_0)\cup A_1\cup\ldots\cup A_k)=\\ H^{q_o+q_1+\ldots+q_k}
(\mathcal C\setminus\mathcal C_0,\mathcal C\setminus\mathcal C_0)=0.
\end{multline*}
Moreover, by the naturality of the cup product (see \cite{Spa}) we have (denoting by $\jmath$ the inclusion map)
\[
\hat\alpha_0\cup\alpha_1\cup\ldots\cup\alpha_k=\jmath^*(\hat\beta_0\cup\beta_1\cup\ldots\cup\beta_k)=\jmath^*(0)=0,
\]
proving that $\cl(\mathcal C\setminus\mathcal C_0,\mathcal C_{\sigma}\setminus\mathcal C_0)<k$.

\textbf{Step 2.} $\cl(\mathcal C\setminus\mathcal C_0,\mathcal C_\sigma\setminus\mathcal C_0)=\cl(X_\sigma,Y_\sigma)$, where
\[
X_\sigma=\{[\gamma]\in\mathcal C\,:\,|\gamma(1)-\gamma(0)|\ge\sigma\},\quad Y_\sigma=\{[\gamma]\in\mathcal C\,:\,|\gamma(1)-\gamma(0)|=\sigma\}.
\]

This is straightforward, once one gets the existence of $H\in C^0([0,1]\times\mathcal C\setminus\mathcal C_0,\mathcal C
\setminus\mathcal C_0)$ such that $H(0,x)=x\,\forall x\in \mathcal C\setminus {\mathcal C_{0}}$,
$H(\tau,x)=x,\,\forall x\in
X_\sigma,\,\forall\tau\in[0,1]$, and
\[
H(1,\mathcal C\setminus\mathcal C_0)=X_\sigma,\quad H(1,\mathcal C_\sigma\setminus\mathcal C_0)=Y_\sigma.
\]

\textbf{Step 3.} $\cl(X_\sigma,Y_\sigma)=\cl(E,\partial E)$, where $E$ is the closed unit disk bundle over the
manifold $\mathds P^{N-1}$ and $\partial E$ its boundary.

\noindent
This is an immediate consequence of the fact that $(X_\sigma,Y_\sigma)$ is
homeomorphic to $(E,\partial E)$.

\textbf{Step 4.} $\cl(E,\partial E)\ge N-1.$

To prove this, let us observe that
\[
H^q(\mathds D^{N-1},\partial\mathds D^{N-1})=\begin{cases}
0,&\text{if\ }q\ne N-1,\\
\mathds Z_2,&\text{if\ }q= N-1.
\end{cases}
\]
Denoting by $\pi$ the canonical projection of $E$ in $\mathds P^{N-1}$, thanks to the contractibility of
$\mathds D^{N-1}$ we see that
\begin{equation}\label{eq:A.5}
\pi^*:H^q(E)\to H^q(\mathds P^{N-1})\text{\ is an isomorphism\ }\forall q\ge 0.
\end{equation}
Since we are considering $\mathds Z_2$--coefficients there are not problems with orientation, and by
\cite[Corollary 5.7.18]{Spa} the fiber bundle pair $((E,\partial E),\mathds P^{N-1},(\mathds
D^{N-1},\partial\mathds D^{N-1}),\pi)$ has a unique orientation cohomology class $\zeta^{N-1}$ with dimension
$N-1$. Then, by Thom isomorphism Theorem \cite[Theorem 5.7.10]{Spa} the homomorphism
\[
\Phi:H^1(\mathds P^{N-1})\to H^{q+N-1}(E,\partial E)
\]
given by $\Phi(z)=\pi^*(z)\cup\zeta^{N-1}$ is an isomorphism for any $q\ge 0$. From this fact and from
\eqref{eq:A.5} we deduce that
\[
\cl(E,\partial E)\ge\cl(E).
\]
Finally, using \eqref{eq:A.5} and standard results in literature (see e.g.\ \cite{Spa}),
$$\cl(E)=\cl(\mathds P^{N-1})=N-1.$$
\end{proof}


\end{document}